\documentclass[10pt,reqno]{amsart} 

\usepackage{float} 

\usepackage{epsfig,amssymb,amsmath,version}
\usepackage{amssymb,version,graphicx,fancybox,mathrsfs,labelfig}

\usepackage{url,hyperref}
\usepackage{subfigure}
\usepackage{color}
\usepackage{stmaryrd}
\usepackage{multirow}
\usepackage{booktabs,siunitx}
\usepackage{booktabs}
\usepackage{multicol,epstopdf}
\usepackage{xypic}
\usepackage[]{graphicx}
\usepackage[all]{xy}
\usepackage{tikz,ifthen}
\usetikzlibrary{positioning, math, decorations.markings, arrows.meta}
\usetikzlibrary{calc,shapes,cd}
\usetikzlibrary{knots} 
\usepgfmodule{decorations}
\allowdisplaybreaks
\tikzset{knotarrow/.pic={ \draw[edge, <-] (0,0) -- +(-.001,0);}}

\tikzset{edge/.style={line width=0.8}}
\tikzset{wall/.style={very thick}}

\tikzset{->-/.style n args={2}{decoration={markings, mark=at position #1 with {\arrow{#2}}}, postaction={decorate}}} 

\ExplSyntaxOn
\cs_new_eq:NN \ifstreqF \str_if_eq:nnF
\cs_new_eq:NN \ifstreqTF \str_if_eq:nnTF
\ExplSyntaxOff

\tikzset{-o-/.code 2 args={\ifstreqF{#2}{} 
{\ifstreqTF{#2}{>}
   {\pgfkeysalso{decoration={markings,mark=at position #1 with {\arrow[scale=0.8]{#2}}}
                    ,postaction={decorate}}
    }
   {\ifstreqTF{#2}{<}
       {\pgfkeysalso{decoration={markings,mark=at position #1 with {\arrow[scale=0.8]{#2}}}
                    ,postaction={decorate}}
        }
       {\pgfkeysalso{decoration={markings,
                    mark=at position #1 with
                    {\draw[black, fill={#2}] circle[radius=2pt];}}
                    ,postaction={decorate}}
        }
     }
  }}}

\allowdisplaybreaks[4]

\newtheorem{theorem}{Theorem}[section]
\newtheorem{lemma}[theorem]{Lemma}
\newtheorem{definition}[theorem]{Definition}
\newtheorem{corollary}[theorem]{Corollary}
\newtheorem{proposition}[theorem]{Proposition}
\newtheorem{rem}[theorem]{Remark}

\definecolor{ligreen}{rgb}{0.0, 0.3, 0.0}

\definecolor{darkblue}{rgb}{0.0, 0.0, 0.55}

\definecolor{anti-flashwhite}{rgb}{0.55, 0.57, 0.68}

\newcommand{\cev}[1]{\reflectbox{\ensuremath{\vec{\reflectbox{\ensuremath{#1}}}}}}

\def\cY{\mathcal Y}
\def\cS{\mathscr S}

\def\N{\mathcal{N}}

\newcommand{\beq}{\begin{equation}}
\newcommand{\eeq}{\end{equation}}

\graphicspath{{./Figures/} } 

\begin{document}
\bibliographystyle{plain}

\title{The classical limit for stated $SL_n$-skein modules}
\author{Zhihao Wang}


 \maketitle


\begin{abstract}

Let $(M,\mathcal{N})$ be a marked  3-manifold. We use $S_n(M,\mathcal{N},v)$ to denote the stated $SL_n$-skein module of $(M,\mathcal{N})$ where $v$ is a nonzero complex number.
We establish a surjective algebra homomorphism from $S_n(M,\mathcal{N},1)$ to the coordinate ring of some algebraic set, and prove its kernel consists of all nilpotents. We prove the universal representation algebra of $\pi_1(M)$ is isomorphic to $S_n(M,\mathcal{N},1)$ when $M$ is connected and $\N$ has only one component. Furthermore, we  show  $S_n(M,\mathcal{N}',1)$ is isomorphic to
$S_n(M,\mathcal{N},1)\otimes O(SL_n)$ as algebras, where $(M,\mathcal{N})$ is a connected marked 3-manifold with $\mathcal{N}\neq\emptyset$, and $\mathcal{N}'$ is obtained from $\mathcal{N}$ by adding one extra marking.
 
\end{abstract}

\tableofcontents{}

\section{Introduction}

In this paper we will work over the complex field $\mathbb{C}$ with a distinguished nonzero element $v=q^{\frac{1}{2n}}$ where $n$ is a positive integer. When we say algebra (or vector space), we just mean an algebra (or vector space) over  $\mathbb{C}$. For any other  ring or algebra $A$, we will use $A$-algebra to mean an algebra over $A$. 
 For any commutative algebra $A$, we use MaxSpec$(A)$ to denote the set of maximal ideals of $A$.

For any topological space $X$, we  use $cl(X)$ to denote the closure of $X$, and use $int(X)$ to denote the interior of $X$, and use $\sharp X$ to denote the number of components of $X$.

 A {\bf marked  3-manifold} is  a pair $(M,\mathcal{N})$, where $M$ is a smooth oriented 3-manifold, and $\N\subset\partial M$ is a one dimensional submanifold of $M$ consisting of oriented open intervals such that  there is no intersection between the closures of any two components of $\N$. Note that we allow $\N$ to be empty. The marked 3-manifold $(M,\mathcal{N})$ is said to be {\bf connected} if $M$ is connected.

L{\^e} and Sikora defined the stated $SL_n$-skein module for marked 3-manifolds \cite{le2021stated}, which is a generalization for the classical case when $n=2$ \cite{lestatedsurvery}. Let $(M,\mathcal{N})$ be a marked 3-manifold, we will use 
$S_n(M,\mathcal{N},v)$ to denote the  stated $SL_n$-skein module of $(M,\mathcal{N})$.

If $(M,\mathcal{N})$ is the disjoint union of $(M_1,\mathcal{N}_1)$ and $(M_2,\mathcal{N}_2)$, we can easily get $S_n(M,\mathcal{N},v)=
S_n(M_1,\mathcal{N}_1,v)\otimes S_n(M_2,\mathcal{N}_2,v)$. Then we can reduce general marked 3-manifolds to connected marked 3-manifolds. So sometimes for convenience, we will assume the marked 3-manifold is connected (the corresponding results can be easily generalized to general marked 3-manifolds).

For the classical $SL_2$-skein theory, Bullock stablished a surjective algebra homomorphism from $S_2(M,\emptyset,v^{2}=-1)$, which is isomorphic to $S_2(M,\emptyset,1)$ \cite{barrett1999skein}, to the coordinate ring of some algebraic set, and showed the kernel of this map consists of all nilpotents \cite{BL1997rings}. This is an important development since it showed the connection between the skein theory and the character variety. Moreover it offers a way to interpret $S_2(M,\emptyset,1)$. It is also useful to understand the  representation theory for $SL_2$-skein algebras \cite{representation1,representation2,representation3}.


We generalize Bullock's work to stated $SL_n$-skein modules for marked 3-manifolds. 
We construct a surjective algebra homomorphism from $S_n(M,\mathcal{N},1)$ to the coordinate ring of some algebraic set and calculate the kernel of this map. Costantino and L{\^e} have constructed this map when $n=2$ and $(M,\mathcal{N})$ is the thickening of an essentially bordered pb surface (please refer to subsection \ref{subb2.4} for the definition of the essentially bordered pb surface),  and proved it is an isomorphism \cite{CL2022stated}.
Motivated by Costantino and L{\^e}, we choose the algebraic set to be the homomorphism from the groupoid $\pi_1(M,\mathcal{N})$ to $SL(n,\mathbb{C})$, whose coordinate ring is denoted as $R_n(M,\mathcal{N})$.

\begin{theorem}
Let $(M,\mathcal{N})$ be a marked 3-manifold. There exists a surjective algebra homomorphism 
$$\Phi^{(M,\mathcal{N})} : S_n(M,\mathcal{N},1)\rightarrow R_n(M,\mathcal{N}).$$
\end{theorem}
 We  show $\Phi^{(M,\mathcal{N})}$ commutes with the splitting map, Theorem \ref{the_split}. Our construction is compatible with  the construction by Costantino and L{\^e}  when $n=2$ and $(M,\mathcal{N})$ is the thickening of an essentially bordered pb surface,  subsection \ref{newsec}.

In order to calculate Ker\,$\Phi^{(M,\mathcal{N})}$, we give an explicit presentation for $S_n(M,\mathcal{N},1)$. We do this in two steps. First we consider the case when $\N$ has only one component. 

\def \N {\mathcal{N}}

\begin{theorem}
Let $(M,\mathcal{N})$ be a connected marked 3-manifold with $\N$ consisting of one component. Then we have $S_n(M,\mathcal{N},1)\simeq \Gamma_n(M)$, where $\Gamma_n(M)$ is the universal presentation algebra of $\pi_1(M)$, Definition \ref{df3.4}.
\end{theorem}

\def\MN {(M,\mathcal{N})}


We use $O(SL_n)$ to denote the coordinate ring of $SL(n,\mathbb{C})$. Then 
$$O(SL_n)=\mathbb{C}[x_{i,j}\mid 1\leq i,j\leq n]/(\text{det}(X) = 1)$$
where $X$ is an $n$ by $n$ matrix such that $X_{i,j} = x_{i,j},1\leq i,j\leq n.$ 
We have $X^{-1}$ makes sense and is an $n$ be $n$ matrix in $O(SL_n)$ because det$(X) = 1$.

\begin{theorem}
Suppose $(M,\mathcal{N})$ is a connected marked 3-manifold with   $\N\neq\emptyset$, and $\mathcal{N}'$ is obtained from $\N$ by adding one extra marking, please refer to Definition \ref{dddddddd}. Then we have $S_n(M,\mathcal{N}',1)\simeq
S_n(M,\mathcal{N},1)\otimes O(SL_n)$ as algebras.
\end{theorem}

Let $(M,\mathcal{N})$ be a connected marked 3-manifold with  $\N\neq \emptyset$. Combining the above two Theorems we have 
$$S_n(M,\mathcal{N},1)\simeq \Gamma_n(M)\otimes O(SL_n)^{\otimes(\sharp \N-1)}$$
as algebras.

For a commutative algebra $A$, we use $\sqrt{0}_A$ to denote the ideal consisting of all nilpotents. We can  omit the subscript for $\sqrt{0}_A$ when there is no confusion with the algebra $A$.
There is a  projection from $\Gamma_n(M)\otimes O(SL_n)^{\otimes(\sharp \N-1)}$ to $R_n(M,\mathcal{N})$, whose kernel is $\sqrt{0}$. 
The isomorphism from $S_n(M,\mathcal{N},1)$ to $\Gamma_n(M)\otimes O(SL_n)^{\otimes(\sharp \N-1)}$ is compatible with $\Phi^{(M,\mathcal{N})}$, that is, the combination of this isomorphism and the projection from $\Gamma_n(M)\otimes O(SL_n)^{\otimes(\sharp \N-1)}$ to $R_n(M,\mathcal{N})$ is $\Phi^{(M,\mathcal{N})}$. Thus we have 
the following theorem.


\begin{theorem}
For any marked 3-manifold $(M,\mathcal{N})$, we have Ker\,$\Phi^{(M,\mathcal{N})} = \sqrt{0}$. 
\end{theorem}

\def\MN{(M,\N)}

Costantino and L{\^e} also defined the  {\bf generalized marked 3-manifold} by allowing oriented closed circle in $\N$ \cite{CL2022TQFT}. 
A generalized marked 3-manifold is said to be connected if the corresponding 3-manifold is connected.
We can obviously define the stated $SL_n$-skein module for the generalized marked 3-manifold. For any generalized marked 3-manifold $\MN$, we  define $\Gamma_n(M,\mathcal{N})$ as a quotient algebra of $\Gamma_n(M)$ (Definition \ref{df7.3}).

\begin{theorem}
Let $(M,\mathcal{N})$ be a connected generalized marked 3-manifold with  $\N\neq \emptyset$.
Then $S_n(M,\mathcal{N},1)\simeq \Gamma_n(M,\mathcal{N})\otimes O(SL_n)^{\otimes(\sharp \N - 1)}$.
\end{theorem}

{\bf Acknowledgements}:
The author would like to thank my supervisors Andrew James Kricker and  Roland van der Veen, and my colleague Jeffrey Weenink Andre for helpful discussion. The research is supported by the NTU  research scholarship.

\section{Preliminaries}

In this section, we will recall some definitions and results about the stated $SL_n$-skein modules,
 and also introduce some conventions for this paper.

\subsection{Stated $SL_n$-skein modules}
In this paper, we follow the definition in \cite{le2021stated} for stated $SL_n$-skein modules.  Here we briefly recall the definition.

\begin{definition}[\cite{le2021stated}]
 An $n$-web $l$ in a marked 3-manifold $(M,\mathcal{N})$ is a disjoint union of oriented closed paths and a directed finite  graph properly embedded into $M$. We also have the following requirements:

(1) $l$ only contains $1$-valent or $n$-valent vertices. Each $n$-valent vertex is a source or a  sink. The set of one valent vertices is denoted as $\partial l$, which are  called endpoints of $l$.

 (2) Every edge of the graph is an embedded  closed interval  in $M$. 

(3) $l$ is equipped with a transversal framing. 

(4) The set of half-edges at each $n$-valent vertex is equipped with a  cyclic order. 

(5) $\partial l$ is contained in $\N$ and the framing at these endpoints is the velocity  vector of $\N$.
\end{definition}



For any two points $a,b\in\partial l$, we say $a$ is higher than $b$ if they belong to a same component $e$ of $\N$ and the direction of $e$ is going from $b$ to $a$.

A state of an $n$-web $l$ is a map  $s: \partial l\rightarrow \{1,2,\dots,n\}$. If there is  such a map $s$ for $l$, we say $l$ is stated by $s$. For any point $a\in\partial l$, we say $a$ is stated by $s(a)$.

Recall that our ground ring is $\mathbb{C}$ with the parameter $v\in\mathbb{C}^{*}$. We set 
$q^{\frac{1}{2n}} = v$, and define the following constants:
\begin{align}
c_i&= (-q)^{n-i} q^{\frac{n-1}{2n}},\quad
t= (-1)^{n-1} q^{\frac{n^2-1}{n}}, \ t^{n/2} =  (-1)^{\frac{(n-1)n}2} q^{\frac{n^2-1}{2}} \label{e.t}\\
a &=   q^{\frac{n+1-2n^2}{4}},\quad 
d_n = (-1)^{n-1}. 
\end{align}

Note that
\begin{equation}\label{e.prodc} 
\prod_{i=1}^n c_i  = t^{n/2}= (-1)^{\frac{(n-1)n}2 } q^\frac{n^2-1}{2} 
\ \text{and}\ c_i\cdot c_{\bar i}=t, \ \text{for}\ i=1,\dots, n.
\end{equation}

We use $S_n$ to denote the permutation group on  $\{1,2,\dots,n\}$.


The stated $SL_n$-skein module of $(M,\mathcal{N})$, denoted as $S_n(M,\mathcal{N},v)$, is obtained in two steps. We first use all isotopy classes of stated 
$n$-webs in $(M,\mathcal{N})$ as a basis to generate a vector space, then  quotient the following relations.

\beq\label{w.cross}
q^{\frac{1}{n}} 
\raisebox{-.20in}{

\begin{tikzpicture}
\tikzset{->-/.style=

{decoration={markings,mark=at position #1 with

{\arrow{latex}}},postaction={decorate}}}
\filldraw[draw=white,fill=gray!20] (-0,-0.2) rectangle (1, 1.2);
\draw [line width =1pt,decoration={markings, mark=at position 0.5 with {\arrow{>}}},postaction={decorate}](0.6,0.6)--(1,1);
\draw [line width =1pt,decoration={markings, mark=at position 0.5 with {\arrow{>}}},postaction={decorate}](0.6,0.4)--(1,0);
\draw[line width =1pt] (0,0)--(0.4,0.4);
\draw[line width =1pt] (0,1)--(0.4,0.6);
\draw[line width =1pt] (0.4,0.6)--(0.6,0.4);
\end{tikzpicture}
}
- q^{-\frac {1}{n}}
\raisebox{-.20in}{
\begin{tikzpicture}
\tikzset{->-/.style=

{decoration={markings,mark=at position #1 with

{\arrow{latex}}},postaction={decorate}}}
\filldraw[draw=white,fill=gray!20] (-0,-0.2) rectangle (1, 1.2);
\draw [line width =1pt,decoration={markings, mark=at position 0.5 with {\arrow{>}}},postaction={decorate}](0.6,0.6)--(1,1);
\draw [line width =1pt,decoration={markings, mark=at position 0.5 with {\arrow{>}}},postaction={decorate}](0.6,0.4)--(1,0);
\draw[line width =1pt] (0,0)--(0.4,0.4);
\draw[line width =1pt] (0,1)--(0.4,0.6);
\draw[line width =1pt] (0.6,0.6)--(0.4,0.4);
\end{tikzpicture}
}
= (q-q^{-1})
\raisebox{-.20in}{

\begin{tikzpicture}
\tikzset{->-/.style=

{decoration={markings,mark=at position #1 with

{\arrow{latex}}},postaction={decorate}}}
\filldraw[draw=white,fill=gray!20] (-0,-0.2) rectangle (1, 1.2);
\draw [line width =1pt,decoration={markings, mark=at position 0.5 with {\arrow{>}}},postaction={decorate}](0,0.8)--(1,0.8);
\draw [line width =1pt,decoration={markings, mark=at position 0.5 with {\arrow{>}}},postaction={decorate}](0,0.2)--(1,0.2);
\end{tikzpicture}
},
\eeq 
\beq\label{w.twist}
\raisebox{-.15in}{
\begin{tikzpicture}
\tikzset{->-/.style=
{decoration={markings,mark=at position #1 with
{\arrow{latex}}},postaction={decorate}}}
\filldraw[draw=white,fill=gray!20] (-1,-0.35) rectangle (0.6, 0.65);
\draw [line width =1pt,decoration={markings, mark=at position 0.5 with {\arrow{>}}},postaction={decorate}](-1,0)--(-0.25,0);
\draw [color = black, line width =1pt](0,0)--(0.6,0);
\draw [color = black, line width =1pt] (0.166 ,0.08) arc (-37:270:0.2);
\end{tikzpicture}}
= t
\raisebox{-.15in}{
\begin{tikzpicture}
\tikzset{->-/.style=
{decoration={markings,mark=at position #1 with
{\arrow{latex}}},postaction={decorate}}}
\filldraw[draw=white,fill=gray!20] (-1,-0.5) rectangle (0.6, 0.5);
\draw [line width =1pt,decoration={markings, mark=at position 0.5 with {\arrow{>}}},postaction={decorate}](-1,0)--(-0.25,0);
\draw [color = black, line width =1pt](-0.25,0)--(0.6,0);
\end{tikzpicture}}
,
\eeq
\beq\label{w.unknot}
\raisebox{-.20in}{
\begin{tikzpicture}
\tikzset{->-/.style=
{decoration={markings,mark=at position #1 with
{\arrow{latex}}},postaction={decorate}}}
\filldraw[draw=white,fill=gray!20] (0,0) rectangle (1,1);
\draw [line width =1pt,decoration={markings, mark=at position 0.5 with {\arrow{>}}},postaction={decorate}](0.45,0.8)--(0.55,0.8);
\draw[line width =1pt] (0.5 ,0.5) circle (0.3);
\end{tikzpicture}}
= (-1)^{n-1} [n]\ 
\raisebox{-.20in}{
\begin{tikzpicture}
\tikzset{->-/.style=
{decoration={markings,mark=at position #1 with
{\arrow{latex}}},postaction={decorate}}}
\filldraw[draw=white,fill=gray!20] (0,0) rectangle (1,1);
\end{tikzpicture}}
,\ \text{where}\ [n]=\frac{q^n-q^{-n}}{q-q^{-1}},
\eeq
\beq\label{wzh.four}
\raisebox{-.30in}{
\begin{tikzpicture}
\tikzset{->-/.style=
{decoration={markings,mark=at position #1 with
{\arrow{latex}}},postaction={decorate}}}
\filldraw[draw=white,fill=gray!20] (-1,-0.7) rectangle (1.2,1.3);
\draw [line width =1pt,decoration={markings, mark=at position 0.5 with {\arrow{>}}},postaction={decorate}](-1,1)--(0,0);
\draw [line width =1pt,decoration={markings, mark=at position 0.5 with {\arrow{>}}},postaction={decorate}](-1,0)--(0,0);
\draw [line width =1pt,decoration={markings, mark=at position 0.5 with {\arrow{>}}},postaction={decorate}](-1,-0.4)--(0,0);
\draw [line width =1pt,decoration={markings, mark=at position 0.5 with {\arrow{<}}},postaction={decorate}](1.2,1)  --(0.2,0);
\draw [line width =1pt,decoration={markings, mark=at position 0.5 with {\arrow{<}}},postaction={decorate}](1.2,0)  --(0.2,0);
\draw [line width =1pt,decoration={markings, mark=at position 0.5 with {\arrow{<}}},postaction={decorate}](1.2,-0.4)--(0.2,0);
\node  at(-0.8,0.5) {$\vdots$};
\node  at(1,0.5) {$\vdots$};
\end{tikzpicture}}=(-q)^{\frac{n(n-1)}{2}}\cdot \sum_{\sigma\in S_n}
(-q^{\frac{1-n}n})^{\ell(\sigma)} \raisebox{-.30in}{
\begin{tikzpicture}
\tikzset{->-/.style=
{decoration={markings,mark=at position #1 with
{\arrow{latex}}},postaction={decorate}}}
\filldraw[draw=white,fill=gray!20] (-1,-0.7) rectangle (1.2,1.3);
\draw [line width =1pt,decoration={markings, mark=at position 0.5 with {\arrow{>}}},postaction={decorate}](-1,1)--(0,0);
\draw [line width =1pt,decoration={markings, mark=at position 0.5 with {\arrow{>}}},postaction={decorate}](-1,0)--(0,0);
\draw [line width =1pt,decoration={markings, mark=at position 0.5 with {\arrow{>}}},postaction={decorate}](-1,-0.4)--(0,0);
\draw [line width =1pt,decoration={markings, mark=at position 0.5 with {\arrow{<}}},postaction={decorate}](1.2,1)  --(0.2,0);
\draw [line width =1pt,decoration={markings, mark=at position 0.5 with {\arrow{<}}},postaction={decorate}](1.2,0)  --(0.2,0);
\draw [line width =1pt,decoration={markings, mark=at position 0.5 with {\arrow{<}}},postaction={decorate}](1.2,-0.4)--(0.2,0);
\node  at(-0.8,0.5) {$\vdots$};
\node  at(1,0.5) {$\vdots$};
\filldraw[draw=black,fill=gray!20,line width =1pt]  (0.1,0.3) ellipse (0.4 and 0.7);
\node  at(0.1,0.3){$\sigma_{+}$};
\end{tikzpicture}},
\eeq
where the ellipse enclosing $\sigma_+$  is the minimum crossing positive braid representing a permutation $\sigma\in S_n$ and $\ell(\sigma)=\mid\{(i,j)\mid 1\leq i<j\leq n, \sigma(i)>\sigma(j)\}|$ is the length of $\sigma\in S_n$.

\beq\label{wzh.five}
   \raisebox{-.30in}{
\begin{tikzpicture}
\tikzset{->-/.style=
{decoration={markings,mark=at position #1 with
{\arrow{latex}}},postaction={decorate}}}
\filldraw[draw=white,fill=gray!20] (-1,-0.7) rectangle (0.2,1.3);
\draw [line width =1pt](-1,1)--(0,0);
\draw [line width =1pt](-1,0)--(0,0);
\draw [line width =1pt](-1,-0.4)--(0,0);
\draw [line width =1.5pt](0.2,1.3)--(0.2,-0.7);
\node  at(-0.8,0.5) {$\vdots$};
\filldraw[fill=white,line width =0.8pt] (-0.5 ,0.5) circle (0.07);
\filldraw[fill=white,line width =0.8pt] (-0.5 ,0) circle (0.07);
\filldraw[fill=white,line width =0.8pt] (-0.5 ,-0.2) circle (0.07);
\end{tikzpicture}}
   = 
   a \sum_{\sigma \in S_n} (-q)^{\ell(\sigma)}\,  \raisebox{-.30in}{
\begin{tikzpicture}
\tikzset{->-/.style=
{decoration={markings,mark=at position #1 with
{\arrow{latex}}},postaction={decorate}}}
\filldraw[draw=white,fill=gray!20] (-1,-0.7) rectangle (0.2,1.3);
\draw [line width =1pt](-1,1)--(0.2,1);
\draw [line width =1pt](-1,0)--(0.2,0);
\draw [line width =1pt](-1,-0.4)--(0.2,-0.4);
\draw [line width =1.5pt,decoration={markings, mark=at position 1 with {\arrow{>}}},postaction={decorate}](0.2,1.3)--(0.2,-0.7);
\node  at(-0.8,0.5) {$\vdots$};
\filldraw[fill=white,line width =0.8pt] (-0.5 ,1) circle (0.07);
\filldraw[fill=white,line width =0.8pt] (-0.5 ,0) circle (0.07);
\filldraw[fill=white,line width =0.8pt] (-0.5 ,-0.4) circle (0.07);
\node [right] at(0.2,1) {$\sigma(n)$};
\node [right] at(0.2,0) {$\sigma(2)$};
\node [right] at(0.2,-0.4){$\sigma(1)$};
\end{tikzpicture}},
\eeq
\beq \label{wzh.six}
\raisebox{-.20in}{
\begin{tikzpicture}
\tikzset{->-/.style=
{decoration={markings,mark=at position #1 with
{\arrow{latex}}},postaction={decorate}}}
\filldraw[draw=white,fill=gray!20] (-0.7,-0.7) rectangle (0,0.7);
\draw [line width =1.5pt,decoration={markings, mark=at position 1 with {\arrow{>}}},postaction={decorate}](0,0.7)--(0,-0.7);
\draw [color = black, line width =1pt] (0 ,0.3) arc (90:270:0.5 and 0.3);
\node [right]  at(0,0.3) {$i$};
\node [right] at(0,-0.3){$j$};
\filldraw[fill=white,line width =0.8pt] (-0.5 ,0) circle (0.07);
\end{tikzpicture}}   = \delta_{\bar j,i }\,  c_i\ \raisebox{-.20in}{
\begin{tikzpicture}
\tikzset{->-/.style=
{decoration={markings,mark=at position #1 with
{\arrow{latex}}},postaction={decorate}}}
\filldraw[draw=white,fill=gray!20] (-0.7,-0.7) rectangle (0,0.7);
\draw [line width =1.5pt](0,0.7)--(0,-0.7);
\end{tikzpicture}},
\eeq
\beq \label{wzh.seven}
\raisebox{-.20in}{
\begin{tikzpicture}
\tikzset{->-/.style=
{decoration={markings,mark=at position #1 with
{\arrow{latex}}},postaction={decorate}}}
\filldraw[draw=white,fill=gray!20] (-0.7,-0.7) rectangle (0,0.7);
\draw [line width =1.5pt](0,0.7)--(0,-0.7);
\draw [color = black, line width =1pt] (-0.7 ,-0.3) arc (-90:90:0.5 and 0.3);
\filldraw[fill=white,line width =0.8pt] (-0.55 ,0.26) circle (0.07);
\end{tikzpicture}}
= \sum_{i=1}^n  (c_{\bar i})^{-1}\, \raisebox{-.20in}{
\begin{tikzpicture}
\tikzset{->-/.style=
{decoration={markings,mark=at position #1 with
{\arrow{latex}}},postaction={decorate}}}
\filldraw[draw=white,fill=gray!20] (-0.7,-0.7) rectangle (0,0.7);
\draw [line width =1.5pt,decoration={markings, mark=at position 1 with {\arrow{>}}},postaction={decorate}](0,0.7)--(0,-0.7);
\draw [line width =1pt](-0.7,0.3)--(0,0.3);
\draw [line width =1pt](-0.7,-0.3)--(0,-0.3);
\filldraw[fill=white,line width =0.8pt] (-0.3 ,0.3) circle (0.07);
\filldraw[fill=black,line width =0.8pt] (-0.3 ,-0.3) circle (0.07);
\node [right]  at(0,0.3) {$i$};
\node [right]  at(0,-0.3) {$\bar{i}$};
\end{tikzpicture}},
\eeq
\beq\label{wzh.eight}
\raisebox{-.20in}{

\begin{tikzpicture}
\tikzset{->-/.style=

{decoration={markings,mark=at position #1 with

{\arrow{latex}}},postaction={decorate}}}
\filldraw[draw=white,fill=gray!20] (-0,-0.2) rectangle (1, 1.2);
\draw [line width =1.5pt,decoration={markings, mark=at position 1 with {\arrow{>}}},postaction={decorate}](1,1.2)--(1,-0.2);
\draw [line width =1pt](0.6,0.6)--(1,1);
\draw [line width =1pt](0.6,0.4)--(1,0);
\draw[line width =1pt] (0,0)--(0.4,0.4);
\draw[line width =1pt] (0,1)--(0.4,0.6);
\draw[line width =1pt] (0.4,0.6)--(0.6,0.4);
\filldraw[fill=white,line width =0.8pt] (0.2 ,0.2) circle (0.07);
\filldraw[fill=white,line width =0.8pt] (0.2 ,0.8) circle (0.07);
\node [right]  at(1,1) {$i$};
\node [right]  at(1,0) {$j$};
\end{tikzpicture}
} =q^{-\frac{1}{n}}\left(\delta_{{j<i} }(q-q^{-1})\raisebox{-.20in}{

\begin{tikzpicture}
\tikzset{->-/.style=

{decoration={markings,mark=at position #1 with

{\arrow{latex}}},postaction={decorate}}}
\filldraw[draw=white,fill=gray!20] (-0,-0.2) rectangle (1, 1.2);
\draw [line width =1.5pt,decoration={markings, mark=at position 1 with {\arrow{>}}},postaction={decorate}](1,1.2)--(1,-0.2);
\draw [line width =1pt](0,0.8)--(1,0.8);
\draw [line width =1pt](0,0.2)--(1,0.2);
\filldraw[fill=white,line width =0.8pt] (0.2 ,0.8) circle (0.07);
\filldraw[fill=white,line width =0.8pt] (0.2 ,0.2) circle (0.07);
\node [right]  at(1,0.8) {$i$};
\node [right]  at(1,0.2) {$j$};
\end{tikzpicture}
}+q^{\delta_{i,j}}\raisebox{-.20in}{

\begin{tikzpicture}
\tikzset{->-/.style=

{decoration={markings,mark=at position #1 with

{\arrow{latex}}},postaction={decorate}}}
\filldraw[draw=white,fill=gray!20] (-0,-0.2) rectangle (1, 1.2);
\draw [line width =1.5pt,decoration={markings, mark=at position 1 with {\arrow{>}}},postaction={decorate}](1,1.2)--(1,-0.2);
\draw [line width =1pt](0,0.8)--(1,0.8);
\draw [line width =1pt](0,0.2)--(1,0.2);
\filldraw[fill=white,line width =0.8pt] (0.2 ,0.8) circle (0.07);
\filldraw[fill=white,line width =0.8pt] (0.2 ,0.2) circle (0.07);
\node [right]  at(1,0.8) {$j$};
\node [right]  at(1,0.2) {$i$};
\end{tikzpicture}
}\right),
\eeq
where   
$\delta_{j<i}= \left \{
 \begin{array}{rr}
     1,                    & j<i\\
     0,                                 & \text{otherwise}
 \end{array}
 \right.,
\delta_{i,j}= \left \{
 \begin{array}{rr}
     1,                    & i=j\\
     0,                                 & \text{otherwise}
 \end{array}
 \right.$. Each shaded rectangle in the above relations is the projection of a cube in $M$. The lines contained in the shaded rectangle represent parts of stated $n$-webs with framing  pointing to  readers. The thick line in the edge of shaded rectangle is part of the marking. For detailed explanation for the above relations, please refer to \cite{le2021stated}.


\subsection{Functoriality}

For any two marked 3-manifolds $(M,\mathcal{N}),(M',\N')$, if an orientation preserving  embedding $f:M\rightarrow M'$ maps 
$\N$ to $\N'$ and preserves  orientations between $\N$ and $\N'$, we call $f$ an embedding from $(M,\mathcal{N})$ to $(M',\N')$. Clearly $f$ induces a linear map $f_{*}:S_n(M,\mathcal{N},v)\rightarrow S_n(M',\mathcal{N}',v)$ \cite{le2021stated}.

\subsection{Splitting map}

Let $(M,\mathcal{N})$ be any marked 3-manifold, and $D$ be a properly embedded disk in $M$ such  that there is no intersection between $D$ and the closure of $\N$.
 After removing a collar open neighborhood of $D$, we get a new 3-manifold $M'$. And $\partial M'$ contains two copies $D_1$ and 
$D_2$ of $D$ such that we can get $M$ from $M'$ by gluing $D_1$ and $D_2$. We use pr to denote the obvious projection from $M'$ to $M$.

\def \cN{\mathbb{N}}
\def \MN {(M,\mathcal{N})}

Let $\beta\subset D$ be an oriented open interval. Suppose $\text{pr}^{-1}(\beta) = \beta_{1}\cup \beta_2$ with
   $\beta_1\in D_1$ and $\beta_2\in D_2$. We cut $(M,\mathcal{N})$ along $(D,\beta)$ to  obtain a new marked 3-manifold   $(M', \N')$, where $\N' = \N\cup \beta_1\cup \beta_2$. We will  denote   $(M', \N')$ as  Cut$_{(D,\beta)}(M, \N )$.
 It is easy to see that Cut$_{(D,\beta)}(M,\mathcal{N})$ is defined up to isomorphism.
If $\beta'$ is another oriented open interval in $D$, obviously we have Cut$_{(D,\beta)}(M,\mathcal{N})$ is isomorphic to Cut$_{(D,\beta')}(M,\mathcal{N})$.

For a stated $n$-web $l$ in $\MN$, we say $l$
 is {\bf $(D,\beta)$-transverse} if the vertices of $l$ are not in $D$, $l\cap D = l\cap \beta$, and the framing of $l$ at each point in $l\cap\beta$ is given by the  velocity vector of $\beta$.  For any map $s:l\cap \beta\rightarrow\{1,2,\dots,n\}$, let $l_s$, which is a stated $n$-web in $\text{Cut}_{(D, \beta)}(M,\mathcal{N})$, be the lift of $l$ such that for every $P\in l\cap\beta$ the two newly created boundary points corresponding to $P$ both  have the state $s(P)$.  There is a linear homomorphism \cite{le2021stated}, called the splitting map, $\Theta_{(D,\beta)}:S_n(M,\mathcal{N},v)\rightarrow S_n(M',\cN',v)$  defined by
\begin{equation}
\Theta_{(D,\beta)}(l) = \sum_{s:l\cap \beta\rightarrow\{1,2,\dots,n\}}l_s.
\end{equation}
 When there is no confusion we can omit the subscript for $\Theta_{(D,\beta)}$.

\subsection{Reversing orientations of $n$-webs}

Let $\cev{\alpha}$ denote an $n$-web $\alpha$ with its orientation reversed (and unchanged framing). 
\begin{corollary}[\cite{le2021stated}]\label{c.orient-rev}
$$\cev {\,\cdot\,}: S_n(M,\mathcal{N},v)\to S_n(M,\mathcal{N},v)$$
is a well defined linear automorphism.
\end{corollary}

\subsection{Punctured bordered surfaces and stated $SL_n$-skein algebras}\label{subb2.4}


A {\bf punctured bordered surface} $\Sigma$ is $\overline{\Sigma}-U$ where $\overline{\Sigma}$ is a compact oriented surface and $U$ is a finite set of $\overline{\Sigma}$ such that every component of $\partial \overline{\Sigma}$ intersects $U$.  
For simplicity, we will call a punctured bordered surface a {\bf pb surface}.


A pb surface $\Sigma$ is called an 
 {\bf essentially bordered pb surface} if every component of $\Sigma$ has non-empty boundary.

The stated $SL_n$-skein algebra, denoted as $S_n(\Sigma,v)$, of a pb surface $\Sigma $ is defined as following: For every  component $c$ of $\partial \Sigma$, we choose a point $x_c$. Let $M = \Sigma\times[-1,1]$ and $\N=\cup_{c}\, x_c \times (-1,1)$  where $c$ goes over all components of $\partial \Sigma$. Then we define $S_n(\Sigma,v) $ to be $S_n(M,\mathcal{N},v)$. We will call $(M,\mathcal{N})$  the thickening of the pb surface $\Sigma$.
Obviously $S_n(\Sigma,v) $ admits an algebra structure. For any two stated $n$-webs $l_1$ and $l_2$ in  the thickening of $\Sigma$, we define $l_1l_2\in S_n({\Sigma},v)$ to be the result of stacking $l_1$ above $l_2$. 


\subsection{Conventions}

An {\bf oriented arc} $\alpha$ in $(M,\mathcal{N})$ is a smooth embedding from $[0,1]$ to $M$ with the orientation given by the positive direction of $[0,1]$ such that $\alpha\cap\partial M
= \{\alpha(0),\alpha(1)\}\cap \N$. A {\bf  framed oriented arc} is an oriented arc with transversal framing
such that the framings at $\{\alpha(0),\alpha(1)\}\cap \N$, if not empty, are given by the velocity  vectors of $\N$.
 An {\bf  oriented circle} in $(M,\mathcal{N})$ is a smooth embedding  $\beta: S^{1}\rightarrow M$ with a chosen orientation such that $\beta\subset int(M)$. A {\bf  framed oriented knot} is an oriented circle with transversal framing.

If both two ends of a framed oriented arc lie in $\N$ 
we call it a  framed oriented boundary arc of $(M,\mathcal{N})$, or just {\bf framed oriented boundary arc} when there is no confusion with $(M,\mathcal{N})$. If the two ends of a framed oriented boundary arc are both stated, we call it a {\bf stated framed oriented boundary arc}.

 For a framed oriented boundary arc $\alpha$, we use $\alpha_{i,j}$ to denote $\alpha$ with two ends stated by $s(\alpha(0 )) = j$ and $s(\alpha(1)) = i$. For a stated framed oriented boundary arc $\alpha$, suppose 
$s(\alpha(0 )) = j$ and $s(\alpha(1)) = i$, we can also use $\alpha_{i,j}$ to denote $\alpha$ to indicate the information that $s(\alpha(0 )) = j$ and $s(\alpha(1)) = i$.
 
For two framed oriented arcs $\alpha,\beta$, we say $\alpha*\beta$ is well-defined if (1) $\alpha\cap \beta=\{\alpha(0) \} = \{\beta(1)\}$  and they have the same framing and velocity vector at point $\beta(1) = \alpha(0)$, or (2) $\alpha\cap \beta=\{\alpha(0),\alpha(1) \} = \{\beta(0),\beta(1)\}$, where $\alpha(0) = \beta(1)$ and
$\alpha(1) = \beta(0)$, and they have the same framings and velocity vectors at their intersecting points. And we use $\alpha*\beta$ to denote the new framed oriented arc (or framed oriented knot) obtained by connecting $\alpha$ and $\beta$ at their intersecting points. Note that it is possible that $\alpha*\beta$ is not a well-defined framed oriented arc (or framed oriented knot)
 because the intersecting points could be contained in $\N$. If this happens, we just isotopicly push  the parts nearby 
intersecting points  to the inside of $M$. Then we obtain a well-defined framed oriented arc (or framed oriented knot), which is still denoted as $\alpha*\beta$.
Note that for case (2), we have both $\alpha*\beta$ and $\beta*\alpha$ are well-defined and they represent the same framed oriented
knot.

Suppose $R_1,R_2$ are two algebras, and $f$ is map from $R_1$ to $R_2$. Let $A=(a_{i,j})$ be a $k_1$ by $k_2$ matrix in $R_1$ where $k_1,k_2$ are two positive integers. We define $f(A)$ to be a $k_1$ by $k_2$ matrix in $R_2$
with $[f(A)]_{i,j} = f(a_{i,j})$. If $f$ is an algebra homomorphism, we have 
$f(I)= I, f(A_1 A_2) = f(A_1)f(A_2), f(A_3+A_4) = f(A_3) + f(A_4), f(A_5^{-1}) = (f(A_5))^{-1}$
where $I$ is the identity matrix in any size and $A_t,1\leq t\leq 5,$ are matrices in $R_1$ such that the above operations for $A_t$ make sense.

In this paper, when we talk about 3-manifold, we always mean a 3-manifold with a chosen orientation and a chosen
 Riemannian metric.

\section{The commutative algebra structure for $S_n(M,\mathcal{N},1)$ and the coordinate ring}\label{sec3}

In this section we will give a commutative algebra structure to $S_n(M,\mathcal{N},1)$, and introduce an algebraic set related to $(M,\mathcal{N})$. The main goal of this section is to construct a surjective algebra homomorphism from $S_n(M,\mathcal{N},1)$ to the coordinate ring of this algebraic set. In next section, we will prove the well-definedness and surjectivity of this algebra homomorphism.

Recall that for any positive integer $k$, we define 
$$[k]=\frac{q^k-q^{-k}}{q-q^{-1}}.$$

\begin{lemma}
For any positive integer $k$, we have 
$$\sum_{\sigma\in S_k} (q^2)^{\ell(\sigma)} = [k]! q^{\frac{k(k-1)}{2}}$$
where $[k]! =\prod_{1\leq i\leq k}[i]$.
\end{lemma}
\begin{proof}
We prove this by using mathmatical induction on $k$. Obviously it is true for $k=1$. 

Suppose we have $\sum_{\sigma\in S_k} (q^2)^{\ell(\sigma)} = [k]! q^{\frac{k(k-1)}{2}}.$ Then 
\begin{align*}
&\sum_{\sigma\in S_{k+1}} (q^2)^{\ell(\sigma)} = \sum_{1\leq i\leq k+1}\left (
\sum_{\sigma\in S_{k+1},\sigma(k+1) 
=i} q^{2\ell(\sigma)}\right ) \\
=&(\sum_{1\leq i\leq k+1} q^{2(k+1-i)}) [k]! q^{\frac{k(k-1)}{2}}=
[k+1]! q^{\frac{(k+1)(k)}{2}}.
\end{align*}
\end{proof}

\subsection{Hight exchange relations for boundary arcs with opposite orientations}  In this subsection, we try to derive the relation between $\raisebox{-.20in}{

\begin{tikzpicture}
\tikzset{->-/.style=

{decoration={markings,mark=at position #1 with

{\arrow{latex}}},postaction={decorate}}}
\filldraw[draw=white,fill=gray!20] (-0,-0.2) rectangle (1, 1.2);
\draw [line width =1.5pt,decoration={markings, mark=at position 1 with {\arrow{>}}},postaction={decorate}](1,-0.2)--(1,1.2);
\draw [line width =1pt,decoration={markings, mark=at position 0.5 with {\arrow{>}}},postaction={decorate}](0,0.8)--(1,0.8);
\draw [line width =1pt,decoration={markings, mark=at position 0.5 with {\arrow{<}}},postaction={decorate}](0,0.2)--(1,0.2);
\node [right]  at(1,0.8) {$j$};
\node [right]  at(1,0.2) {$i$};
\end{tikzpicture}
}$ and $\raisebox{-.20in}{

\begin{tikzpicture}
\tikzset{->-/.style=

{decoration={markings,mark=at position #1 with

{\arrow{latex}}},postaction={decorate}}}
\filldraw[draw=white,fill=gray!20] (-0,-0.2) rectangle (1, 1.2);
\draw [line width =1.5pt,decoration={markings, mark=at position 1 with {\arrow{>}}},postaction={decorate}](1,1.2)--(1,-0.2);
\draw [line width =1pt,decoration={markings, mark=at position 0.5 with {\arrow{>}}},postaction={decorate}](0,0.8)--(1,0.8);
\draw [line width =1pt,decoration={markings, mark=at position 0.5 with {\arrow{<}}},postaction={decorate}](0,0.2)--(1,0.2);
\node [right]  at(1,0.8) {$j$};
\node [right]  at(1,0.2) {$i$};
\end{tikzpicture}
}$
(between $\raisebox{-.20in}{

\begin{tikzpicture}
\tikzset{->-/.style=

{decoration={markings,mark=at position #1 with

{\arrow{latex}}},postaction={decorate}}}
\filldraw[draw=white,fill=gray!20] (-0,-0.2) rectangle (1, 1.2);
\draw [line width =1.5pt,decoration={markings, mark=at position 1 with {\arrow{>}}},postaction={decorate}](1,-0.2)--(1,1.2);
\draw [line width =1pt,decoration={markings, mark=at position 0.5 with {\arrow{<}}},postaction={decorate}](0,0.8)--(1,0.8);
\draw [line width =1pt,decoration={markings, mark=at position 0.5 with {\arrow{>}}},postaction={decorate}](0,0.2)--(1,0.2);
\node [right]  at(1,0.8) {$j$};
\node [right]  at(1,0.2) {$i$};
\end{tikzpicture}
}$ and $\raisebox{-.20in}{

\begin{tikzpicture}
\tikzset{->-/.style=

{decoration={markings,mark=at position #1 with

{\arrow{latex}}},postaction={decorate}}}
\filldraw[draw=white,fill=gray!20] (-0,-0.2) rectangle (1, 1.2);
\draw [line width =1.5pt,decoration={markings, mark=at position 1 with {\arrow{>}}},postaction={decorate}](1,1.2)--(1,-0.2);
\draw [line width =1pt,decoration={markings, mark=at position 0.5 with {\arrow{<}}},postaction={decorate}](0,0.8)--(1,0.8);
\draw [line width =1pt,decoration={markings, mark=at position 0.5 with {\arrow{>}}},postaction={decorate}](0,0.2)--(1,0.2);
\node [right]  at(1,0.8) {$j$};
\node [right]  at(1,0.2) {$i$};
\end{tikzpicture}
}$).

\begin{proposition}\label{prop3.1}
Let $(M,\mathcal{N})$ be any marked 3-manifold with $\N\neq \emptyset$. In $S_n(M,\mathcal{N},v)$, we have 

\begin{align*}
\raisebox{-.20in}{
\begin{tikzpicture}
\tikzset{->-/.style=
{decoration={markings,mark=at position #1 with
{\arrow{latex}}},postaction={decorate}}}
\filldraw[draw=white,fill=gray!20] (-0,-0.2) rectangle (1, 1.2);
\draw [line width =1.5pt,decoration={markings, mark=at position 1 with {\arrow{>}}},postaction={decorate}](1,-0.2)--(1,1.2);
\draw [line width =1pt](0,0.8)--(1,0.8);
\draw [line width =1pt](0,0.2)--(1,0.2);
\filldraw[fill=black,line width =0.8pt] (0.2 ,0.8) circle (0.07);
\filldraw[fill=white,line width =0.8pt] (0.2 ,0.2) circle (0.07);
\node [right]  at(1,0.8) {$j$};
\node [right]  at(1,0.2) {$i$};
\end{tikzpicture}
}
=
\raisebox{-.20in}{
\begin{tikzpicture}
\tikzset{->-/.style=
{decoration={markings,mark=at position #1 with
{\arrow{latex}}},postaction={decorate}}}
\filldraw[draw=white,fill=gray!20] (-0,-0.2) rectangle (1, 1.2);
\draw [line width =1.5pt,decoration={markings, mark=at position 1 with {\arrow{>}}},postaction={decorate}](1,1.2)--(1,-0.2);
\draw [line width =1pt](0.6,0.6)--(1,1);
\draw [line width =1pt](0.6,0.4)--(1,0);
\draw[line width =1pt] (0,0)--(0.4,0.4);
\draw[line width =1pt] (0,1)--(0.4,0.6);
\draw[line width =1pt] (0.4,0.6)--(0.6,0.4);
\filldraw[fill=white,line width =0.8pt] (0.2 ,0.2) circle (0.07);
\filldraw[fill=black,line width =0.8pt] (0.2 ,0.8) circle (0.07);
\node [right]  at(1,1) {$i$};
\node [right]  at(1,0) {$j$};
\end{tikzpicture}
}
&= \left \{
 \begin{array}{ll}
     q^{\frac{1-n}{n}}
\left(
\raisebox{-.20in}{
\begin{tikzpicture}
\tikzset{->-/.style=
{decoration={markings,mark=at position #1 with
{\arrow{latex}}},postaction={decorate}}}
\filldraw[draw=white,fill=gray!20] (-0,-0.2) rectangle (1, 1.2);
\draw [line width =1.5pt,decoration={markings, mark=at position 1 with {\arrow{>}}},postaction={decorate}](1,1.2)--(1,-0.2);
\draw [line width =1pt](0,0.8)--(1,0.8);
\draw [line width =1pt](0,0.2)--(1,0.2);
\filldraw[fill=black,line width =0.8pt] (0.2 ,0.8) circle (0.07);
\filldraw[fill=white,line width =0.8pt] (0.2 ,0.2) circle (0.07);
\node [right]  at(1,0.8) {$j$};
\node [right]  at(1,0.2) {$i$};
\end{tikzpicture}
}
+c_i (1-q^2)\sum_{j<k\leq n}
c_{\bar{k}}^{-1}
\raisebox{-.20in}{
\begin{tikzpicture}
\tikzset{->-/.style=
{decoration={markings,mark=at position #1 with
{\arrow{latex}}},postaction={decorate}}}
\filldraw[draw=white,fill=gray!20] (-0,-0.2) rectangle (1, 1.2);
\draw [line width =1.5pt,decoration={markings, mark=at position 1 with {\arrow{>}}},postaction={decorate}](1,1.2)--(1,-0.2);
\draw [line width =1pt](0,0.8)--(1,0.8);
\draw [line width =1pt](0,0.2)--(1,0.2);
\filldraw[fill=black,line width =0.8pt] (0.2 ,0.8) circle (0.07);
\filldraw[fill=white,line width =0.8pt] (0.2 ,0.2) circle (0.07);
\node [right]  at(1,0.8) {$k$};
\node [right]  at(1,0.2) {$\bar{k}$};
\end{tikzpicture}
}
\right),                    & \text{if }j=\bar{i},\\
     q^{\frac{1}{n}}
\raisebox{-.20in}{
\begin{tikzpicture}
\tikzset{->-/.style=
{decoration={markings,mark=at position #1 with
{\arrow{latex}}},postaction={decorate}}}
\filldraw[draw=white,fill=gray!20] (-0,-0.2) rectangle (1, 1.2);
\draw [line width =1.5pt,decoration={markings, mark=at position 1 with {\arrow{>}}},postaction={decorate}](1,1.2)--(1,-0.2);
\draw [line width =1pt](0,0.8)--(1,0.8);
\draw [line width =1pt](0,0.2)--(1,0.2);
\filldraw[fill=black,line width =0.8pt] (0.2 ,0.8) circle (0.07);
\filldraw[fill=white,line width =0.8pt] (0.2 ,0.2) circle (0.07);
\node [right]  at(1,0.8) {$j$};
\node [right]  at(1,0.2) {$i$};
\end{tikzpicture}
}
,     & \text{if }j\neq\bar{i},\\
 \end{array}
 \right.
\end{align*}

\end{proposition}
\begin{proof}
We only prove one case. The other case can be obtained by reversing all the orientations of stated $n$-webs.

We have
\begin{equation}\label{eq2.1}
\raisebox{-.20in}{
\begin{tikzpicture}
\tikzset{->-/.style=
{decoration={markings,mark=at position #1 with
{\arrow{latex}}},postaction={decorate}}}
\filldraw[draw=white,fill=gray!20] (-0,-0.2) rectangle (1, 1.2);
\draw [line width =1.5pt,decoration={markings, mark=at position 1 with {\arrow{>}}},postaction={decorate}](1,-0.2)--(1,1.2);
\draw [line width =1pt,decoration={markings, mark=at position 0.5 with {\arrow{>}}},postaction={decorate}](0,0.8)--(1,0.8);
\draw [line width =1pt,decoration={markings, mark=at position 0.5 with {\arrow{<}}},postaction={decorate}](0,0.2)--(1,0.2);
\node [right]  at(1,0.8) {$j$};
\node [right]  at(1,0.2) {$i$};
\end{tikzpicture}}
=
\raisebox{-.20in}{
\begin{tikzpicture}
\tikzset{->-/.style=
{decoration={markings,mark=at position #1 with
{\arrow{latex}}},postaction={decorate}}}
\filldraw[draw=white,fill=gray!20] (-0,-0.2) rectangle (1, 1.2);
\draw [line width =1.5pt,decoration={markings, mark=at position 1 with {\arrow{>}}},postaction={decorate}](1,1.2)--(1,-0.2);
\draw [line width =1pt](0.6,0.6)--(1,1);
\draw [line width =1pt](0.6,0.4)--(1,0);
\draw[line width =1pt,decoration={markings, mark=at position 0.5 with {\arrow{<}}},postaction={decorate}] (0,0)--(0.4,0.4);
\draw[line width =1pt,decoration={markings, mark=at position 0.5 with {\arrow{>}}},postaction={decorate}] (0,1)--(0.4,0.6);
\draw[line width =1pt] (0.4,0.6)--(0.6,0.4);
\node [right]  at(1,1) {$i$};
\node [right]  at(1,0) {$j$};
\end{tikzpicture}}
=  \delta_{\bar j,i } c_i q^{\frac{1-n}{n}} \; 
\raisebox{-.20in}{
\begin{tikzpicture}
\tikzset{->-/.style=
{decoration={markings,mark=at position #1 with
{\arrow{latex}}},postaction={decorate}}}
\filldraw[draw=white,fill=gray!20] (-0.7,-0.7) rectangle (0,0.7);
\draw [line width =1.5pt](0,0.7)--(0,-0.7);
\draw [color = black, line width =1pt] (-0.7 ,-0.3) arc (-90:90:0.5 and 0.3);
\draw [line width =1pt,decoration={markings, mark=at position 0.5 with {\arrow{<}}},postaction={decorate}](-0.2,-0.02)--(-0.2,0.02);
\end{tikzpicture}}
-(-1)^{\frac{n(n-1)}{2}} ([n-2]!)^{-1}q^{\frac{1}{n}}
\raisebox{-.60in}{

\begin{tikzpicture}
\tikzset{->-/.style=

{decoration={markings,mark=at position #1 with

{\arrow{latex}}},postaction={decorate}}}

\filldraw[draw=white,fill=gray!20] (0,0) rectangle (2.6, 4);
\draw[line width =2pt,decoration={markings, mark=at position 1.0 with {\arrow{>}}},postaction={decorate}](2.6,4) --(2.6,0);
\draw [line width =0.8pt,decoration={markings, mark=at position 0.5 with {\arrow{>}}},postaction={decorate}](1.3,1)--(0,0.5);
\draw[line width =0.8pt,decoration={markings, mark=at position 0.5 with {\arrow{>}}},postaction={decorate}] (1.3,1)--(2.6,0.5)node[right]{{$j$}};
\draw [line width =0.8pt,decoration={markings, mark=at position 1.0 with {\arrow{>}}},postaction={decorate}](1.3,1)--(1.3,2) node[right]{{$n-2$}};
\draw [line width =0.8pt](1.3,2)--(1.3,3);
\draw [line width =0.8pt,decoration={markings, mark=at position 0.5 with {\arrow{>}}},postaction={decorate}](0,3.5)--(1.3,3);
\draw[line width =0.8pt,decoration={markings, mark=at position 0.5 with {\arrow{<}}},postaction={decorate}] (1.3,3)--(2.6,3.5) node[right]{{$i$}};
\end{tikzpicture}
}\hspace*{-.1in}.
\end{equation}
The second equality is because of the parallel equation of equation (51) in \cite{le2021stated}.
We have
$$
\raisebox{-.60in}{

\begin{tikzpicture}
\tikzset{->-/.style=

{decoration={markings,mark=at position #1 with

{\arrow{latex}}},postaction={decorate}}}

\filldraw[draw=white,fill=gray!20] (0,0) rectangle (2.6, 4);
\draw[line width =2pt,decoration={markings, mark=at position 1.0 with {\arrow{>}}},postaction={decorate}](2.6,4) --(2.6,0);
\draw [line width =0.8pt,decoration={markings, mark=at position 0.5 with {\arrow{>}}},postaction={decorate}](1.3,1)--(0,0.5);
\draw[line width =0.8pt,decoration={markings, mark=at position 0.5 with {\arrow{>}}},postaction={decorate}] (1.3,1)--(2.6,0.5)node[right]{{$j$}};
\draw [line width =0.8pt,decoration={markings, mark=at position 1.0 with {\arrow{>}}},postaction={decorate}](1.3,1)--(1.3,2) node[right]{{$n-2$}};
\draw [line width =0.8pt](1.3,2)--(1.3,3);
\draw [line width =0.8pt,decoration={markings, mark=at position 0.5 with {\arrow{>}}},postaction={decorate}](0,3.5)--(1.3,3);
\draw[line width =0.8pt,decoration={markings, mark=at position 0.5 with {\arrow{<}}},postaction={decorate}] (1.3,3)--(2.6,3.5) node[right]{{$i$}};
\end{tikzpicture}
}\hspace*{-.1in}=a^2
\sum_{\tau,\sigma\in S_n}  (-q)^{\ell(\tau)}(-q)^{\ell(\sigma)}
\raisebox{-.99in}{

\begin{tikzpicture}
\tikzset{->-/.style=

{decoration={markings,mark=at position #1 with

{\arrow{latex}}},postaction={decorate}}}

\filldraw[draw=white,fill=gray!20] (0,0.4) rectangle (2, 7.6);
\draw[line width =2pt,decoration={markings, mark=at position 1.0 with {\arrow{>}}},postaction={decorate}](2,7.6) --(2,0.4);
\draw [line width =0.8pt,decoration={markings, mark=at position 0.8 with {\arrow{<}}},postaction={decorate}](2,4.4) arc (90:270:0.8 and 0.4) node[right]{\small{$\tau(n)$}};
\draw [line width =0.8pt,decoration={markings, mark=at position 0.8 with {\arrow{<}}},postaction={decorate}](2,5.2) arc (90:270:1.2 and 1.2) node[right]{\small{$\tau(4)$}};
\draw [line width =0.8pt,decoration={markings, mark=at position 0.8 with {\arrow{<}}},postaction={decorate}](2,5.6) arc (90:270:1.4 and 1.6) node[right]{\small{$\tau(3)$}};
\draw [line width =0.8pt,decoration={markings, mark=at position 0.5 with {\arrow{<}}},postaction={decorate}](0,2)--(2,2)node[right]{\small{$\tau(2)$}};
\draw [line width =0.8pt,decoration={markings, mark=at position 0.5 with {\arrow{>}}},postaction={decorate}](0,6)--(2,6);
\draw [line width =0.8pt,decoration={markings, mark=at position 0.2 with {\arrow{>}}},postaction={decorate}](2,1.6) arc (90:270:0.8 and 0.4) node[right]{\small{$j$}};
\draw [line width =0.8pt,decoration={markings, mark=at position 0.8 with {\arrow{>}}},postaction={decorate}](2,7.2) arc (90:270:0.8 and 0.4) node[right]{\small{$\sigma(n)$}};
\node at(2.5,4.4) {\small $\sigma(1)$};
\node[right] at(2,5.2) {\small $\sigma(n-3)$};
\node [right] at(2,5.6) {\small $\sigma(n-2)$};
\node [right] at(2,6) {\small $\sigma(n-1)$};
\node [right] at(2,7.2) {\small $i$};
\node at(2.5,1.6) {\small $\tau(1)$};

\node at(1.7,4.8) {\vdots};
\end{tikzpicture}
}\hspace*{-.1in}$$
$$
=a^2 t^{\frac{n}{2}}
\sum_{\substack{\tau(1)=\bar{j},\sigma(n)=\bar{i}\\ \overline{\sigma(1)}=\tau(n),\dots,\overline{\sigma(n-2)}=\tau(3),\sigma,\tau\in S_n   }}(-q)^{\ell(\sigma)+\ell(\tau)} c_i c_{\tau(1)}c_{\sigma(n-1)}^{-1}c_{\sigma(n)}^{-1}
\;
\raisebox{-.20in}{

\begin{tikzpicture}
\tikzset{->-/.style=

{decoration={markings,mark=at position #1 with

{\arrow{latex}}},postaction={decorate}}}
\filldraw[draw=white,fill=gray!20] (-0,-0.2) rectangle (1, 1.2);
\draw [line width =1.5pt,decoration={markings, mark=at position 1 with {\arrow{>}}},postaction={decorate}](1,1.2)--(1,-0.2);
\draw [line width =1pt,decoration={markings, mark=at position 0.5 with {\arrow{>}}},postaction={decorate}](0,0.8)--(1,0.8);
\draw [line width =1pt,decoration={markings, mark=at position 0.5 with {\arrow{<}}},postaction={decorate}](0,0.2)--(1,0.2);
\node [right]  at(1,0.8) {$\sigma(n-1)$};
\node [right]  at(1,0.2) {$\tau(2)$};
\end{tikzpicture}
}
$$ where the first equality comes from relation \eqref{wzh.five}.
Then $j\in \{\sigma(n-1), \sigma(n)\}, i\in\{\tau(1),\tau(2)\},$\\$\{\tau(1),\tau(2)\}=\{\overline{\sigma(n-1)},\overline{\sigma(n)}\}$.

Case 1 when $i\neq \overline{j}$. Then $\tau(1) = \overline{j} \neq i =\overline{\sigma(n)}$, furthermore we have 
$\tau(1) = \overline{\sigma(n-1)}=\overline{j}$ and $\tau(2) = \overline{\sigma(n)} = i$.
Then we have 
\begin{align*}
&a^2 t^{\frac{n}{2}}
\sum_{\substack{\tau(1)=\bar{j},\sigma(n)=\bar{i}\\ \overline{\sigma(1)}=\tau(n),\dots,\overline{\sigma(n-2)}=\tau(3),\sigma,\tau\in S_n   }}(-q)^{\ell(\sigma)+\ell(\tau)} c_i c_{\tau(1)}c_{\sigma(n-1)}^{-1}c_{\sigma(n)}^{-1}
\;
\raisebox{-.20in}{
\begin{tikzpicture}
\tikzset{->-/.style=
{decoration={markings,mark=at position #1 with
{\arrow{latex}}},postaction={decorate}}}
\filldraw[draw=white,fill=gray!20] (-0,-0.2) rectangle (1, 1.2);
\draw [line width =1.5pt,decoration={markings, mark=at position 1 with {\arrow{>}}},postaction={decorate}](1,1.2)--(1,-0.2);
\draw [line width =1pt,decoration={markings, mark=at position 0.5 with {\arrow{>}}},postaction={decorate}](0,0.8)--(1,0.8);
\draw [line width =1pt,decoration={markings, mark=at position 0.5 with {\arrow{<}}},postaction={decorate}](0,0.2)--(1,0.2);
\node [right]  at(1,0.8) {$\sigma(n-1)$};
\node [right]  at(1,0.2) {$\tau(2)$};
\end{tikzpicture}
}
\\
=&a^2 t^{\frac{n}{2}}c_i^{2}c_j^{-2}
\sum_{\substack{\tau(1) = \overline{\sigma(n-1)}=\overline{j},\tau(2) = \overline{\sigma(n)} = i\\ \overline{\sigma(1)}=\tau(n),\dots,\overline{\sigma(n-2)}=\tau(3),\sigma,\tau\in S_n   }}(-q)^{l(\sigma)+l(\tau)} 
\;
\raisebox{-.20in}{
\begin{tikzpicture}
\tikzset{->-/.style=
{decoration={markings,mark=at position #1 with
{\arrow{latex}}},postaction={decorate}}}
\filldraw[draw=white,fill=gray!20] (-0,-0.2) rectangle (1, 1.2);
\draw [line width =1.5pt,decoration={markings, mark=at position 1 with {\arrow{>}}},postaction={decorate}](1,1.2)--(1,-0.2);
\draw [line width =1pt,decoration={markings, mark=at position 0.5 with {\arrow{>}}},postaction={decorate}](0,0.8)--(1,0.8);
\draw [line width =1pt,decoration={markings, mark=at position 0.5 with {\arrow{<}}},postaction={decorate}](0,0.2)--(1,0.2);
\node [right]  at(1,0.8) {$j$};
\node [right]  at(1,0.2) {$i$};
\end{tikzpicture}
}
\\
=&a^2 t^{\frac{n}{2}}c_i^{2}c_j^{-2}
(-q)^{2n + 2i -2j -3}[n-2]!q^{\frac{(n-2)(n-3)}{2}}
\;\raisebox{-.20in}{
\begin{tikzpicture}
\tikzset{->-/.style=
{decoration={markings,mark=at position #1 with
{\arrow{latex}}},postaction={decorate}}}
\filldraw[draw=white,fill=gray!20] (-0,-0.2) rectangle (1, 1.2);
\draw [line width =1.5pt,decoration={markings, mark=at position 1 with {\arrow{>}}},postaction={decorate}](1,1.2)--(1,-0.2);
\draw [line width =1pt,decoration={markings, mark=at position 0.5 with {\arrow{>}}},postaction={decorate}](0,0.8)--(1,0.8);
\draw [line width =1pt,decoration={markings, mark=at position 0.5 with {\arrow{<}}},postaction={decorate}](0,0.2)--(1,0.2);
\node [right]  at(1,0.8) {$j$};
\node [right]  at(1,0.2) {$i$};
\end{tikzpicture}
}\\
=&-(-1)^{\frac{n(n-1)}{2}}[n-2]!
\;
\raisebox{-.20in}{
\begin{tikzpicture}
\tikzset{->-/.style=
{decoration={markings,mark=at position #1 with
{\arrow{latex}}},postaction={decorate}}}
\filldraw[draw=white,fill=gray!20] (-0,-0.2) rectangle (1, 1.2);
\draw [line width =1.5pt,decoration={markings, mark=at position 1 with {\arrow{>}}},postaction={decorate}](1,1.2)--(1,-0.2);
\draw [line width =1pt,decoration={markings, mark=at position 0.5 with {\arrow{>}}},postaction={decorate}](0,0.8)--(1,0.8);
\draw [line width =1pt,decoration={markings, mark=at position 0.5 with {\arrow{<}}},postaction={decorate}](0,0.2)--(1,0.2);
\node [right]  at(1,0.8) {$j$};
\node [right]  at(1,0.2) {$i$};
\end{tikzpicture}
}.
\end{align*}
From equation (\ref{eq2.1}), we have 
$\raisebox{-.20in}{
\begin{tikzpicture}
\tikzset{->-/.style=
{decoration={markings,mark=at position #1 with
{\arrow{latex}}},postaction={decorate}}}
\filldraw[draw=white,fill=gray!20] (-0,-0.2) rectangle (1, 1.2);
\draw [line width =1.5pt,decoration={markings, mark=at position 1 with {\arrow{>}}},postaction={decorate}](1,-0.2)--(1,1.2);
\draw [line width =1pt,decoration={markings, mark=at position 0.5 with {\arrow{>}}},postaction={decorate}](0,0.8)--(1,0.8);
\draw [line width =1pt,decoration={markings, mark=at position 0.5 with {\arrow{<}}},postaction={decorate}](0,0.2)--(1,0.2);
\node [right]  at(1,0.8) {$j$};
\node [right]  at(1,0.2) {$i$};
\end{tikzpicture}}
=
\raisebox{-.20in}{
\begin{tikzpicture}
\tikzset{->-/.style=
{decoration={markings,mark=at position #1 with
{\arrow{latex}}},postaction={decorate}}}
\filldraw[draw=white,fill=gray!20] (-0,-0.2) rectangle (1, 1.2);
\draw [line width =1.5pt,decoration={markings, mark=at position 1 with {\arrow{>}}},postaction={decorate}](1,1.2)--(1,-0.2);
\draw [line width =1pt](0.6,0.6)--(1,1);
\draw [line width =1pt](0.6,0.4)--(1,0);
\draw[line width =1pt,decoration={markings, mark=at position 0.5 with {\arrow{<}}},postaction={decorate}] (0,0)--(0.4,0.4);
\draw[line width =1pt,decoration={markings, mark=at position 0.5 with {\arrow{>}}},postaction={decorate}] (0,1)--(0.4,0.6);
\draw[line width =1pt] (0.4,0.6)--(0.6,0.4);
\node [right]  at(1,1) {$i$};
\node [right]  at(1,0) {$j$};
\end{tikzpicture}}=  q^{\frac{1}{n}}\raisebox{-.20in}{
\begin{tikzpicture}
\tikzset{->-/.style=
{decoration={markings,mark=at position #1 with
{\arrow{latex}}},postaction={decorate}}}
\filldraw[draw=white,fill=gray!20] (-0,-0.2) rectangle (1, 1.2);
\draw [line width =1.5pt,decoration={markings, mark=at position 1 with {\arrow{>}}},postaction={decorate}](1,1.2)--(1,-0.2);
\draw [line width =1pt,decoration={markings, mark=at position 0.5 with {\arrow{>}}},postaction={decorate}](0,0.8)--(1,0.8);
\draw [line width =1pt,decoration={markings, mark=at position 0.5 with {\arrow{<}}},postaction={decorate}](0,0.2)--(1,0.2);
\node [right]  at(1,0.8) {$j$};
\node [right]  at(1,0.2) {$i$};
\end{tikzpicture}
}$.

Case 2 when $i=\overline{j}$. We have $\tau(1) =\overline{j} = i =\overline{\sigma(n)}, 
\tau(2) = \overline{\sigma(n-1)}$, futhermore $\tau(k) = \overline{\sigma(n+1-k)}, 1\leq k\leq n$.
Then we have $\ell(\tau) = \ell(\sigma)$. 
Thus 
\begin{align*}
&a^2 t^{\frac{n}{2}}
\sum_{\substack{\tau(1)=\bar{j},\sigma(n)=\bar{i}\\ \overline{\sigma(1)}=\tau(n),\dots,\overline{\sigma(n-2)}=\tau(3),\sigma,\tau\in S_n   }}(-q)^{\ell(\sigma)+\ell(\tau)} c_i c_{\tau(1)}c_{\sigma(n-1)}^{-1}c_{\sigma(n)}^{-1}
\;
\raisebox{-.20in}{
\begin{tikzpicture}
\tikzset{->-/.style=
{decoration={markings,mark=at position #1 with
{\arrow{latex}}},postaction={decorate}}}
\filldraw[draw=white,fill=gray!20] (-0,-0.2) rectangle (1, 1.2);
\draw [line width =1.5pt,decoration={markings, mark=at position 1 with {\arrow{>}}},postaction={decorate}](1,1.2)--(1,-0.2);
\draw [line width =1pt,decoration={markings, mark=at position 0.5 with {\arrow{>}}},postaction={decorate}](0,0.8)--(1,0.8);
\draw [line width =1pt,decoration={markings, mark=at position 0.5 with {\arrow{<}}},postaction={decorate}](0,0.2)--(1,0.2);
\node [right]  at(1,0.8) {$\sigma(n-1)$};
\node [right]  at(1,0.2) {$\tau(2)$};
\end{tikzpicture}
}
\\
=&a^2 t^{\frac{n}{2}} t^{-1} c_i^{3}
\sum_{\sigma\in S_n,\sigma(n) = \bar{i}}(-q)^{2\ell(\sigma)}c_{\sigma(n-1)}^{-1}
\;
\raisebox{-.20in}{
\begin{tikzpicture}
\tikzset{->-/.style=
{decoration={markings,mark=at position #1 with
{\arrow{latex}}},postaction={decorate}}}
\filldraw[draw=white,fill=gray!20] (-0,-0.2) rectangle (1, 1.2);
\draw [line width =1.5pt,decoration={markings, mark=at position 1 with {\arrow{>}}},postaction={decorate}](1,1.2)--(1,-0.2);
\draw [line width =1pt,decoration={markings, mark=at position 0.5 with {\arrow{>}}},postaction={decorate}](0,0.8)--(1,0.8);
\draw [line width =1pt,decoration={markings, mark=at position 0.5 with {\arrow{<}}},postaction={decorate}](0,0.2)--(1,0.2);
\node [right]  at(1,0.8) {$\sigma(n-1)$};
\node [right]  at(1,0.2) {$\overline{\sigma(n-1)}$};
\end{tikzpicture}
}
\\
=&a^2 t^{\frac{n}{2}} t^{-1} c_i^{3}\sum_{1\leq k\leq n,k\neq \bar{i}}\left (
\sum_{\sigma\in S_n,\sigma(n-1) = k,\sigma(n) = \bar{i}}(-q)^{2\ell(\sigma)}c_{k}^{-1}
\;
\raisebox{-.20in}{
\begin{tikzpicture}
\tikzset{->-/.style=
{decoration={markings,mark=at position #1 with
{\arrow{latex}}},postaction={decorate}}}
\filldraw[draw=white,fill=gray!20] (-0,-0.2) rectangle (1, 1.2);
\draw [line width =1.5pt,decoration={markings, mark=at position 1 with {\arrow{>}}},postaction={decorate}](1,1.2)--(1,-0.2);
\draw [line width =1pt,decoration={markings, mark=at position 0.5 with {\arrow{>}}},postaction={decorate}](0,0.8)--(1,0.8);
\draw [line width =1pt,decoration={markings, mark=at position 0.5 with {\arrow{<}}},postaction={decorate}](0,0.2)--(1,0.2);
\node [right]  at(1,0.8) {$k$};
\node [right]  at(1,0.2) {$\bar{k}$};
\end{tikzpicture}
}
 \right )\\
=&a^2 t^{\frac{n}{2}} t^{-1} c_i^{3}  \sum_{1\leq k\leq n,k\neq \bar{i}} c_{k}^{-1}
\;
\raisebox{-.20in}{
\begin{tikzpicture}
\tikzset{->-/.style=
{decoration={markings,mark=at position #1 with
{\arrow{latex}}},postaction={decorate}}}
\filldraw[draw=white,fill=gray!20] (-0,-0.2) rectangle (1, 1.2);
\draw [line width =1.5pt,decoration={markings, mark=at position 1 with {\arrow{>}}},postaction={decorate}](1,1.2)--(1,-0.2);
\draw [line width =1pt,decoration={markings, mark=at position 0.5 with {\arrow{>}}},postaction={decorate}](0,0.8)--(1,0.8);
\draw [line width =1pt,decoration={markings, mark=at position 0.5 with {\arrow{<}}},postaction={decorate}](0,0.2)--(1,0.2);
\node [right]  at(1,0.8) {$k$};
\node [right]  at(1,0.2) {$\bar{k}$};
\end{tikzpicture}
}
 \left (
\sum_{\sigma\in S_n,\sigma(n-1) = k,\sigma(n) = \bar{i}}(-q)^{2\ell(\sigma)} \right )  \\
=&a^2 t^{\frac{n}{2}} t^{-1} c_i^{3}  \sum_{1\leq k\leq n,k\neq \bar{i}} c_{k}^{-1} [n-2]! q^{\frac{(n-2)(n-3)}{2}} q^{4n -2\bar{i} - 2} q^{-2k + 2\delta_{k>\bar{i}}}
\;
\raisebox{-.20in}{
\begin{tikzpicture}
\tikzset{->-/.style=
{decoration={markings,mark=at position #1 with
{\arrow{latex}}},postaction={decorate}}}
\filldraw[draw=white,fill=gray!20] (-0,-0.2) rectangle (1, 1.2);
\draw [line width =1.5pt,decoration={markings, mark=at position 1 with {\arrow{>}}},postaction={decorate}](1,1.2)--(1,-0.2);
\draw [line width =1pt,decoration={markings, mark=at position 0.5 with {\arrow{>}}},postaction={decorate}](0,0.8)--(1,0.8);
\draw [line width =1pt,decoration={markings, mark=at position 0.5 with {\arrow{<}}},postaction={decorate}](0,0.2)--(1,0.2);
\node [right]  at(1,0.8) {$k$};
\node [right]  at(1,0.2) {$\bar{k}$};
\end{tikzpicture}
}
  \\
=&a^2 t^{\frac{n}{2}} t^{-2} c_i^{3} [n-2]!  q^{\frac{(n-2)(n-3)}{2}} q^{4n -2\bar{i} - 2}\sum_{1\leq k\leq n,k\neq \bar{i}} c_{\bar{k}}  q^{-2k + 2\delta_{k>\bar{i}}}
\;
\raisebox{-.20in}{
\begin{tikzpicture}
\tikzset{->-/.style=
{decoration={markings,mark=at position #1 with
{\arrow{latex}}},postaction={decorate}}}
\filldraw[draw=white,fill=gray!20] (-0,-0.2) rectangle (1, 1.2);
\draw [line width =1.5pt,decoration={markings, mark=at position 1 with {\arrow{>}}},postaction={decorate}](1,1.2)--(1,-0.2);
\draw [line width =1pt,decoration={markings, mark=at position 0.5 with {\arrow{>}}},postaction={decorate}](0,0.8)--(1,0.8);
\draw [line width =1pt,decoration={markings, mark=at position 0.5 with {\arrow{<}}},postaction={decorate}](0,0.2)--(1,0.2);
\node [right]  at(1,0.8) {$k$};
\node [right]  at(1,0.2) {$\bar{k}$};
\end{tikzpicture}
}
   \\
=&a^2 t^{\frac{n}{2}} t^{-2} c_i^{3} [n-2]!  q^{\frac{(n-2)(n-3)}{2}} q^{4n -2\bar{i} - 2}
q^{-1-\frac{1}{n}}
\sum_{1\leq k\leq n,k\neq \bar{i}} c_{\bar{k}}^{-1}  q^{ 2\delta_{k>\bar{i}}}
\;
\raisebox{-.20in}{
\begin{tikzpicture}
\tikzset{->-/.style=
{decoration={markings,mark=at position #1 with
{\arrow{latex}}},postaction={decorate}}}
\filldraw[draw=white,fill=gray!20] (-0,-0.2) rectangle (1, 1.2);
\draw [line width =1.5pt,decoration={markings, mark=at position 1 with {\arrow{>}}},postaction={decorate}](1,1.2)--(1,-0.2);
\draw [line width =1pt,decoration={markings, mark=at position 0.5 with {\arrow{>}}},postaction={decorate}](0,0.8)--(1,0.8);
\draw [line width =1pt,decoration={markings, mark=at position 0.5 with {\arrow{<}}},postaction={decorate}](0,0.2)--(1,0.2);
\node [right]  at(1,0.8) {$k$};
\node [right]  at(1,0.2) {$\bar{k}$};
\end{tikzpicture}
}
  \\
=&(-1)^{\frac{n(n-1)}{2}}c_i [n-2]! q^{-1}
\sum_{1\leq k\leq n,k\neq \bar{i}} c_{\bar{k}}^{-1}  q^{ 2\delta_{k>\bar{i}}}
\;\raisebox{-.20in}{
\begin{tikzpicture}
\tikzset{->-/.style=
{decoration={markings,mark=at position #1 with
{\arrow{latex}}},postaction={decorate}}}
\filldraw[draw=white,fill=gray!20] (-0,-0.2) rectangle (1, 1.2);
\draw [line width =1.5pt,decoration={markings, mark=at position 1 with {\arrow{>}}},postaction={decorate}](1,1.2)--(1,-0.2);
\draw [line width =1pt,decoration={markings, mark=at position 0.5 with {\arrow{>}}},postaction={decorate}](0,0.8)--(1,0.8);
\draw [line width =1pt,decoration={markings, mark=at position 0.5 with {\arrow{<}}},postaction={decorate}](0,0.2)--(1,0.2);
\node [right]  at(1,0.8) {$k$};
\node [right]  at(1,0.2) {$\bar{k}$};
\end{tikzpicture}
}   .\\
\end{align*}
From equation (\ref{eq2.1}) and relation (\ref{wzh.seven}), we have 
$$\raisebox{-.20in}{
\begin{tikzpicture}
\tikzset{->-/.style=
{decoration={markings,mark=at position #1 with
{\arrow{latex}}},postaction={decorate}}}
\filldraw[draw=white,fill=gray!20] (-0,-0.2) rectangle (1, 1.2);
\draw [line width =1.5pt,decoration={markings, mark=at position 1 with {\arrow{>}}},postaction={decorate}](1,-0.2)--(1,1.2);
\draw [line width =1pt,decoration={markings, mark=at position 0.5 with {\arrow{>}}},postaction={decorate}](0,0.8)--(1,0.8);
\draw [line width =1pt,decoration={markings, mark=at position 0.5 with {\arrow{<}}},postaction={decorate}](0,0.2)--(1,0.2);
\node [right]  at(1,0.8) {$j$};
\node [right]  at(1,0.2) {$i$};
\end{tikzpicture}}
=
\raisebox{-.20in}{
\begin{tikzpicture}
\tikzset{->-/.style=
{decoration={markings,mark=at position #1 with
{\arrow{latex}}},postaction={decorate}}}
\filldraw[draw=white,fill=gray!20] (-0,-0.2) rectangle (1, 1.2);
\draw [line width =1.5pt,decoration={markings, mark=at position 1 with {\arrow{>}}},postaction={decorate}](1,1.2)--(1,-0.2);
\draw [line width =1pt](0.6,0.6)--(1,1);
\draw [line width =1pt](0.6,0.4)--(1,0);
\draw[line width =1pt,decoration={markings, mark=at position 0.5 with {\arrow{<}}},postaction={decorate}] (0,0)--(0.4,0.4);
\draw[line width =1pt,decoration={markings, mark=at position 0.5 with {\arrow{>}}},postaction={decorate}] (0,1)--(0.4,0.6);
\draw[line width =1pt] (0.4,0.6)--(0.6,0.4);
\node [right]  at(1,1) {$i$};
\node [right]  at(1,0) {$j$};
\end{tikzpicture}}=  q^{\frac{1-n}{n}}\left(\raisebox{-.20in}{
\begin{tikzpicture}
\tikzset{->-/.style=
{decoration={markings,mark=at position #1 with
{\arrow{latex}}},postaction={decorate}}}
\filldraw[draw=white,fill=gray!20] (-0,-0.2) rectangle (1, 1.2);
\draw [line width =1.5pt,decoration={markings, mark=at position 1 with {\arrow{>}}},postaction={decorate}](1,1.2)--(1,-0.2);
\draw [line width =1pt,decoration={markings, mark=at position 0.5 with {\arrow{>}}},postaction={decorate}](0,0.8)--(1,0.8);
\draw [line width =1pt,decoration={markings, mark=at position 0.5 with {\arrow{<}}},postaction={decorate}](0,0.2)--(1,0.2);
\node [right]  at(1,0.8) {$j$};
\node [right]  at(1,0.2) {$i$};
\end{tikzpicture}
}+c_i (1-q^2)\sum_{j<k\leq n}
c_{\bar{k}}^{-1}
\raisebox{-.20in}{
\begin{tikzpicture}
\tikzset{->-/.style=
{decoration={markings,mark=at position #1 with
{\arrow{latex}}},postaction={decorate}}}
\filldraw[draw=white,fill=gray!20] (-0,-0.2) rectangle (1, 1.2);
\draw [line width =1.5pt,decoration={markings, mark=at position 1 with {\arrow{>}}},postaction={decorate}](1,1.2)--(1,-0.2);
\draw [line width =1pt,decoration={markings, mark=at position 0.5 with {\arrow{>}}},postaction={decorate}](0,0.8)--(1,0.8);
\draw [line width =1pt,decoration={markings, mark=at position 0.5 with {\arrow{<}}},postaction={decorate}](0,0.2)--(1,0.2);
\node [right]  at(1,0.8) {$k$};
\node [right]  at(1,0.2) {$\bar{k}$};
\end{tikzpicture}
}\right).$$

\end{proof}

\subsection{Commutative algebra structure for $S_n(M,\mathcal{N},1)$} In this subsection, we will give a commutative algebra structure for $S_n(M,\mathcal{N},1)$.
To do so, first we define the product of two stated $n$-webs as the disjoint union, that is, we first isotope two stated $n$-webs such that they have no intersection, then take their union as the product.  To prove this product is well defined, we have to show the product is independent of how we take the union of these two webs. 

\begin{corollary}\label{cccc3.2}
For any marked 3-manifold $\MN$, we have $S_n(M,\mathcal{N},1)$ is a commutative algebra under the above defined multiplication.
\end{corollary}
\begin{proof}
Because of relations \eqref{w.cross} and \eqref{wzh.eight}, it suffices to show  for any two states $i,j$, we have 
$$\raisebox{-.20in}{
\begin{tikzpicture}
\tikzset{->-/.style=
{decoration={markings,mark=at position #1 with
{\arrow{latex}}},postaction={decorate}}}
\filldraw[draw=white,fill=gray!20] (-0,-0.2) rectangle (1, 1.2);
\draw [line width =1.5pt,decoration={markings, mark=at position 1 with {\arrow{>}}},postaction={decorate}](1,-0.2)--(1,1.2);
\draw [line width =1pt,decoration={markings, mark=at position 0.5 with {\arrow{>}}},postaction={decorate}](0,0.8)--(1,0.8);
\draw [line width =1pt,decoration={markings, mark=at position 0.5 with {\arrow{<}}},postaction={decorate}](0,0.2)--(1,0.2);
\node [right]  at(1,0.8) {$j$};
\node [right]  at(1,0.2) {$i$};
\end{tikzpicture}
}= \raisebox{-.20in}{
\begin{tikzpicture}
\tikzset{->-/.style=
{decoration={markings,mark=at position #1 with
{\arrow{latex}}},postaction={decorate}}}
\filldraw[draw=white,fill=gray!20] (-0,-0.2) rectangle (1, 1.2);
\draw [line width =1.5pt,decoration={markings, mark=at position 1 with {\arrow{>}}},postaction={decorate}](1,1.2)--(1,-0.2);
\draw [line width =1pt,decoration={markings, mark=at position 0.5 with {\arrow{>}}},postaction={decorate}](0,0.8)--(1,0.8);
\draw [line width =1pt,decoration={markings, mark=at position 0.5 with {\arrow{<}}},postaction={decorate}](0,0.2)--(1,0.2);
\node [right]  at(1,0.8) {$j$};
\node [right]  at(1,0.2) {$i$};
\end{tikzpicture}
},\;
\raisebox{-.20in}{
\begin{tikzpicture}
\tikzset{->-/.style=
{decoration={markings,mark=at position #1 with
{\arrow{latex}}},postaction={decorate}}}
\filldraw[draw=white,fill=gray!20] (-0,-0.2) rectangle (1, 1.2);
\draw [line width =1.5pt,decoration={markings, mark=at position 1 with {\arrow{>}}},postaction={decorate}](1,-0.2)--(1,1.2);
\draw [line width =1pt,decoration={markings, mark=at position 0.5 with {\arrow{<}}},postaction={decorate}](0,0.8)--(1,0.8);
\draw [line width =1pt,decoration={markings, mark=at position 0.5 with {\arrow{>}}},postaction={decorate}](0,0.2)--(1,0.2);
\node [right]  at(1,0.8) {$j$};
\node [right]  at(1,0.2) {$i$};
\end{tikzpicture}
}=\raisebox{-.20in}{
\begin{tikzpicture}
\tikzset{->-/.style=
{decoration={markings,mark=at position #1 with
{\arrow{latex}}},postaction={decorate}}}
\filldraw[draw=white,fill=gray!20] (-0,-0.2) rectangle (1, 1.2);
\draw [line width =1.5pt,decoration={markings, mark=at position 1 with {\arrow{>}}},postaction={decorate}](1,1.2)--(1,-0.2);
\draw [line width =1pt,decoration={markings, mark=at position 0.5 with {\arrow{<}}},postaction={decorate}](0,0.8)--(1,0.8);
\draw [line width =1pt,decoration={markings, mark=at position 0.5 with {\arrow{>}}},postaction={decorate}](0,0.2)--(1,0.2);
\node [right]  at(1,0.8) {$j$};
\node [right]  at(1,0.2) {$i$};
\end{tikzpicture}
}$$
 when $v=1$, which can be eaily derived from Proposition \ref{prop3.1}.

\end{proof}

\subsection{Relative spin structure and character variety}\label{sss3.3}
Let $(M,\mathcal{N})$ be any marked 3-manifold, and $\zeta: UM\rightarrow M$ be the unit tangent bundle. We know the fiber of this bundle is $SO(3)$, whose fundamental group
is $\mathbb{Z}_{2}$. For any point $P\in M$, we use $\vartheta_P$ to denote the nontrivial element in the fundamental group of $\zeta^{-1}(P)$. We have $\vartheta_P$ is homotopic to $\vartheta_Q$ for any two points $P,Q$ in a same component of $M$. For a component $Y$ of $M$, we use $\vartheta_Y$ to denote this homotopy type. When $M$ is connected, we will use $\vartheta$ to denote this unique homotopy type.

For any component $e\in \N$, $e$ has  a unique lift $\tilde{e}$ in $UM$. For any point $x\in e$, let $u_x$ be the unit velocity vector at point $x$, let $w_x$ be the unit tangent vector at $x$ such that $w_x$
is orthogonal to $\partial M$ and pointing inside of $M$. Then the orientation of $M$ determines the second unit tangent vector $v_x$ such that $(u_x, v_x, w_x)$ is the orientation of $M$. Obviously 
$\tilde{e} =\{(x,u_x,v_x,w_x)\mid x\in e\}$ is a smooth path in $UM$, which is diffeomorphic to $(0,1)$. Let $\tilde{\N} =\cup_{e}\tilde{e}$ where the union takes over all components $e$ of $\N$. Note that $\tilde{\N}
=\emptyset$ when $\N=\emptyset$.

From now on, for a topological space $X$, we will use $Com(X)$ (respectively $PCom(X)$) to denote the set of components of $X$ (the set of path connected components of $X$).

\begin{definition}
A {\bf relative spin structure} of a marked 3-manifold $(M,\mathcal{N})$ is defined to be  a group homomorphism $h:H_1(UM,\tilde{\N})\rightarrow \mathbb{Z}_2$ such that $h(\vartheta_Y)= 1$ for $Y\in Com(M)$.
\end{definition}

Note that when $\N=\emptyset$, $h$ is just the usual spin structure.

Suppose $(M,\mathcal{N})$ is the disjoint union of $(M_1,\mathcal{N}_1)$ and $(M_2,\mathcal{N}_2)$. Then we have 
$H_1(UM,\tilde{\N}) = H_1(UM_1,\tilde{\N_1})\oplus H_1(UM_2,\tilde{\N_2})$. For each $i=1,2$, let $h_i$
be a relative spin structure for $(M_i,\N_i)$. Then $(h_1,h_2): H_1(UM,\tilde{\N})\rightarrow \mathbb{Z}_2$, defined by $(h_1,h_2)(x_1,x_2) = h_1(x_1)+h_2(x_2)$, is a relative spin structure for $(M,\mathcal{N})$. Clearly every relative spin structure of $(M,\mathcal{N})$ is of  this form.

\begin{rem}\label{rrrmmm}
Let $X$ be any  path  connected topological space, and $P$ be a set of finite points in $X$. We suppose 
$P = \{p_0,p_1,\dots,p_{m-1}\}$ where $m$ is a positive integer. For each $1\leq i\leq m-1$, let $\alpha_{i}$ be a path connecting $p_{0}$ and $p_{i}$. Then 
$H_1(X,P) = H_1(P)\oplus \mathbb{Z}([\alpha_1])\oplus\dots\oplus \mathbb{Z}([\alpha_{n-1}])$.

Suppose $M$ is connected.
When $\N$ has only one component, we have $H_1(UM,\tilde{\N}) = H_1(UM)$. Thus in this case the relative spin structure for $(M,\mathcal{N})$ is just the usual spin structure for $M$. When $\sharp \N>1$, suppose the set of components of $\N$ is
$ \{e_0,e_1,\dots, e_{m-1}\}$. For each $1\leq i\leq m-1$, let $\alpha_{i}$ be a path connecting $\widetilde{e_0}$ and $\widetilde{e_i}$. Then we have 
$H_1(UM,\tilde{\N}) = H_1(UM)\oplus \mathbb{Z}([\alpha_1])\oplus\dots\oplus \mathbb{Z}([\alpha_{m-1}])$.
For any spin structure $h$ for $M$, we can extend $h$ to a relative spin structure for $(M,\mathcal{N})$ by defining
$h([\alpha_i]) = r_i\in \mathbb{Z}_{2}, 1\leq i\leq m-1,$ where $r_i,1\leq i\leq m-1,$ are $m-1$ arbitrary elements in $\mathbb{Z}_2$. Reversely, for any relative spin structure, we can restrict $h$ to $H_1(UM)$ to obtain a spin structure for $M$.

Suppose $\N'$ is obtained from $\N$ ($\N\neq \emptyset$) by adding one extra marking $e$ such that $cl(e)\cap cl(\N) = \emptyset$. Let $\alpha$ be a path connecting $\tilde{\N}$ and $\tilde{e}$ such that
$\alpha(0)\in \tilde{\N}$ and $\alpha(1)\in\tilde{e}$. Then $H_1(UM,\tilde{\N'}) = H_1(UM,\tilde{\N})\oplus
\mathbb{Z}([\alpha])$.
Any relative spin structure $h$ for $(M,\mathcal{N})$ can be extended to a relative spin structure for $(M,\mathcal{N}')$ by defining $h([\alpha]) = r$ where $r$ is an arbitrary element in $\mathbb{Z}_2$.
Reversely any relative spin structure $h$  for $(M,\mathcal{N}')$ can be restricted to $H_1(UM,\tilde{\N})$ to obtain a relative spin structure for $(M,\mathcal{N})$.
\end{rem}

For a path connected topological space $X$, we use
  $\pi_1(X)$ to denote the fundamental group for $X$. For  $[\alpha],[\beta] \in \pi_1(X)$, $[\alpha][\beta] = [\alpha*\beta]$ where $\alpha*\beta$ is obtained by first going through $\beta$, then going through $\alpha$. Note that here $\alpha*\beta$ is different with conventional definition.

\begin{definition}\label{df3.4}
For any path connected toplogical space $X$,
define 
$$\Gamma_n(X) = \mathbb{C}[[\alpha]_{i,j}\mid [\alpha]\in \pi_1(X),1\leq i,j\leq n]/(Q_{[\alpha]}Q_{[\beta]}
= Q_{[\alpha *\beta]}, det(Q_{[\alpha]}) = 1, Q_{[o]}=I)$$
where $[\alpha],[\beta]$ go through all elements in $\pi_1(X)$, $[o]$ is the trivial loop in $\pi_1(X)$
, $Q_{[\eta]} = ([\eta]_{i,j})_{1\leq i,j\leq n}$ for any element $[\eta]\in \pi_1(X)$.

Note that $\pi_1(X)$  has an action on $\Gamma_n(X)$, defined by 
$[\alpha]([\beta]_{i,j}) = [\alpha*\beta*\alpha^{-1}]_{i,j}$ for any $[\alpha],[\beta],1\leq i,j\leq n$.
We use $G_n(X)$ to denote the subalgebra of $\Gamma_n(X)$ fixed by this action.
\end{definition}

Note that Trace($Q_{[\alpha]})\in G_n(X)$ for any $[\alpha]\in \pi_1(X)$, and
 $G_n(X)$ is generated by Trace$(Q_{[\alpha]})$, $[\alpha]\in \pi_1(X)$ as an algebra \cite{S2001SLn}. 

\begin{rem}
We can generalize Definition \ref{df3.4} to general topological space. Let $X$ be a topological space. Suppose $PCom(X) = \{X_1,\dots,X_m\}$, then define 
$$\Gamma_n(X) = \Gamma_n(X_1)\otimes \dots \otimes \Gamma_n(X_m),\;G_n(X) = G_n(X_1)\otimes \dots \otimes G_n(X_m).$$
But $\Gamma_n(X), G_n(X)$ are only well-defined up to isomorphism, since different  order of $X_i$ give different algebras.
\end{rem}

\begin{definition}[\cite{CL2022stated}]\label{bf1}
Let $X$ be a  topological space and $\{E_j\}_{j\in J}$ be disjoint contractible subspaces of $X$. The fundamental groupoid $\pi_1(X, \cup_{j\in J}E_j)$  is the groupoid (i.e. a category with invertible morphisms) whose objects are $\{E_j\}_{j\in J}$  and whose morphisms are the homotopy classes of oriented paths in $X$ with end points  in $\cup_{j\in J}E_j$. A morphism of groupoids is a functor of the corresponding categories. We can regard the group as the groupoid consisting of only one object and all group elements being all morphisms.
\end{definition}

For a marked 3-manifold $\MN$, we use $\pi_1^{Mor}\MN$ to denote the set of morphisms in $\pi_1\MN$.

Let $\mathcal{A},\mathcal{B}$ be two categories. We try to define a new category   $\mathcal{A}\cup \mathcal{B}$. The objects of $\mathcal{A}\cup \mathcal{B}$ is the union of objects in $\mathcal{A}$ and objects in $\mathcal{B}$.
Let $U,V$ be any two objects in $\mathcal{A}\cup \mathcal{B}$. If $U,V$ both belong to $\mathcal{A}$
(respectively $\mathcal{B}$), then we define Hom${}_{\mathcal{A}\cup \mathcal{B}}(U,V)$ to be
Hom${}_{\mathcal{A}}(U,V)$ (respectively Hom${}_{\mathcal{B}}(U,V)$). Otherwise we define Hom${}_{\mathcal{A}\cup \mathcal{B}}(U,V) = \emptyset$.

Let $X$ be a  topological space. Suppose $PCom(X) = \{X_1,\dots,X_m\}$. For each $1\leq t\leq m$, let $\{E_j\}_{j\in J_{t}}$ be disjoint contractible subspaces of $X_t$.
Obviously we have 
$$\pi_1(X,\cup_{1\leq t\leq m}(\cup_{j\in J}E_j))= \pi_1(X_1, \cup_{j\in J_1}E_j)\cup \dots\cup
\pi_1(X_m, \cup_{j\in J_m}E_j) .$$

\begin{definition}
For a marked 3-manifold $(M,\mathcal{N})$ with every component of $M$ containing at least one marking,
define $$\chi_n(M,\mathcal{N}) = Hom(\pi_1(M,\mathcal{N}), SL(n,\mathbb{C})),$$ and 
$$\tilde{\chi}_n(M,\mathcal{N}) = \{\tilde{\rho}\in Hom(\pi_1(UM,\tilde{\N}), SL(n,\mathbb{C}))\mid\tilde{\rho}
(\vartheta_Y) = d_n I,\;\text{for all }Y\in Com(M)\}$$ where $I$ is the identity matrix.
\end{definition}

Suppose $(M,\mathcal{N})$ is the disjoint union of $(M_1,\mathcal{N}_1)$ and $(M_2,\mathcal{N}_2)$, then we have 
$$\chi_n(M,\mathcal{N})\simeq \chi_n(M_1,\mathcal{N}_1)\times \chi_n(M_2,\mathcal{N}_2),\;\tilde{\chi}_n(M,\mathcal{N})
\simeq \tilde{\chi}_n(M_1,\mathcal{N}_1)\times \tilde{\chi}_n(M_2,\mathcal{N}_2).$$
From Lemma 8.1 in \cite{CL2022stated}, if $M$ is connected we have 
$$\chi_n(M,\mathcal{N})\simeq Hom(\pi_1(M), SL(n,\mathbb{C}))\times SL(n,\mathbb{C})^{\sharp \N-1}.$$
 Then $\chi_n(M,\mathcal{N})$  is an affine algebraic set, whose 
coordinate ring is denoted as $R_n(M,\mathcal{N})$.

\begin{definition}\label{rrr}
When $M$ is connected and $\N$ is empty, we define $\chi_n(M,\mathcal{N})$ to be the 
$SL(n,\mathbb{C})$-character variety of $M$.
That is $$\chi_n(M,\mathcal{N})=Hom(\pi_1(M), SL(n, \mathbb{C}))/\simeq$$ where $\rho\simeq \rho'\in
Hom(\pi_1(M), SL(n, \mathbb{C}))$ if and only if 
$$\text{Trace}\rho([\alpha])=
\text{Trace}\rho'([\alpha])$$ for all $[\alpha] \in Hom(\pi_1(M), SL(n, \mathbb{C}))$.

Similarly we define 
$$\tilde{\chi}_n(M,\emptyset)=\{\tilde{\rho}\in Hom(\pi_1(UM), SL(n, \mathbb{C}))\mid \tilde{\rho}(\vartheta)
= d_n I\}/\simeq$$
where the definition for $\simeq$ is the same as above (that is two elements are considered the same if and only if they have the same Trace).
\end{definition}

\begin{rem}\label{rem01}
From \cite{S2001SLn}, we know $\chi_n(M,\emptyset)$ is an affine algebraic set. We also denote its  coordinate ring 
as $R_n(M,\emptyset)$. 
 There is a
surjective algebra homomorphism $\cY: G_n(M)\rightarrow R_n(M,\emptyset)$ defined by
$$\cY(\text{Trace}(Q_{[\alpha]}))(\rho) = \text{Trace}(\rho([\alpha])) \;\text{where}
\; [\alpha]\in \pi_1(M), \rho\in \chi_n(M,\emptyset),$$
and Ker$\cY = \sqrt{0}$.
\end{rem}

We can simply generalize  definitions for $\chi_n(M,\mathcal{N}),\tilde{\chi}_n(M,\mathcal{N}), R_n(M,\mathcal{N})$
to all marked 3-manifolds by taking product (or tensor product) for disjoint union. 

\begin{proposition} \label{prop3.6}
For any marked 3-manifold  $(M,\mathcal{N})$, we have $\chi_n(M,\mathcal{N})\simeq \tilde{\chi}_n(M,\mathcal{N})$.
\end{proposition}
\begin{proof}
We can assume $M$ is connected.
Here we only consider the case when $\N\neq \emptyset$ since we can prove the case when $\N=\emptyset$
by using the same technique.

Let $h$ be a relative spin structure for $(M,\mathcal{N})$. We use $h$ to establish an isomorphism $f_h:
\chi_n(M,\mathcal{N})\rightarrow \tilde{\chi}_n(M,\mathcal{N})$. For any $\rho \in \chi_n(M,\mathcal{N})$, define 
$$f_h(\rho)(\tilde{\alpha})
= d_n^{h(\tilde{\alpha})}\rho(\alpha)\text{\;where\;}\tilde{\alpha} \in \pi_1(UM,\tilde{\N}),\;\alpha = \zeta(\tilde{\alpha}).$$ 
Clearly $f_h(\rho)$ is a homomorphism from $\pi_1(UM,\tilde{\N})$ to $SL(n,\mathbb{C})$, and $f_h(\rho)(\vartheta) = d_n^{h(\vartheta)}\rho(\zeta(\vartheta)) = d_n I$, thus $f_h\in \tilde{\chi}_n(M,\mathcal{N})$. 

Then we try to define 
$g_h : \tilde{\chi}_n(M,\mathcal{N}) \rightarrow \chi_n(M,\mathcal{N})$. For any $\tilde{\rho}\in \tilde{\chi}_n(M,\mathcal{N})$ and 
$\alpha\in \pi_1(M,\mathcal{N})$
$$g_h(\tilde{\rho})(\alpha) = d_n^{h(\tilde{\alpha})} \tilde{\rho}(\tilde{\alpha})
\;\text{where} \;\tilde{\alpha}\in \pi_1(UM, \tilde{\N})\;\text{such that}\; \zeta(\tilde{\alpha}) = \alpha.$$
Suppose $\zeta(\tilde{\alpha}) = \zeta(\tilde{\beta}) = \alpha$. Since 
$\tilde{\rho}(\vartheta) = d_n I$ and $h(\vartheta) = 1$, we have 
$$\tilde{\rho}(\tilde{\alpha}^{-1}\tilde{\beta}) = d_n^{h(\tilde{\alpha}^{-1}) +h(\tilde{\beta})}I
= d_n^{h(\tilde{\alpha})+ h(\tilde{\beta})}I.$$
Then $d_n^{h(\tilde{\beta})} \tilde{\rho}(\tilde{\beta}) = d_n^{h(\tilde{\alpha})} \tilde{\rho}(\tilde{\alpha})$, which shows $g_h(\tilde{\rho})$ is well-defined. Obviously we have $g_h(\tilde{\rho}) \in \chi(M,\mathcal{N})$.

We have 
$$f_h(g_h(\tilde{\rho}))(\tilde{\alpha}) = d_n^{h(\tilde{\alpha})} g_h(\tilde{\rho})(\alpha)
= d_n^{h(\tilde{\alpha})}d_n^{h(\tilde{\alpha})} \tilde{\rho}(\tilde{\alpha}) =\tilde{\rho}(\tilde{\alpha}),$$
$$g_h(f_h(\rho))(\alpha) = d_n^{h(\tilde{\alpha})}f_h(\rho)(\tilde{\alpha}) =
d_n^{h(\tilde{\alpha})}d_n^{h(\tilde{\alpha})}\rho(\alpha) = \rho(\alpha).$$
Thus $g_h$ and $f_h$ are inverse to each other.
\end{proof}

When there is a fixed relative spin structure, we don't have to distinguish between 
$\chi_n(M,\mathcal{N})$ and $\tilde{\chi}_n(M,\mathcal{N})$, and we also regard $R_n(M,\mathcal{N})$ as the coordinate ring 
for $\tilde{\chi}_n(M,\mathcal{N})$ using Proposition \ref{prop3.6}. 

\begin{rem}
For any two relative spin structures $h_1,h_2$, we have $F_{h_2 - h_1}\circ f_{h_1} = f_{h_2}$
where $F_{h_2 - h_1}$ is an isomorphism from $\tilde{\chi}_n(M,\mathcal{N})$ to $\tilde{\chi}_n(M,\mathcal{N})$ defined as
$$F_{h_2-h_1}(\tilde{\rho})(\tilde{\alpha}) = d_n^{h_2(\tilde{\alpha}) - h_1(\tilde{\alpha})}\tilde{\rho}(\tilde{\alpha}).$$
\end{rem}

\begin{rem}\label{rem3.14}
From Remark \ref{rem01}, we know there is a surjective algebra homomorphism 
$\cY: G_n(M) \rightarrow R_n(M,\emptyset)$, and Ker$\cY =\sqrt{ 0}$. When there is a spin structure $h$, we can regard $R_n(M,\emptyset)$ as the coordinate ring for 
$\tilde{\chi}_n(M,\emptyset)$ using Proposition \ref{prop3.6}. Then $\cY$ is given by
\begin{equation}\label{eee3.14}
\cY(\text{Trace}(Q_{[\alpha]}))(\tilde{\rho}) = d_n^{h([\tilde{\alpha}])} \text{Trace}( \tilde{\rho}([\tilde{\alpha}]))
\end{equation}
where $[\alpha]\in \pi_1(M)$ and $[\tilde{\alpha}]\in \pi_1(UM)$ is any lift of $[\alpha]$.
Since  the definition of $\cY$ in equation (\ref{eee3.14}) is related to $h$, we will use $\cY_h$, instead of $\cY$, to denote the map defined by equation (\ref{eee3.14}).
\end{rem}

\subsection{Surjective algebra homomorphism from $S_n(M,\mathcal{N},1)$ to the coordinate ring}\label{subb3.4}
For any marked 3-manifold $(M,\mathcal{N})$, we are trying to define a surjective algebra homomorphism $\Phi_h^{(M,\mathcal{N})} :S_n(M,\mathcal{N},1)\rightarrow R_n(M,\mathcal{N})$ (here we regard $R_n(M,\mathcal{N})$ as the coordinate ring 
for $\tilde{\chi}_n(M,\mathcal{N})$).

 Recall that $\tilde{\N}$ is lifted by $\N$. For any component $e\in \N$, we have 
$\tilde{e} = \{(x, u_x, v_x, w_x) \mid x\in e\}$, where $u_x$ is the unit velocity vector at $x$, $w_x$ is the unit tangent vector at $x$ orthogonal to $\partial M$ pointing into $M$ and  $(u_x, v_x, w_x)$ is the orientation of $M$.

For any $n$-web $l$ and a component $e$ of $\N$, we  can isotope $l$ such that the velocity vector of $l$ at each its end point $x$  contained in $e$ is parallel to $v_x$. Then we say $l$ is in {\bf good position with respect to $e$}. If $l$ is in good position with respect to every component of $\N$, we say it is in good position with respect to $(M,\mathcal{N})$, or just $l$ is in {\bf good position} when there is no confusion with $(M,\mathcal{N})$.

Let $\alpha$ be any stated framed oriented boundary  arc in $(M,\mathcal{N})$,  then we can lift $\alpha$ to an element $\tilde{\alpha}$ in $\pi_1(UM,\tilde{\N})$. We first isotope $\alpha$ such that $\alpha$ is in good position and the framing is normal everywhere. Then $\alpha$ lifts to $\tilde{\alpha}$  where
the first vector is the framing, the second vector is the velocity vector, and the third vector is determined by the orientation of $M$.  
Suppose $s(\alpha(0)) = j$ and $s(\alpha(1)) = i$. 
Then for any
$\tilde{\rho}\in\tilde{\chi}_n(M,\mathcal{N})$, define 
$$tr_{\alpha}(\tilde{\rho}) = [A\tilde{\rho}(\tilde{\alpha})]_{\overline{i},\overline{j}},\;\text{where}\; A_{i,j} = (-1)^{i+1}\delta_{\overline{i},j}, 1\leq i,j\leq n.$$
Note that det$ A =1$ and $A^2 = d_n I$.

Let $\alpha$ be any framed oriented  knot in $S_n(M,\mathcal{N})$, then we lift $\alpha$ to a closed path $\tilde{\alpha}$  in $UM$ as above (first isotope $\alpha$ such that the framing is normal everywhere, then use framing as the first vector and use velocity vector as the second vector). We use a
path to connect $\tilde{\N}$ (respectively the base point for $\pi_1(UM)$) and $\tilde{\alpha}$ when
$\N\neq \emptyset$ (respectively $\N=\emptyset$). This gives an element in $\pi_1(UM,\tilde{\N})$ or $\pi_1(UM)$, which is still denoted as $\tilde{\alpha}$. For any $\tilde{\rho}\in\tilde{\chi}_n(M,\mathcal{N})$ define 
$$tr_{\alpha}(\tilde{\rho}) =  \text{Trace}(\tilde{\rho}(\tilde{\alpha})).$$
%
Since Trace is invariant under the same conjugacy class, we have $tr_{\alpha}(\tilde{\rho})$ is well-defined.

\begin{theorem}\label{thm3.11}

Let $(M,\mathcal{N})$ be a marked 3-manifold with $M$ being connected. 
Then
there exists a surjective algebra homomorphism 
$\Phi^{(M,\mathcal{N})} : S_n(M,\mathcal{N},1)\rightarrow R_n(M,\mathcal{N})$ defined as following: For any stated $n$-web $l$ in $(M,\mathcal{N})$, we use relation (\ref{wzh.five}) to kill all the sinks and sources to obtain $l'$ consisting of acrs and knots if $\N\neq\emptyset$ (we use relation (\ref{wzh.four}) to kill all the sinks and sources to obtain $l'$ consisting of knots if $\N=\emptyset$). Suppose $l' = \cup_{\alpha}\alpha$ where each $\alpha$ is a stated framed oriented boundary arc or a framed oriented knot, then define
$$\Phi^{(M,\mathcal{N})}(l)(\tilde{\rho}) = \prod_{\alpha} tr_{\alpha}(\tilde{\rho})$$
where $\tilde{\rho}\in \tilde{\chi}_n(M,\mathcal{N})$.

\end{theorem}

Although we assume $M$ is connected in Theorem \ref{thm3.11} for simplicity, we can easily generalize  Theorem \ref{thm3.11} to general marked 3-manifolds. 

 When there is no confusion with the marked 3-manifold $(M,\mathcal{N})$, we can  omit the superscript for $\Phi^{(M,\mathcal{N})}$.
We will prove Theorem \ref{thm3.11} in next section. 

\subsection{Compatibility with the construction by Costantino and L{\^e} for essentially bordered pb
 surfaces when $n=2$.
}\label{newsec}

Let $\Sigma$ be an essentailly bordered pb surface, and $(M,\mathcal{N})$ be the thickening of $\Sigma$. 
Recall that for every boundary component $e$, we select a point $x_{e}\in e$, and set $\N=\cup_{e} (\{x_e\}\times (-1,1))$ where $e$ is taken over all the boundary components of $\Sigma$.
The orientation and the Riemannian matric of $M$ are the product orientation and the product Riemannian matric respectively. For simplicity, we can assume $\Sigma$ is connected.

If we regard the state $"2"$ as the state $"+"$, and the state $"1"$ as the state $"-"$. Then 
$S_2(\Sigma,1)$ is the same as the commutative stated skein algebra $\cS_{1}(\Sigma)$, mentioned in Section 8 in \cite{CL2022stated}, as shown in \cite{le2021stated}.
 When $n=2$, Costantino and L{\^e} established an isomorphism $tr : S_2(\Sigma,1)\rightarrow \chi(\Sigma)$ where $\chi(\Sigma)$ is the coordinate ring of the so called flat twisted $SL(2,\mathbb{C})$-bundle.
Here we briefly recall the definition of the  flat twisted $SL(2,\mathbb{C})$-bundle,
 please refer to Section 8 in \cite{CL2022stated} for more details. We follow their notation, and use $\pi_1(U\Sigma,\widetilde{\partial\Sigma})$ to denote their groupoid, where $U\Sigma$ is the  unit tangent bundle over
$\Sigma$ and $\widetilde{\partial\Sigma}$ is the lift of $\partial \Sigma$.
 Then every point in $U\Sigma$ is a pair $(x,v_x)$ where $x\in\Sigma$ and $v_x$ is a unit tangent vector at point $x$.
 For a point $x\in\Sigma$ the fiber $\mathbb{O}$ is a circle, and we orient it according to the  orientation of $\Sigma$. 
Then the flat twisted $SL(2,\mathbb{C})$-bundle is defined to be 
$$\{\rho\in Hom( \pi_1(U\Sigma,\widetilde{\partial\Sigma}), SL(2,\mathbb{C})  )\mid \rho(\mathbb{O}) = -I\}.$$

Let $pr: M\rightarrow \Sigma$ be the projection given by $pr(x,t) = x$ for all $x\in \Sigma,t\in [-1,1]$, and $l:\Sigma \rightarrow M$ be the embedding  given by $l(x) = (x,0)$.
We use $pr_*$ (respectievly $l_*$) to denote the induced map from 
$T(M)$ to $ T(\Sigma)$ (respectively from $ T(\Sigma)$ to $T(M)$), where $T(\cdot)$ is the tangent bundle. 
For every point $y\in M$, we use $u_y$ to denote the unit vertical tangent vector at $y$ such that $u_y$ points from $-1$ to $1$. Then $l$ induces an embedding 
\begin{align*}
l_{\sharp}: U\Sigma&\rightarrow UM\\
(x,v_x) &\mapsto (l(x), u_{l(x)}, l_*(v_x), w_{l(x)})
\end{align*}
where  $w_{l(x)}$ is determined by the orientation of $M$. Let $VM = \{(y,a_y,b_y,c_y)\in UM\mid a_y = u_y\}$ be a submanifold of $UM$ with one dimension less. Then $pr$ induces a projection
\begin{align*}
 pr_{\sharp}: VM&\rightarrow \Sigma\\
(y, a_y, b_y, c_y)&\mapsto (pr(y), pr_*(b_y)).
\end{align*}
Clearly Im$l_{\sharp}\subset VM$, and $pr_{\sharp}\circ l_{\sharp} = Id_{U\Sigma}$.

 For each boundary component $e$ of $\Sigma$, we use $\widetilde{x_e}$ to denote a point in $\tilde{e}$ whose projection on $e$ is the point $x_e$.

Define
\begin{align*}
f_{l} :\pi_1(U\Sigma,\widetilde{\partial \Sigma}) &\rightarrow \pi_1(UM,\tilde{\N})\\
[\alpha]&\mapsto [l_{\sharp}\circ \alpha]
\end{align*}
where $\alpha$ is a representative of $[\alpha]$ such that the two endpoints of $\alpha$ belong to
$\cup_{e}\{\widetilde{x_e}\}$.
And define 
\begin{align*}
f_{pr}:\pi_1(UM,\tilde{\N})&\rightarrow \pi_1(\Sigma,\widetilde{\partial\Sigma})\\
[\beta]&\mapsto [pr_{\sharp}\circ \beta]
\end{align*}
where $\beta$ is a representative of $[\beta]$ such that Im$\beta\subset VM$. It is easy to show $f_{l}$ and 
$f_{pr}$ are inverse to each other, and 
 $f_{l}(\mathbb{O})= \vartheta$. 
Then $f_l$ induces isomorphism from $\tilde{\chi}_2(M,\mathcal{N})$ to $\{\rho\in Hom( \pi_1(U\Sigma,\widetilde{\partial\Sigma}), SL(2,\mathbb{C})  )\mid \rho(\mathbb{O}) = -I\}$, which further induces an isomorphism $f_* :\chi(\Sigma)\rightarrow R_2(M,\mathcal{N})$. Then it is a trivial check that 
$$f_* \circ tr = \Phi .$$

\subsection{Compatibility with the splitting map}\label{sub3.4}
In this subsection we discuss the splitting maps for both $R_n(M,\mathcal{N})$  and $S_n(M,\mathcal{N},1)$ and the commutativity for these two splitting maps. 

Recall that when $D$ is a properly embedded disk in $M$ and $\beta\subset D$ is an embedded oriented open interval, there exists a splitting map $\Theta_{(D,\beta)}:S_n(M,\mathcal{N},1)\rightarrow S_n(\text{Cut}_{(D,\beta)}(M,\mathcal{N}),1).$

\begin{lemma}
The above linear map $\Theta_{(D,\beta)}$ is an algebra homomorphism.
\end{lemma}
\begin{proof}
It followes easily from the definition of $\Theta_{(D,\beta)}$.
\end{proof}

We will use $(M', \N')$ to denote $\text{Cut}_{(D,\beta)}(M,\mathcal{N})$. Then there is a projection
pr $:(M', \N')\rightarrow (M,\mathcal{N}).$
If we  orient $\partial D$, the orientations of $\partial D$ and $M$ give a way to distinguish between $D_1$ and $D_2$ such that the orientation pointing from $D_2$ to $D_1$ and the orientation of $\partial D$ coincide with the orientaion of $M$, see Figure \ref{fig:1}.
\begin{figure}[!h]
\centering
\includegraphics[scale=0.6]{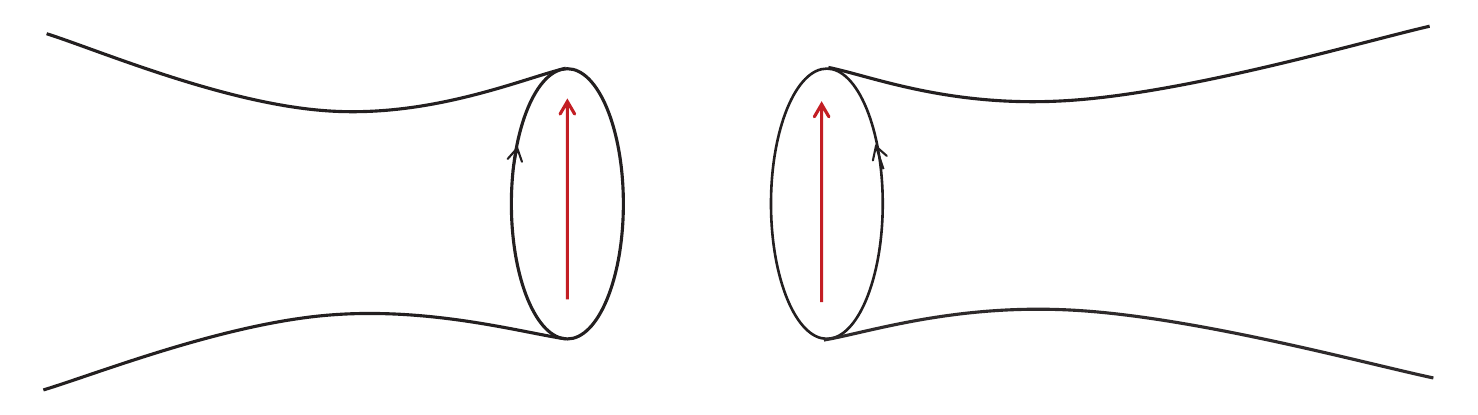}
\caption{The orientation of $D$ is indicated by the arrow, the orientation of $M$ is right handed.
The left (respectively right) disk copy is $D_1$ (respectively $D_2$). The left (respectively right) red arrow is $\beta_1$ (respectively $\beta_2$).}
\label{fig:1}
\end{figure}


In the following discussion, we fix an orientation for $\partial D$. Note that $\beta\in D$ lifts to an element
in $UM$. For every point $P$ in $\beta$, the velocity vector gives the first unit tangent vector, the orientation of $\partial D$ and the orientation of $M$ give the third unit tangent vector (since the orientation of $\partial D$ and the orientation of $M$ give an orientation of $D$, which determines a unit tangent vector at $P$), then the orientation of $M$ gives the second unit tangent vector. We use $\tilde{\beta}$ to denote this lift. Recall that $\zeta : UM\rightarrow M$ is the projection, we define $\tilde{D} = \zeta^{-1}(D)$. The projection pr $:(M', \N')\rightarrow (M,\mathcal{N})$ induces a projection $\tilde{\text{pr}}: UM'\rightarrow UM$. Then 
$\tilde{\text{pr}}^{-1}(\tilde{\beta}) = \tilde{\beta_1}'\cup \tilde{\beta_2}'$
where $\tilde{\beta_1}' \in \widetilde{D_1}, \tilde{\beta_2}' \in \widetilde{D_2}$ ($ \widetilde{D_1}$
and $ \widetilde{D_2}$ are defined in the same way with $\tilde{D}$).  Note that 
$\widetilde{\beta_1} =  \tilde{\beta_1}'$ and $\widetilde{\beta_2} \neq  \tilde{\beta_2}'$. The orientation of $\beta$ determines a path $a_{\beta}$ from $\widetilde{\beta_2}$ to $\tilde{\beta_2}'$ and a path $b_{\beta}$
from $\tilde{\beta_2}'$ to $\widetilde{\beta_2}$ such that both $a_{\beta} *  b_{\beta}$ and 
$b_{\beta}* a_{\beta}$ are in the same homotopy type with $\vartheta$.

\begin{rem}
According to Lemma 8.1 in \cite{CL2022stated},  any $\rho'\in \tilde{\chi}_n(M',\N')$ can be extended to a homomorpism $\rho'':\pi_1(UM',\tilde{\N'}\cup \{\tilde{\beta_2}'\})\rightarrow SL(n,\mathbb{C})$ by setting 
$\rho''(a_{\beta}) = d_n A$.

For any $\alpha\in \pi_1(UM,\tilde{\N})$, we can isotope $\alpha$ such that $\alpha \cap \tilde{D}
= \alpha\cap \tilde{\beta}$ and $\alpha\cap \tilde{\beta}$ consists of finite points. Then 
$\alpha = \alpha_k* \alpha_{k-1}*\dots*\alpha_1$ where each
$\alpha_i\in \pi_1(UM, \tilde{\N}\cup\{\tilde{\beta}\})$  intersects $\tilde{D}$ at most in its endpoints and exactly along $\tilde{\beta}$. For any $\rho'\in \tilde{\chi}_n(M',\N')$, define
 $$\nu^{*}(\rho')(\alpha) = \rho''(\alpha_k')\rho''(\alpha_{k-1}')\dots
\rho''(\alpha_1')$$
where $\alpha_i' = \tilde{\text{pr}}^{-1}(\alpha_i)\in \pi_1(UM', \tilde{\N'}\cup \{\tilde{\beta_2}'\})$.
\end{rem}


\begin{proposition}
Let $(M,\mathcal{N})$ be a marked 3-manifold, $D$ be any properly embedded disk with an embedded  oriented open interval $\beta\in D$. Then there is a surjective homomorphism $\nu^{*}:\tilde{\chi}_n(M',\N')\rightarrow 
\tilde{\chi}_n(M,\mathcal{N})$ where $(M', \N') =$ Cut$_{(D,\beta)}(M, \N )$. Especially $\nu^{*}$ induces
an injective algebra homomorphism $\nu: R_n(M,\mathcal{N})\rightarrow R_n(M',\N')$.
\end{proposition}
\begin{proof}
First we show $\nu^{*}(\rho')\in \tilde{\chi}_n(M,\mathcal{N})$. Clearly we have $ \rho''( \vartheta_Y) = d_n I$ for  any component 
$Y$ of $M'$.
Since $\rho''$ preserves height exchange and crossing exchange, to show $\nu^{*}(\rho')$ is well-defined,  it suffices to show 
$\nu^{*}(\rho')$ preserves the following two moves:
\begin{equation}\label{e111}
\raisebox{-.30in}{
\begin{tikzpicture}
\tikzset{->-/.style=
{decoration={markings,mark=at position #1 with
{\arrow{latex}}},postaction={decorate}}}
\draw [line width =1pt,color=red,decoration={markings, mark=at position 1 with {\arrow{>}}},postaction={decorate}](0.8,-0.5)--(0.8,1.5);
\draw[line width =1pt] (0,0) arc (-90:90:0.5);
\draw [line width =1pt](-1,0)--(0,0);
\draw [line width =1pt](-1,1)--(0,1);
\filldraw[draw=black,fill=white] (-0.5,0) circle (0.07);
\node [right] at(0.9,1.3) {$\beta$};
\end{tikzpicture}
}
\longleftrightarrow
\raisebox{-.30in}{
\begin{tikzpicture}
\tikzset{->-/.style=
{decoration={markings,mark=at position #1 with
{\arrow{latex}}},postaction={decorate}}}
\draw [line width =1pt,color=red,decoration={markings, mark=at position 1 with {\arrow{>}}},postaction={decorate}](0,-0.5)--(0,1.5);
\draw[line width =1pt] (0,0) arc (-90:90:0.5);
\draw [line width =1pt](-1,0)--(0,0);
\draw [line width =1pt](-1,1)--(0,1);
\filldraw[draw=black,fill=white] (-0.5,0) circle (0.07);
\node [right] at(0.1,1.3) {$\beta$};
\end{tikzpicture}
}
\end{equation}
\begin{equation}\label{e222}
\raisebox{-.30in}{
\begin{tikzpicture}
\tikzset{->-/.style=
{decoration={markings,mark=at position #1 with
{\arrow{latex}}},postaction={decorate}}}
\draw [line width =1pt,color=red,decoration={markings, mark=at position 1 with {\arrow{>}}},postaction={decorate}](-0.8,-0.5)--(-0.8,1.5);
\draw[line width =1pt] (0,1) arc (90:270:0.5);
\draw [line width =1pt](1,0)--(0,0);
\draw [line width =1pt](1,1)--(0,1);
\filldraw[draw=black,fill=white] (0.5,0) circle (0.07);
\node [left] at(-0.9,1.3) {$\beta$};
\end{tikzpicture}
}
\longleftrightarrow
\raisebox{-.30in}{
\begin{tikzpicture}
\tikzset{->-/.style=
{decoration={markings,mark=at position #1 with
{\arrow{latex}}},postaction={decorate}}}
\draw [line width =1pt,color=red,decoration={markings, mark=at position 1 with {\arrow{>}}},postaction={decorate}](-0,-0.5)--(-0,1.5);
\draw[line width =1pt] (0,1) arc (90:270:0.5);
\draw [line width =1pt](1,0)--(0,0);
\draw [line width =1pt](1,1)--(0,1);
\filldraw[draw=black,fill=white] (0.5,0) circle (0.07);
\node [left] at(-0.1,1.3) {$\beta$};
\end{tikzpicture}
}
\end{equation}
 The red arrow in equations (\ref{e111})  and  (\ref{e222}) is the projection of $\beta$, to get the original  $\beta$, we just rotate the
red arrow in equations (\ref{e111})  and  (\ref{e222}) 90 degrees such that it points towards readers. The black line represents part of path in $UM$, and the white dot represents the direction of the path. The first unit tangent vector of the path is the one pointing towards readers  and the second one is given by the velocity vector of the black line.

Here we only prove $\nu^{*}(\rho')$ preserves
$$
\raisebox{-.30in}{
\begin{tikzpicture}
\tikzset{->-/.style=
{decoration={markings,mark=at position #1 with
{\arrow{latex}}},postaction={decorate}}}
\draw [line width =1pt,color=red,decoration={markings, mark=at position 1 with {\arrow{>}}},postaction={decorate}](0.8,-0.5)--(0.8,1.5);
\draw[line width =1pt] (0,0) arc (-90:90:0.5);
\draw [line width =1pt,decoration={markings, mark=at position 0.5 with {\arrow{>}}},postaction={decorate}](-1,0)--(0,0);
\draw [line width =1pt](-1,1)--(0,1);
\node [right] at(0.9,1.3) {$\beta$};
\end{tikzpicture}
}
\longleftrightarrow
\raisebox{-.30in}{
\begin{tikzpicture}
\tikzset{->-/.style=
{decoration={markings,mark=at position #1 with
{\arrow{latex}}},postaction={decorate}}}
\draw [line width =1pt,color=red,decoration={markings, mark=at position 1 with {\arrow{>}}},postaction={decorate}](0,-0.5)--(0,1.5);
\draw[line width =1pt] (0,0) arc (-90:90:0.5);
\draw [line width =1pt,decoration={markings, mark=at position 0.5 with {\arrow{>}}},postaction={decorate}](-1,0)--(0,0);
\draw [line width =1pt](-1,1)--(0,1);
\node [right] at(0.1,1.3) {$\beta$};
\end{tikzpicture}
}.
$$
The same proving technique applies for other three cases.

We isotope 
$\raisebox{-.30in}{
\begin{tikzpicture}
\tikzset{->-/.style=
{decoration={markings,mark=at position #1 with
{\arrow{latex}}},postaction={decorate}}}
\draw [line width =1pt,color=red,decoration={markings, mark=at position 1 with {\arrow{>}}},postaction={decorate}](0,-0.5)--(0,1.5);
\draw[line width =1pt] (0,0) arc (-90:90:0.5);
\draw [line width =1pt,decoration={markings, mark=at position 0.5 with {\arrow{>}}},postaction={decorate}](-1,0)--(0,0);
\draw [line width =1pt](-1,1)--(0,1);
\node [right] at(0.1,1.3) {$\beta$};
\end{tikzpicture}
}$ to 
$\raisebox{-.30in}{
\begin{tikzpicture}
\tikzset{->-/.style=
{decoration={markings,mark=at position #1 with
{\arrow{latex}}},postaction={decorate}}}
%
\draw [line width =1pt,color=red,decoration={markings, mark=at position 1 with {\arrow{>}}},postaction={decorate}](0.2,-0.5)--(0.2,1.5);
\draw [line width =1pt,decoration={markings, mark=at position 0.5 with {\arrow{>}}},postaction={decorate}](-1,0)--(0,0);
\draw [line width =1pt](-1,1)--(0,1);
\draw[line width =1pt] (0,0) arc (-90:0:0.2);
\draw[line width =1pt] (0.2,0.8) arc (0:90:0.2);
\draw[line width =1pt] (0.6,0.2) arc (0:180:0.2);
\draw[line width =1pt] (0.2,0.8) arc (-180:0:0.2);
\draw[line width =1pt] (0.6,0.2) arc (-180:0:0.3);
\draw[line width =1pt] (1.2,0.8) arc (0:180:0.3);
\draw[line width =1pt] (1.2,0.2)--(1.2,0.8);
\node [right] at(0.3,1.3) {$\beta$};
\end{tikzpicture}
}.$ 
From the definition of $\nu^{*}(\rho')$, we know
\begin{align*}
\nu^{*}(\rho')(\raisebox{-.30in}{
\begin{tikzpicture}
\tikzset{->-/.style=
{decoration={markings,mark=at position #1 with
{\arrow{latex}}},postaction={decorate}}}
\draw [line width =1pt,color=red,decoration={markings, mark=at position 1 with {\arrow{>}}},postaction={decorate}](0,-0.5)--(0,1.5);
\draw[line width =1pt] (0,0) arc (-90:90:0.5);
\draw [line width =1pt,decoration={markings, mark=at position 0.5 with {\arrow{>}}},postaction={decorate}](-1,0)--(0,0);
\draw [line width =1pt](-1,1)--(0,1);
\node [right] at(0.1,1.3) {$\beta$};
\end{tikzpicture}
}) =& \rho''(\raisebox{-.30in}{
\begin{tikzpicture}
\tikzset{->-/.style=
{decoration={markings,mark=at position #1 with
{\arrow{latex}}},postaction={decorate}}}
%
\draw [line width =1pt,color=red,decoration={markings, mark=at position 1 with {\arrow{>}}},postaction={decorate}](0.2,-0.5)--(0.2,1.5);
\draw [line width =1pt,decoration={markings, mark=at position 0.5 with {\arrow{<}}},postaction={decorate}](-1,1)--(0,1);
\draw[line width =1pt] (0.2,0.8) arc (0:90:0.2);
\node [right] at(0.3,1.3) {$\beta$};
\end{tikzpicture}
})
\rho''(\raisebox{-.30in}{
\begin{tikzpicture}
\tikzset{->-/.style=
{decoration={markings,mark=at position #1 with
{\arrow{latex}}},postaction={decorate}}}
%
\draw [line width =1pt,color=red,decoration={markings, mark=at position 1 with {\arrow{>}}},postaction={decorate}](0.2,-0.5)--(0.2,1.5);
\draw[line width =1pt] (0.6,0.2) arc (0:180:0.2);
\draw[line width =1pt] (0.2,0.8) arc (-180:0:0.2);
\draw[line width =1pt] (0.6,0.2) arc (-180:0:0.3);
\draw[line width =1pt] (1.2,0.8) arc (0:180:0.3);
\draw[line width =1pt,decoration={markings, mark=at position 0.5 with {\arrow{>}}},postaction={decorate}] (1.2,0.2)--(1.2,0.8);
\node [right] at(0.3,1.3) {$\beta$};
\end{tikzpicture}
})
\rho''(\raisebox{-.30in}{
\begin{tikzpicture}
\tikzset{->-/.style=
{decoration={markings,mark=at position #1 with
{\arrow{latex}}},postaction={decorate}}}
%
\draw [line width =1pt,color=red,decoration={markings, mark=at position 1 with {\arrow{>}}},postaction={decorate}](0.2,-0.5)--(0.2,1.5);
\draw [line width =1pt,decoration={markings, mark=at position 0.5 with {\arrow{>}}},postaction={decorate}](-1,0)--(0,0);
\draw[line width =1pt] (0,0) arc (-90:0:0.2);
\node [right] at(0.3,1.3) {$\beta$};
\end{tikzpicture}
})
\\
 =& \rho''(\raisebox{-.30in}{
\begin{tikzpicture}
\tikzset{->-/.style=
{decoration={markings,mark=at position #1 with
{\arrow{latex}}},postaction={decorate}}}
%
\draw [line width =1pt,color=red,decoration={markings, mark=at position 1 with {\arrow{>}}},postaction={decorate}](0.2,-0.5)--(0.2,1.5);
\draw [line width =1pt,decoration={markings, mark=at position 0.5 with {\arrow{<}}},postaction={decorate}](-1,1)--(0,1);
\draw[line width =1pt] (0.2,0.8) arc (0:90:0.2);
\node [right] at(0.3,1.3) {$\beta$};
\end{tikzpicture}
})
\rho''(\raisebox{-.30in}{
\begin{tikzpicture}
\tikzset{->-/.style=
{decoration={markings,mark=at position #1 with
{\arrow{latex}}},postaction={decorate}}}
%
\draw [line width =1pt,color=red,decoration={markings, mark=at position 1 with {\arrow{>}}},postaction={decorate}](0.2,-0.5)--(0.2,1.5);
\draw [line width =1pt,decoration={markings, mark=at position 0.5 with {\arrow{>}}},postaction={decorate}](-1,0)--(0,0);
\draw[line width =1pt] (0,0) arc (-90:0:0.2);
\node [right] at(0.3,1.3) {$\beta$};
\end{tikzpicture}
}) =
\rho''(\raisebox{-.30in}{
\begin{tikzpicture}
\tikzset{->-/.style=
{decoration={markings,mark=at position #1 with
{\arrow{latex}}},postaction={decorate}}}
\draw [line width =1pt,color=red,decoration={markings, mark=at position 1 with {\arrow{>}}},postaction={decorate}](0.8,-0.5)--(0.8,1.5);
\draw[line width =1pt] (0,0) arc (-90:90:0.5);
\draw [line width =1pt,decoration={markings, mark=at position 0.5 with {\arrow{>}}},postaction={decorate}](-1,0)--(0,0);
\draw [line width =1pt](-1,1)--(0,1);
\node [right] at(0.9,1.3) {$\beta$};
\end{tikzpicture}
})
\\
=& \nu^{*}(\rho')(\raisebox{-.30in}{
\begin{tikzpicture}
\tikzset{->-/.style=
{decoration={markings,mark=at position #1 with
{\arrow{latex}}},postaction={decorate}}}
\draw [line width =1pt,color=red,decoration={markings, mark=at position 1 with {\arrow{>}}},postaction={decorate}](0.8,-0.5)--(0.8,1.5);
\draw[line width =1pt] (0,0) arc (-90:90:0.5);
\draw [line width =1pt,decoration={markings, mark=at position 0.5 with {\arrow{>}}},postaction={decorate}](-1,0)--(0,0);
\draw [line width =1pt](-1,1)--(0,1);
\node [right] at(0.9,1.3) {$\beta$};
\end{tikzpicture}
}).
\end{align*}
Then, trivially we have $\nu^{*}(\rho')\in \tilde{\chi}_n(M,\mathcal{N})$.

Then we want to show $\nu^{*}$ is surjective. We use $-\tilde{\beta}$ to denote $\tilde{\text{pr}}(\widetilde{\beta_2})$, and use
$\overline{a_{\beta}}$ to denote $\tilde{\text{pr}}(a_{\beta})$. Then $\overline{a_{\beta}}$ is a path from $-\tilde{\beta}$ to
$\tilde{\beta}$. For any $\rho\in \tilde{\chi}_n(M,\mathcal{N})$, we use Lemma 8.1 in \cite{CL2022stated} to extend $\rho$ to
$\rho'':\pi_1(UM,\tilde{\N}\cup\{\tilde{\beta},-\tilde{\beta}\})$ setting in particular $\rho''(\overline{a_{\beta}})
= d_n A$. The projection $\tilde{\text{pr}}:UM'\rightarrow UM$ induces a homomorphism $\text{pr}_* :
\pi_1(UM',\tilde{\N'}\cup \{\tilde{\beta_2}'\}) \rightarrow \pi_1(UM,\tilde{\N}\cup\{\tilde{\beta},-\tilde{\beta}\})$. Then 
$\rho''\circ \text{pr}_*$ is a homomorphism from $\pi_1(UM',\tilde{\N'}\cup \{\tilde{\beta_2}'\})$ to $SL(n,\mathbb{C})$. Set $\rho' $ to be the restriction of $\rho''\circ \text{pr}_*$ on $\pi_1(UM',\tilde{\N'})$. Then it is easy to show we have
$\rho'\in \tilde{\chi}_n(M',\N')$ and $\nu^{*}(\rho') = \rho$.

\end{proof}

\begin{theorem}\label{the_split}
Let $(M,\mathcal{N})$ be any marked 3-manifold, $D$ be any properly embedded disk with an embedded  oriented open interval $\beta\in D$. Then we have 
$$\Phi^{(M',\N')}\circ \Theta_{(D,\beta)} = \nu\circ \Phi^{(M,\mathcal{N})}.$$
\end{theorem}
\begin{proof}
We can assume $M$ is connected. Note that $M'$ may not be connected.

Since both $\Phi^{(M',\N')}\circ \Theta_{(D,\beta)}$ and $\nu\circ \Phi^{(M,\mathcal{N})}$ are algebra homomorphisms, it suffices to show $\Phi^{(M',\N')}(\Theta_{(D,\beta)} (\alpha))= \nu( \Phi^{(M,\mathcal{N})}(\alpha))$ for any framed oriented knot or stated framed oriented boundary arc $\alpha$. If there is no intersection between $\alpha$ and $D$, it is obvious. Then we look at the case when $\alpha$ intersects $D$. We isotope $\alpha$ such that $\alpha$ is transverse to $D$, $\alpha\cap D\subset \beta$, and the framing at each point of $\alpha\cap \beta$ is the velocity vector of $\beta$.

If $\alpha$ is a stated framed oriented boundary arc with $s(\alpha(0)) = i$ and $s(\alpha(1)) = j$. Then $\alpha = \alpha_m * \alpha_{m-1}* \dots * \alpha_1$ where all $\alpha_i$ are framed oriented arcs and are parts of $\alpha$ such that each $\alpha_t$ has two ends on $\beta$ for $2\leq t\leq m-1$ and
$\alpha_1(1), \alpha_m(0)\in\beta$ and the interior of each $\alpha_t$ has no intersection with $D$. Let $\alpha_t' = \text{pr}^{-1}(\alpha_t), 1\leq t\leq m$, then
$$\Theta_{(D,\beta)}(\alpha) = \sum_{1\leq k_1,\dots,k_{m-1}\leq n} (\alpha_m')_{j,k_{m-1}}
(\alpha_{m-1}')_{k_{m-1},k_{m-2}}\dots (\alpha_1')_{k_1,i}.$$

For any $\rho'\in \tilde{\chi}_n(M',\N')$, we have 
\begin{equation*}
\begin{split}
\Phi^{(M',\N')}(\Theta_{(D,\beta)} (\alpha))(\rho')
&= \sum_{1\leq k_1,\dots,k_{m-1}\leq n} (A\rho'(\widetilde{\alpha'_{m}}))_{\overline{j},\overline{k_{m-1}}}
 (A\rho'(\widetilde{\alpha'_{m-1}}))_{\overline{k_{m-1}},\overline{k_{m-2}}}\dots(A\rho'(\widetilde{\alpha'_{1}}))_{\overline{k_1},\overline{i}}\\
&=(A\rho'(\widetilde{\alpha'_{m}})A\rho'(\widetilde{\alpha'_{m-1}})\dots A\rho'(\widetilde{\alpha'_{1}}))_{\overline{j},\overline{i}}
\end{split}
\end{equation*}
and
\begin{equation*}
\nu( \Phi^{(M,\mathcal{N})}(\alpha))(\rho') =  \Phi^{(M,\mathcal{N})}(\alpha)(\nu^{*}(\rho')) =(A\;\nu^{*}(\rho')(\tilde{\alpha}))_{\overline{j},\overline{i}}.
\end{equation*}
According to the definition of $\nu^{*}$, we know 
$$\nu^{*}(\rho')(\tilde{\alpha}) = \rho''(\tilde{\alpha}_m')\rho''(\tilde{\alpha}_{k-1}')\dots
\rho''(\tilde{\alpha}_1').$$
It is easy to see 
$$\rho''(\tilde{\alpha}_m')\rho''(\tilde{\alpha}_{k-1}')\dots
\rho''(\tilde{\alpha}_1')
=\rho''(\widetilde{\alpha'_{m}}) \rho''(\gamma_{m-1}) \rho''(\widetilde{\alpha'_{m-1}})
\rho''(\gamma_{m-2})\dots \rho''(\gamma_1) \rho''(\widetilde{\alpha'_{1}})$$
where $\gamma_t = a_{\beta}^{-1}$ or $b_{\beta}^{-1}$ for $1\leq t\leq m-1$. Since 
$\rho''( a_{\beta}^{-1}) =  \rho''( b_{\beta}^{-1}) = A$, we have 
\begin{align*}
\nu^{*}(\rho')(\tilde{\alpha}) &= \rho''( \widetilde{\alpha'_{m}}) \rho''(\gamma_{m-1}) \rho''(\widetilde{\alpha'_{m-1}})
\rho''(\gamma_{m-2})\dots \rho''(\gamma_1) \rho''(\widetilde{\alpha'_{1}})\\
&=  \rho'( \widetilde{\alpha'_{m}}) A \rho'(\widetilde{\alpha'_{m-1}})
A\dots A\rho'(\widetilde{\alpha'_{1}}).
\end{align*}
Thus we get
$$\nu( \Phi^{(M,\mathcal{N})}(\alpha))(\rho') =(A\rho'(\widetilde{\alpha'_{m}})A\rho'(\widetilde{\alpha'_{m-1}})\dots A\rho'(\widetilde{\alpha'_{1}}))_{\overline{j},\overline{i}} = \Phi^{(M',\N')}(\Theta_{(D,\beta)} (\alpha))(\rho').$$

If $\alpha$ is a framed knot. Let $\eta$ be a path in $UM$ connecting the base point of $\pi_1(UM)$
to $\tilde{\alpha}$ when $\N=\emptyset$ or a path connecting $\tilde{\N}$ to $\tilde{\alpha}$ when
$\N\neq \emptyset$ such that $\eta\cap \tilde{D} = \emptyset$. Similarly suppose  $\alpha = \alpha_m * \alpha_{m-1}* \dots * \alpha_1$ where all $\alpha_i$ are framed oriented arcs and parts of $\alpha$ such that each $\alpha_t$ has two ends on $\beta$ and does not intersect with $D$ on its interior. 
Let $\alpha_t' = \text{pr}^{-1}(\alpha_t), 1\leq t\leq m$, then
$$\Theta_{(D,\beta)}(\alpha) = \sum_{1\leq i,k_1,\dots,k_{m-1}\leq n} (\alpha_m')_{i,k_{m-1}}
(\alpha_{m-1}')_{k_{m-1},k_{m-2}}\dots (\alpha_1')_{k_1,i}.$$

For any $\rho'\in \tilde{\chi}_n(M',\N')$, we have 
\begin{equation*}
\begin{split}
\Phi^{(M',\N')}(\Theta_{(D,\beta)} (\alpha))(\rho')
&= \sum_{1\leq i,k_1,\dots,k_{m-1}\leq n} (A\rho'(\widetilde{\alpha'_{m}}))_{\overline{i},\overline{k_{m-1}}}
 (A\rho'(\widetilde{\alpha'_{m-1}}))_{\overline{k_{m-1}},\overline{k_{m-2}}}\dots(A\rho'(\widetilde{\alpha'_{1}}))_{\overline{k_1},\overline{i}}\\
&=\text{Trace}(A\rho'(\widetilde{\alpha'_{m}})A\rho'(\widetilde{\alpha'_{m-1}})\dots A\rho'(\widetilde{\alpha'_{1}}))
\end{split}
\end{equation*}
and
\begin{equation*}
\nu( \Phi^{(M,\mathcal{N})}(\alpha))(\rho') =  \Phi^{(M,\mathcal{N})}(\alpha)(\nu^{*}(\rho')) =\text{Trace}(\nu^{*}(\rho')(\eta^{-1}*\tilde{\alpha} *\eta)).
\end{equation*}
We assume $\eta(1)\in \widetilde{\alpha'_{1}}$, otherwise we can relabel $\alpha_i$ to make this happpen, and $\eta(1)$ divides $\widetilde{\alpha'_{1}}$
into two parts $\widetilde{\alpha'_{1}}', \widetilde{\alpha'_{1}}''$
such that $\widetilde{\alpha'_{1}} = \widetilde{\alpha'_{1}}'* \widetilde{\alpha'_{1}}''$.
Using the same technique as $\alpha$ being an arc, we get
\begin{align*}
\nu^{*}(\rho')(\eta^{-1}*\tilde{\alpha}*\eta) = \rho''(\eta^{-1} *\widetilde{\alpha'_{1}}'')
 A\rho''( \widetilde{\alpha'_{m}}) A \rho''(\widetilde{\alpha'_{m-1}})
A\dots A\rho''(\widetilde{\alpha'_{1}}' * \eta).
\end{align*}
Then we have 
\begin{align*}
\nu( \Phi^{(M,\mathcal{N})}(\alpha))(\rho')
&=\text{Trace}( \rho''(\eta^{-1} *\widetilde{\alpha'_{1}}'')
 A\rho''( \widetilde{\alpha'_{m}}) A \rho''(\widetilde{\alpha'_{m-1}})
A\dots A\rho''(\widetilde{\alpha'_{1}}' * \eta))\\
&= \text{Trace}( 
 A\rho''( \widetilde{\alpha'_{m}}) A \rho''(\widetilde{\alpha'_{m-1}})
A\dots A\rho''(\widetilde{\alpha'_{1}}' * \eta) \rho''(\eta^{-1} *\widetilde{\alpha'_{1}}''))\\
&= \text{Trace}( 
 A\rho''( \widetilde{\alpha'_{m}}) A \rho''(\widetilde{\alpha'_{m-1}})
A\dots A\rho''(\widetilde{\alpha'_{1}}' *\widetilde{\alpha'_{1}}''))\\
&= \text{Trace}( 
 A\rho''( \widetilde{\alpha'_{m}}) A \rho''(\widetilde{\alpha'_{m-1}})
A\dots A\rho''(\widetilde{\alpha'_{1}}))\\
&= \text{Trace}( 
 A\rho'( \widetilde{\alpha'_{m}}) A \rho'(\widetilde{\alpha'_{m-1}})
A\dots A\rho'(\widetilde{\alpha'_{1}})) \\
&= \Phi^{(M',\N')}(\Theta_{(D,\beta)} (\alpha))(\rho').
\end{align*}

\end{proof}

\section{Proof for Theorem \ref{thm3.11}}
For any $m$ by $k$ matrix $A$, $C_t(A),1\leq t\leq k,$ denotes the $t$-th column of $A$, 
$R_t(A),1\leq t\leq m$, denotes the $t$-th row of $A$.

To simplify the notation, we will omit the superscript for $\Phi^{(M,\mathcal{N})}$ when there is no confusion for $\MN$.

\subsection{The case when $\N=\emptyset$}\label{sub4.1}

Sikora proved $S_n(M;\mathbb{C},1)\simeq G_n(M)$ \cite{sikora2005skein}.
 L{\^e} and Sikora proved $S_n(M,\emptyset,1)\simeq S_n(M;\mathbb{C},1)$, which is related to the spin structure $h$ \cite{le2021stated}.
Then it is easy to show the combination of 
$S_n(M,\emptyset,1)\simeq S_n(M;\mathbb{C},1)\simeq G_n(M)\rightarrow R_n(M,\emptyset)$ is $\Phi$
,where the third map $G_n(M)\rightarrow R_n(M,\emptyset)$ is $\cY_h$ in Reamrk \ref{rem3.14}. Thus $\Phi$ is a well-defined surjective algebra homomorphism. Especially Ker$\Phi = \sqrt{0}$ since Ker$\cY_h = \sqrt{0}$.

\subsection{Independence of how to kill sinks and souces ($\mathcal{N}\neq \emptyset$)}\label{sub4.2}
When we try to use  relation (\ref{wzh.five}) to kill all the sinks and sources, we first drag all the sinks and sources close enough to some component of $\N$, then use relation (\ref{wzh.five}). In this subsection, we want to show $\Phi$ is independent of how to kill sinks and sources, that is, to show $\Phi$ is independent of how we drag sinks and sources close to $\N$.

Let $l$ be a stated $n$-web. Suppose $l'$ and $l''$ are obtained from $l$ by killing all the sinks and sources. Note that, for each sink or source of $l$, we may use different ways to kill this sink or source to get $l'$ and $l''$. First assume we kill all the sinks and sources in the same way to obtain $l'$ and $l''$ except one source or sink, which is denoted as $\mathfrak{S}$. Let $l_1$ be obtained from $l$ by first  killing all the sinks and sources except $\mathfrak{S}$ using the same way as $l'$ and $l''$ then eliminating the component containing $\mathfrak{S}$.

Suppose $\mathfrak{S}$ is a sink.
Then
\begin{align*}
l' = [\sum_{\sigma\in S_n}(-1)^{\ell(\sigma)}(\eta_n*\alpha_n)_{\sigma(n),u_n}(\eta_{n-1}*\alpha_{n-1})_{\sigma(n-1),u_{n-1}}\dots (\eta_{1}*\alpha_{1})_{\sigma(1),u_1}]\;l_1\\
l'' = [\sum_{\sigma\in S_n}(-1)^{\ell(\sigma)}(\gamma_n*\alpha_n)_{\sigma(n),u_n}(\gamma_{n-1}*\alpha_{n-1})_{\sigma(n-1),u_{n-1}}\dots (\gamma_{1}*\alpha_{1})_{\sigma(1),u_1}]\;l_1
\end{align*}
where $\alpha_t,\eta_t,\gamma_t,1\leq t\leq n, $ are framed oriented  arcs such that $\eta_t*\alpha_t,
\gamma_t*\alpha_t,1\leq t\leq n$, are well-defined framed oriented  boundary arcs in $(M,\mathcal{N})$. Note that
$\eta_t(1),1\leq t\leq n$, belong to  a same component of $\N$ (the same with $\gamma_t(1)$), and $\eta_t,1\leq t\leq n$ are isotopic to each other (the same with $\gamma_t$).

For any element $\rho\in\tilde{\chi}_n(M,\mathcal{N})$, we have 
\begin{align*}
&(\sum_{\sigma\in S_n}(-1)^{\ell(\sigma)}tr_{(\eta_n*\alpha_n)_{\sigma(n),u_n}}
tr_{(\eta_{n-1}*\alpha_{n-1})_{\sigma(n-1),u_{n-1}}}\dots tr_{(\eta_1*\alpha_1)_{\sigma(1),u_1}})(\rho)\\
= &\sum_{\sigma\in S_n}(-1)^{\ell(\sigma)}[A\rho(\widetilde{\eta_n*\alpha_n})]_{\overline{\sigma(n)},\overline{u_n}}
[A\rho(\widetilde{\eta_{n-1}*\alpha_{n-1}})]_{\overline{\sigma(n-1)},\overline{u_{n-1}}}\dots
[A\rho(\widetilde{\eta_1*\alpha_1})]_{\overline{\sigma(1)},\overline{u_1}}\\
= &(-1)^{\frac{n(n-1)}{2}}\sum_{\sigma\in S_n}(-1)^{\ell(\sigma)}[A\rho(\widetilde{\eta_n}*\widetilde{\alpha_n})]_{\sigma(n),\overline{u_n}}
[A\rho(\widetilde{\eta_{n-1}}*\widetilde{\alpha_{n-1}})]_{\sigma(n-1),\overline{u_{n-1}}}\dots
[A\rho(\widetilde{\eta_1}*\widetilde{\alpha_1})]_{\sigma(1),\overline{u_1}}\\
= &(-1)^{\frac{n(n-1)}{2}}\text{det}
\begin{pmatrix}
C_{\overline{u_1}}(A\rho(\widetilde{\eta_1}*\widetilde{\alpha_1}))&
\dots& C_{\overline{u_{n-1}}}(A\rho(\widetilde{\eta_{n-1}}*\widetilde{\alpha_{n-1}}))&
C_{\overline{u_n}}(A\rho(\widetilde{\eta_n}*\widetilde{\alpha_n}))
\end{pmatrix}\\
= &(-1)^{\frac{n(n-1)}{2}}\text{det}(A) \text{det}
\begin{pmatrix}
C_{\overline{u_1}}(\rho(\widetilde{\eta_1}*\widetilde{\alpha_1}))&\dots&
C_{\overline{u_{n-1}}}(\rho(\widetilde{\eta_{n-1}}*\widetilde{\alpha_{n-1}}))&
C_{\overline{u_n}}(\rho(\widetilde{\eta_n}*\widetilde{\alpha_n}))
\end{pmatrix}\\
= &(-1)^{\frac{n(n-1)}{2}} \text{det}
\begin{pmatrix}
C_{\overline{u_1}}(\rho(\widetilde{\eta_1}*\widetilde{\alpha_1}))&\dots &
C_{\overline{u_{n-1}}}(\rho(\widetilde{\eta_{n-1}}*\widetilde{\alpha_{n-1}})) &
C_{\overline{u_n}}(\rho(\widetilde{\eta_n}*\widetilde{\alpha_n}))
\end{pmatrix}.
\end{align*}
Similarly we have 
\begin{align*}
&(\sum_{\sigma\in S_n}(-1)^{\ell(\sigma)}tr_{(\gamma_n*\alpha_n)_{\sigma(n),u_n}}
tr_{(\gamma_{n-1}*\alpha_{n-1})_{\sigma(n-1),u_{n-1}}}\dots tr_{(\gamma_1*\alpha_1)_{\sigma(1),u_1}})(\rho)\\
= &(-1)^{\frac{n(n-1)}{2}} \text{det}
\begin{pmatrix}
C_{\overline{u_1}}(\rho(\widetilde{\gamma_1}*\widetilde{\alpha_1}))&\dots&
C_{\overline{u_{n-1}}}(\rho(\widetilde{\gamma_{n-1}}*\widetilde{\alpha_{n-1}}))&
C_{\overline{u_n}}(\rho(\widetilde{\gamma_n}*\widetilde{\alpha_n}))
\end{pmatrix}\\
= &(-1)^{\frac{n(n-1)}{2}} \text{det}
\begin{pmatrix}
C_{\overline{u_1}}(\rho(\widetilde{\gamma_1}*\widetilde{\eta_1}^{-1})\rho(\widetilde{\eta_1}*\widetilde{\alpha_1}))&\dots&
C_{\overline{u_n}}(\rho(\widetilde{\gamma_n}*\widetilde{\eta_n}^{-1})\rho(\widetilde{\eta_n}*\widetilde{\alpha_n}))
\end{pmatrix}\\
= &(-1)^{\frac{n(n-1)}{2}} \text{det}
\begin{pmatrix}
C_{\overline{u_1}}(\rho(\widetilde{\eta_1}*\widetilde{\alpha_1}))&\dots&
C_{\overline{u_{n-1}}}(\rho(\widetilde{\eta_{n-1}}*\widetilde{\alpha_{n-1}}))&
C_{\overline{u_n}}(\rho(\widetilde{\eta_n}*\widetilde{\alpha_n}))
\end{pmatrix}.
\end{align*}
The last equality is because $\rho(\widetilde{\gamma_1}*\widetilde{\eta_1}^{-1})
=\dots = \rho(\widetilde{\gamma_n}*\widetilde{\eta_n}^{-1})\in SL(n,\mathbb{C})$.

Suppose $\mathfrak{S}$ is a source. Similarly we have 
\begin{align*}
l' = [\sum_{\sigma\in S_n}(-1)^{\ell(\sigma)}(\beta_n*\epsilon_n)_{v_n, \sigma(n)}(\beta_{n-1}*\epsilon_{n-1})_{v_{n-1}, \sigma(n-1)}\dots (\beta_1*\epsilon_1)_{v_1, \sigma(1)}]\;l_1\\
l'' = [\sum_{\sigma\in S_n}(-1)^{\ell(\sigma)}(\beta_n*\zeta_n)_{v_n, \sigma(n)}(\beta_{n-1}*\zeta_{n-1})_{v_{n-1}, \sigma(n-1)}\dots (\beta_1*\zeta_1)_{v_1, \sigma(1)}]\;l_1\\
\end{align*}
where $\beta_t,\epsilon_t,\zeta_t,1\leq t\leq n$, are framed oriented  arcs such that $\beta_t*\epsilon_t,
\beta_t*\zeta_t,1\leq t\leq n$, are well-defined framed oriented  boundary arcs in $(M,\mathcal{N})$. 

For any element $\rho\in\tilde{\chi}_n(M,\mathcal{N})$, similarly we can get 
\begin{align*}
&(\sum_{\sigma\in S_n}(-1)^{\ell(\sigma)}tr_{(\beta_n*\epsilon_n)_{v_n, \sigma(n)}}
tr_{((\beta_{n-1}*\epsilon_{n-1})_{v_{n-1}, \sigma(n-1)}}\dots tr_{(\beta_1*\epsilon_1)_{v_1, \sigma(1)}})(\rho)\\
= &(-1)^{\frac{n(n-1)}{2}} \text{det}
\begin{pmatrix}
R_{\overline{v_1}}(A\rho(\widetilde{\beta_1}*\widetilde{\epsilon_1}))\\
\vdots\\
R_{\overline{v_{n-1}}}(\rho(\widetilde{\beta_{n-1}}*\widetilde{\epsilon_{n-1}}))\\
R_{\overline{v_n}}(\rho(\widetilde{\beta_n}*\widetilde{\epsilon_n}))
\end{pmatrix}.
\end{align*}
Then we have 
\begin{align*}
&(\sum_{\sigma\in S_n}(-1)^{\ell(\sigma)}tr_{(\beta_n*\zeta_n)_{v_n, \sigma(n)}}
tr_{((\beta_{n-1}*\zeta_{n-1})_{v_{n-1}, \sigma(n-1)}}\dots tr_{(\beta_1*\zeta_1)_{v_1, \sigma(1)}})(\rho)\\
= &(-1)^{\frac{n(n-1)}{2}} \text{det}
\begin{pmatrix}
R_{\overline{v_1}}(A\rho(\widetilde{\beta_1}*\widetilde{\zeta_1}))\\
\vdots\\
R_{\overline{v_{n-1}}}(\rho(\widetilde{\beta_{n-1}}*\widetilde{\zeta_{n-1}}))\\
R_{\overline{v_n}}(\rho(\widetilde{\beta_n}*\widetilde{\zeta_n}))
\end{pmatrix}\\=&
(-1)^{\frac{n(n-1)}{2}} \text{det}
\begin{pmatrix}
R_{\overline{v_1}}(A\rho(\widetilde{\beta_1}*\widetilde{\epsilon_1})\rho(\widetilde{\epsilon_1}^{-1} * \widetilde{\zeta_1}) )\\
\vdots\\
R_{\overline{v_{n-1}}}(\rho(\widetilde{\beta_{n-1}}*\widetilde{\epsilon_{n-1}})\rho(\widetilde{\epsilon_{n-1}}^{-1} * \widetilde{\zeta_{n-1}}))\\
R_{\overline{v_n}}(\rho(\widetilde{\beta_n}*\widetilde{\epsilon_n})\rho(\widetilde{\epsilon_n}^{-1} * \widetilde{\zeta_n}) )\\
\end{pmatrix}\\
=&(-1)^{\frac{n(n-1)}{2}} \text{det}
\begin{pmatrix}
R_{\overline{v_1}}(A\rho(\widetilde{\beta_1}*\widetilde{\epsilon_1}))\\
\vdots\\
R_{\overline{v_{n-1}}}(\rho(\widetilde{\beta_{n-1}}*\widetilde{\epsilon_{n-1}}))\\
R_{\overline{v_n}}(\rho(\widetilde{\beta_n}*\widetilde{\epsilon_n}))
\end{pmatrix}.
\end{align*}

Thus we have $\Phi(l') = \Phi(l'')$. In the general case, we have a sequence $l' = l^{(1)}, l^{(2)}, \dots, l^{(k)} = l''$ such that each $l^{(t)}$ (where $1 \leq t \leq k$) is obtained from $l$ by eliminating all sources and sinks using relation (\ref{wzh.five}). Additionally, each $l^{(t)}$ and $l^{(t+1)}$ (where $1 \leq t \leq k-1$) are obtained from $l$ using the same method to eliminate all but one sink or source.
 Then $\Phi(l') = \Phi(l^{(1)})=\Phi(l^{(2)})=\dots=
\Phi( l^{(k)})= \Phi(l'')$.
Then $\Phi$ is well-defined on the set of framed $n$-webs.

Suppose the stated $n$-webs $l_1$ and $l_2$ are isotopic. From the definition of $\Phi$, we first use relation relation (\ref{wzh.five}) to kill all the sinks and sources to obtain $(l_1)'$. According to the isotopy between $l_1$ and $l_2$ and how we kill sinks and sources in $l_1$,
 we can pick a way to kill all the sinks and sources in $l_2$ to obtain $(l_2)'$ such that  $(l_1)'$ and 
 $(l_2)'$ are homopotic relative to $\N$. Thus $\Phi(l_1) = \Phi(l_2)$, that is, $\Phi$ is well-defined on the set of isotopy classes of stated $n$-webs.

\subsection{Checking for relations ($\N\neq \emptyset$)}\label{sub4.3}
From the definition of $\Phi$, we know $\Phi$ respects relations (\ref{w.cross}),(\ref{wzh.five}) and
 (\ref{wzh.eight}).  Since, for any $\rho \in \tilde{\chi}_n(M,\mathcal{N})$, $\rho(\vartheta) = d_n I$, then $\Phi$ respects relations
(\ref{w.twist})
and
(\ref{w.unknot}).

We use $l'$ (respectively $l''$) to denote the stated $n$-webs on the left (repectively right) hand side of "$=$" in relations (\ref{wzh.four}), (\ref{wzh.six}) and (\ref{wzh.seven}).


Relation  (\ref{wzh.four}): From the definition of $\Phi$ and Subsection \ref{sub4.2}, we can suppose the parts outside of the box are $2n$ framed oriented arcs connecting to the box. We label the framed oriented arcs connecting to the box on the left edge as $\alpha_n, \dots, \alpha_2,\alpha_1$ from top to bottom, 
label the framed oriented arcs connecting to the box on the right edge as $\beta_n, \dots, \beta_2,\beta_1$ from top to bottom. Suppose $s(\alpha_t(0)) = u_t, s(\beta_t(1))=v_t,1\leq t\leq n$. To kill the sink and source in $l'$ using relation (\ref{wzh.five}), we use the same path to drag them close to $\N$. 

For $\rho\in\tilde{\chi}_n(M,\mathcal{N})$, we have
\begin{align*}
\Phi(l')(\rho) =& \text{det}
\begin{pmatrix}
R_{\overline{v_n}}(A\rho(\widetilde{\beta_n}*\widetilde{\epsilon_n}))\\
\vdots\\
R_{\overline{v_1}}(A\rho(\widetilde{\beta_1}*\widetilde{\epsilon_1}))
\end{pmatrix}\text{det}
\begin{pmatrix}
C_{\overline{u_1}}(\rho(\widetilde{\eta_1}*\widetilde{\alpha_1}))&\dots &
C_{\overline{u_n}}(\rho(\widetilde{\eta_n}*\widetilde{\alpha_n}))
\end{pmatrix}\\
=&\text{det}
\begin{pmatrix}
R_{\overline{v_n}}(A\rho(\widetilde{\beta_n}*\widetilde{\epsilon_n}))C_{\overline{u_1}}(\rho(\widetilde{\eta_1}*\widetilde{\alpha_1}))&\dots&
R_{\overline{v_n}}(A\rho(\widetilde{\beta_n}*\widetilde{\epsilon_n}))C_{\overline{u_n}}(\rho(\widetilde{\eta_n}*\widetilde{\alpha_n}))\\
\vdots& &\vdots\\
R_{\overline{v_1}}(A\rho(\widetilde{\beta_1}*\widetilde{\epsilon_1}))C_{\overline{u_1}}(\rho(\widetilde{\eta_1}*\widetilde{\alpha_1}))&\dots&
R_{\overline{v_1}}(A\rho(\widetilde{\beta_1}*\widetilde{\epsilon_1}))C_{\overline{u_n}}(\rho(\widetilde{\eta_n}*\widetilde{\alpha_n}))\\
\end{pmatrix}\\
=&\text{det}
\begin{pmatrix}
[A\rho(\widetilde{\beta_n}*\widetilde{\epsilon_n})\rho(\widetilde{\eta_1}*\widetilde{\alpha_1})]_{\overline{v_n},\overline{u_1}}&\dots&
[A\rho(\widetilde{\beta_n}*\widetilde{\epsilon_n})\rho(\widetilde{\eta_n}*\widetilde{\alpha_n})]_{\overline{v_n},\overline{u_n}}\\
\vdots& &\vdots\\
[A\rho(\widetilde{\beta_1}*\widetilde{\epsilon_1})\rho(\widetilde{\eta_1}*\widetilde{\alpha_1})]_{\overline{v_1},\overline{u_1}}&\dots&
[A\rho(\widetilde{\beta_1}*\widetilde{\epsilon_1})\rho(\widetilde{\eta_n}*\widetilde{\alpha_n})]_{\overline{v_1},\overline{u_n}}\\
\end{pmatrix}\\
=&(-1)^{\frac{n(n-1)}{2}}\text{det}
\begin{pmatrix}
[A\rho(\widetilde{\beta_1}*\widetilde{\epsilon_1}*\widetilde{\eta_1}*\widetilde{\alpha_1})]_{\overline{v_1},\overline{u_1}}&\dots&
[A\rho(\widetilde{\beta_1}*\widetilde{\epsilon_1}*\widetilde{\eta_n}*\widetilde{\alpha_n})]_{\overline{v_1},\overline{u_n}}\\
\vdots& &\vdots\\
[A\rho(\widetilde{\beta_n}*\widetilde{\epsilon_n}*\widetilde{\eta_1}*\widetilde{\alpha_1})]_{\overline{v_n},\overline{u_1}}&\dots&
[A\rho(\widetilde{\beta_n}*\widetilde{\epsilon_n}*\widetilde{\eta_n}*\widetilde{\alpha_n})]_{\overline{v_n},\overline{u_n}}\\
\end{pmatrix}\\
\end{align*}

For each pair $1\leq i,j\leq n$, we use $a_{i,j}$ to denote the oriented straight line in the shaded box (its framing is the one pointing to readers)
connecting  $\alpha_j(1)$  and $\beta_i(0)$  such that $\beta_i * a_{i,j} * \alpha_j$ is a well-defined stated framed oriented boundary arc.
Because we use the same path to drag the source and the sink,
then $$\rho(\widetilde{\beta_i}*\widetilde{\epsilon_i}*\widetilde{\eta_j}*\widetilde{\alpha_j})
= d_n \rho(\widetilde{\beta_i}*\widetilde{a_{i,j}}*\widetilde{\alpha_j})$$ for all $1\leq i,j\leq n$,
or 
$$\rho(\widetilde{\beta_i}*\widetilde{\epsilon_i}*\widetilde{\eta_j}*\widetilde{\alpha_j})
=  \rho(\widetilde{\beta_i}*\widetilde{a_{i,j}}*\widetilde{\alpha_j})$$ for all $1\leq i,j\leq n$.

For any $\rho\in\tilde{\chi}_n(M,\mathcal{N})$, since $(d_n)^{n} = 1$, we have 
\begin{align*}
&\Phi(l'')(\rho) \\= &(-1)^{\frac{n(n-1)}{2}}
\sum_{\sigma\in S_n}(-1)^{l(\sigma)}[A\rho(\widetilde{\beta_{\sigma(1)}}*\widetilde{\epsilon_{\sigma(1)}}*\widetilde{\eta_1}*\widetilde{\alpha_1})]_{\overline{v_{\sigma(1)}},\overline{u_1}}\dots
[A\rho(\widetilde{\beta_{\sigma(n)}}*\widetilde{\epsilon_{\sigma(n)}}*\widetilde{\eta_n}*\widetilde{\alpha_n})]_{\overline{v_{\sigma(n)}},\overline{u_n}}\\
=&(-1)^{\frac{n(n-1)}{2}}\text{det}
\begin{pmatrix}
[A\rho(\widetilde{\beta_1}*\widetilde{\epsilon_1}*\widetilde{\eta_1}*\widetilde{\alpha_1})]_{\overline{v_1},\overline{u_1}}&\dots&
[A\rho(\widetilde{\beta_1}*\widetilde{\epsilon_1}*\widetilde{\eta_n}*\widetilde{\alpha_n})]_{\overline{v_1},\overline{u_n}}\\
\vdots& &\vdots\\
[A\rho(\widetilde{\beta_n}*\widetilde{\epsilon_n}*\widetilde{\eta_1}*\widetilde{\alpha_1})]_{\overline{v_n},\overline{u_1}}&\dots&
[A\rho(\widetilde{\beta_n}*\widetilde{\epsilon_n}*\widetilde{\eta_n}*\widetilde{\alpha_n})]_{\overline{v_n},\overline{u_n}}\\
\end{pmatrix}\\
\end{align*}

Thus $\Phi(l') = \Phi(l'')$.


Relation (\ref{wzh.six}): We need to show
$$\Phi(\raisebox{-.20in}{
\begin{tikzpicture}
\tikzset{->-/.style=
{decoration={markings,mark=at position #1 with
{\arrow{latex}}},postaction={decorate}}}
\filldraw[draw=white,fill=gray!20] (-0.7,-0.7) rectangle (0,0.7);
\draw [line width =1.5pt,decoration={markings, mark=at position 1 with {\arrow{>}}},postaction={decorate}](0,0.7)--(0,-0.7);
\draw [color = black, line width =1pt] (0 ,0.3) arc (90:270:0.5 and 0.3);
\node [right]  at(0,0.3) {$i$};
\node [right] at(0,-0.3){$j$};
\draw [line width =1pt,decoration={markings, mark=at position 0.5 with {\arrow{>}}},postaction={decorate}](-0.5,0.02)--(-0.5,-0.02);
\end{tikzpicture}} ) = \Phi( \raisebox{-.20in}{
\begin{tikzpicture}
\tikzset{->-/.style=
{decoration={markings,mark=at position #1 with
{\arrow{latex}}},postaction={decorate}}}
\filldraw[draw=white,fill=gray!20] (-0.7,-0.7) rectangle (0,0.7);
\draw [line width =1.5pt,decoration={markings, mark=at position 1 with {\arrow{>}}},postaction={decorate}](0,0.7)--(0,-0.7);
\draw [color = black, line width =1pt] (0 ,0.3) arc (90:270:0.5 and 0.3);
\node [right]  at(0,0.3) {$i$};
\node [right] at(0,-0.3){$j$};
\draw [line width =1pt,decoration={markings, mark=at position 0.5 with {\arrow{<}}},postaction={decorate}](-0.5,0.02)--(-0.5,-0.02);
\end{tikzpicture}} ) = \delta_{\bar j,i }\,  (-1)^{n-i}.$$
We use $\alpha_1$ to denote $\raisebox{-.20in}{
\begin{tikzpicture}
\tikzset{->-/.style=
{decoration={markings,mark=at position #1 with
{\arrow{latex}}},postaction={decorate}}}
\filldraw[draw=white,fill=gray!20] (-0.7,-0.7) rectangle (0,0.7);
\draw [line width =1.5pt,decoration={markings, mark=at position 1 with {\arrow{>}}},postaction={decorate}](0,0.7)--(0,-0.7);
\draw [color = black, line width =1pt] (0 ,0.3) arc (90:270:0.5 and 0.3);
\node [right]  at(0,0.3) {$i$};
\node [right] at(0,-0.3){$j$};
\draw [line width =1pt,decoration={markings, mark=at position 0.5 with {\arrow{>}}},postaction={decorate}](-0.5,0.02)--(-0.5,-0.02);
\end{tikzpicture}} $ and $\alpha_2$ to denote 
$\raisebox{-.20in}{
\begin{tikzpicture}
\tikzset{->-/.style=
{decoration={markings,mark=at position #1 with
{\arrow{latex}}},postaction={decorate}}}
\filldraw[draw=white,fill=gray!20] (-0.7,-0.7) rectangle (0,0.7);
\draw [line width =1.5pt,decoration={markings, mark=at position 1 with {\arrow{>}}},postaction={decorate}](0,0.7)--(0,-0.7);
\draw [color = black, line width =1pt] (0 ,0.3) arc (90:270:0.5 and 0.3);
\node [right]  at(0,0.3) {$i$};
\node [right] at(0,-0.3){$j$};
\draw [line width =1pt,decoration={markings, mark=at position 0.5 with {\arrow{<}}},postaction={decorate}](-0.5,0.02)--(-0.5,-0.02);
\end{tikzpicture}} $, we have 
$$\Phi(\alpha_1)(\rho) = [A\rho(\widetilde{\alpha_1})]_{\bar{j},\bar{i}}
= d_n A_{\bar{j},\bar{i}} =  \delta_{\bar j,i }\,  (-1)^{n-i},\Phi(\alpha_2)(\rho) = [A\rho(\widetilde{\alpha_2})]_{\bar{i},\bar{j}}
= A_{\bar{i},\bar{j}} =  \delta_{\bar j,i }\,  (-1)^{n-i}.$$

Relation (\ref{wzh.seven}): We only prove the case when the white dot represents an arrow going from right to left.
From the definition of $\Phi$ and Subsection \ref{sub4.2}, we only have two cases to consider: (1) the left hand side of "$=$" is a knot, (2)
the left hand side of "$=$" is an arc. 

When the left hand side of "$=$" is a knot, the right hand side of "$=$" is a framed oriented boundary arc, which is denoted as $\alpha$. 
Then for $\rho\in\tilde{\chi}_n(M,\mathcal{N})$, we have 
\begin{align*}
\Phi(l'')(\rho) = \sum_{1\leq i\leq n}  (-1)^{i+1}[A\rho(\tilde{\alpha})]_{i,\bar{i}}=
\sum_{1\leq i\leq n}  [\rho(\tilde{\alpha})]_{\bar{i},\bar{i}} = \text{Trace}(\rho(\tilde{\alpha}))
= \Phi(l')(\rho).
\end{align*}

When the left hand side of "$=$" is an arc, the right hand side of "$=$" consists two framed oriented boundary arcs, which are denoted as $\gamma_2$ and $\gamma_1$ such that $\gamma_2$ is above $\gamma_1$ is the box. 
Suppose $s(\gamma_1(0)) = v, s(\gamma_2(1)) = u$.
Then for $\rho\in\tilde{\chi}_n(M,\mathcal{N})$, we have 
\begin{align*}
\Phi(l'')(\rho) &= \sum_{1\leq i\leq n}  (-1)^{i+1}[A\rho(\widetilde{\gamma_2})]_{\bar{u},\bar{i}}
[A\rho(\widetilde{\gamma_1})]_{{i},\bar{v}}=
 \sum_{1\leq i\leq n}  [A\rho(\widetilde{\gamma_2})]_{\bar{u},\bar{i}}
[\rho(\widetilde{\gamma_1})]_{\bar{i},\bar{v}}\\
&=[A\rho(\widetilde{\gamma_2})\rho(\widetilde{\gamma_1})]_{\bar{u},\bar{v}}
=[A\rho(\widetilde{\gamma_2}*\widetilde{\gamma_1})]_{\bar{u},\bar{v}}
= \Phi(l')(\rho).
\end{align*}

\subsection{Algebra homomorphism and surjectivity}
The definition of $\Phi$ implies it is an algebra homormophism. 

When $\N$ is empty, we already proved $\Phi$ is surjective.
Assume $\N\neq \emptyset$.
 For any element $[\alpha]\in \pi_1^{Mor}(M,\mathcal{N})$, we choose a representative $\alpha$ for $[\alpha]$ such that $\alpha$ has no self-intersection and only intersects $\partial M$ at its endpoints. Then we can give a framing for $\alpha$ to make $\alpha$ a framed oriented boundary arc for $(M,\mathcal{N})$. Then $R_n(M,\mathcal{N})$ is generated by $\Phi(\alpha_{j,i}),1\leq j,i\leq n, [\alpha]
\in\pi_1^{Mor}(M,\mathcal{N})$, as an algebra. Thus $\Phi$ is surjective.

In Section \ref{subb5}, we will give a unique way to lift  $[\alpha]$ to a framed oriented boundary arc.

\def \N{\mathcal{N}}

\section{Classical limit  and Ker$\Phi$}\label{subb5}

In this section we try to understand the classical limit of the stated $SL_n$-skein module of the marked 3-manifold. Then we will use the classical limit to show the kernel of $\Phi$ is $\sqrt{0}$. Using Lemma \ref{6688},  we can  reduce the general marked 3-manifold to the  connected marked 3-manifold.

\begin{lemma}[\cite{blyth2018module,przytycki1998fundamentals}]\label{8866}
Suppose $0\rightarrow A_1\rightarrow B_1 \rightarrow C_1\rightarrow 0,$ and 
$0\rightarrow A_2\rightarrow B_2 \rightarrow C_2\rightarrow 0$ are two short exact sequences, then 
$$0\rightarrow A_1\otimes B_2 + B_1\otimes A_2\rightarrow B_1\otimes B_2\rightarrow C_1\otimes C_2\rightarrow 0$$ is an exact sequence. All $A_i,B_i,C_i$ are vector spaces over $\mathbb{C}$ and all the involved maps are linear maps.

\end{lemma}


\begin{lemma}\label{6688}
Suppose $(M,\mathcal{N})$ is the disjoint union of $(M_1,\mathcal{N}_1)$ and $(M_2,\mathcal{N}_2)$. If Ker$\,\Phi^{(M_i,\N_i)}
=\sqrt{0}_{S_n(M_i,\N_i,1)}$ for $i=1,2$, then  Ker$\,\Phi^{(M,\mathcal{N})}
=\sqrt{0}_{S_n(M,\mathcal{N},1)}$.
\end{lemma}

\begin{proof}
Since $R_n(M,\mathcal{N})$ contains no nonzero nilpotents, we have  $\sqrt{0}_{S_n(M,\mathcal{N},1)}\subset$\,Ker$\,\Phi^{(M,\mathcal{N})}.$
We know $S_n(M,\mathcal{N},1) = S_n(M_1,\mathcal{N}_1,1)\otimes S_n(M_2,\mathcal{N}_2,1),\;R_n(M,\mathcal{N}) = R_n(M_1,\mathcal{N}_1) \otimes R_n(M_2,\mathcal{N}_2)$. From the assumption, we have the following two exact sequences:
\begin{align*}
0\rightarrow \sqrt{0}_{S_n(M_i,\N_i,1)}\rightarrow S_n(M_i,\N_i,1) \rightarrow R_n(M_i,\N_i)\rightarrow 0
\end{align*}
for $i=1,2$.
Then  Lemma \ref{8866} implies the following exact sequence:
\begin{align*}0\rightarrow \sqrt{0}_{S_n(M_1,\mathcal{N}_1,1)}\otimes S_n(M_2,\mathcal{N}_2,1) + S_n(M_1,\mathcal{N}_1,1)\otimes\sqrt{0}_{S_n(M_2,\mathcal{N}_2,1)}\rightarrow\\ S_n(M_1,\mathcal{N}_1,1)\otimes S_n(M_2,\mathcal{N}_2,1)\rightarrow R_n(M_1,\mathcal{N}_1)\otimes R_n(M_2,\mathcal{N}_2)\rightarrow 0.
\end{align*}
Thus we have 
$$\text{Ker}\Phi^{(M,\mathcal{N})} = \sqrt{0}_{S_n(M_1,\mathcal{N}_1,1)}\otimes S_n(M_2,\mathcal{N}_2,1) + S_n(M_1,\mathcal{N}_1,1)\otimes\sqrt{0}_{S_n(M_2,\mathcal{N}_2,1)}\subset \sqrt{0}_{S_n(M,\mathcal{N},1)}.$$
Then
$$\text{Ker}\Phi^{(M,\mathcal{N})} = \sqrt{0}_{S_n(M_1,\mathcal{N}_1,1)}\otimes S_n(M_2,\mathcal{N}_2,1) + S_n(M_1,\mathcal{N}_1,1)\otimes\sqrt{0}_{S_n(M_2,\mathcal{N}_2,1)}= \sqrt{0}_{S_n(M,\mathcal{N},1)}.$$

\end{proof}

In the remaining of this section, we will assume all the marked 3-manifolds involved are connected.
We also fix a relative spin structure $h$ for $(M,\mathcal{N})$. For any (stated) framed oriented  boundary arc $\alpha$ in $(M,\mathcal{N})$, we consider $[\alpha]$ as an element in $\pi_1^{Mor}(M,\mathcal{N})$ by forgetting the framing of $\alpha$.

\begin{rem}\label{rre5.1}
A morphism $[\alpha] \in \pi_1^{Mor}(M,\mathcal{N})$ and two integers $1\leq i,j\leq n$ uniquely determine an element in $S_n(M,\mathcal{N},1)$ in the following way: We choose a good representative $\alpha$ such that $\alpha$ is a properly embedded arc in $M$. Then we give a framing to $\alpha$ respecting $\N$, that is, the framing at endpoints are given by the velocity vectors of $\N$. We denote this framed oriented boundary arc as $\hat{\alpha}$. We choose the framing such that $h(\widetilde{\hat{\alpha}}) = 0$, then we obtain an element $\hat{\alpha}_{i,j}
\in S_n(M,\mathcal{N},1)$. Suppose we choose a different good representative $\alpha'$. We have $[\alpha] = [\alpha']\in \pi_1^{Mor}(M,\mathcal{N})$ and $h(\widetilde{\hat{\alpha'}}) = h(\widetilde{\hat{\alpha}}) = 0$. Then $\hat{\alpha'}_{i,j} = \hat{\alpha}_{i,j}$ because of relations (\ref{w.cross}), (\ref{w.twist}) , (\ref{wzh.eight}) and Corollary \ref{cccc3.2} (here we  use a  standard fact that two embeddings of a compact graph in $M$ are homotopic if and only if  one can be obtained from the other by crossing changes, height changes, and   isotopies \cite{skeingroup}).

 We use 
$S^{[\alpha]}_{i,j}$ to denote $\hat{\alpha}_{i,j}$, and use $S^{[\alpha]}$ to denote an $n$ by $n$  matrix in $S_n(M,\mathcal{N},1)$ such that $(S^{[\alpha]})_{i,j} = S^{[\alpha]}_{i,j},1\leq i,j\leq n$.

For any two stated oriented framed boundary arcs $\alpha_1,\alpha_2$, suppose $s(\alpha_1(0))
= s(\alpha_2(0))$ and $s(\alpha_1(1))
= s(\alpha_2(1))$. If $h(\widetilde{\alpha_1}) = h(\widetilde{\alpha_2}) $ and $[\alpha_1] = [\alpha_2]\in
\pi_1^{Mor}(M,\mathcal{N})$, then $\alpha_1 =\alpha_2 \in S_n(M,\mathcal{N},1)$ because of relations (\ref{w.cross}), (\ref{w.twist}), (\ref{wzh.eight}) and Corollary \ref{cccc3.2}.

The for any stated oriented framed boundary arc $\alpha_{i,j}$, we have $$\alpha_{i,j} =
d_n^{h(\tilde{\alpha})} S^{[\alpha]}_{i,j}\in S_n(M,\mathcal{N},1).$$
\end{rem}

\begin{proposition}\label{prop5.2}
(a) For any two morphisms $[\alpha],[\beta]\in \pi_1^{Mor}(M,\mathcal{N})$, if $[\beta][\alpha]$ makes sense, then 
$A S^{[\beta*\alpha]} = A S^{[\beta]} A S^{[\alpha]}$.

(b) For any $[\eta]\in \pi_1^{Mor}(M,\mathcal{N})$, we have 
det$(S^{[\eta]}) = 1$. Especially det$(A S^{[\eta]}) = 1$.

(c) Suppose $[o]\in \pi_1(M,\mathcal{N})$ is the identity morphism for an object, then $S^{[o]} = d_n A$.
Especially $A S^{[o]} = I$.
\end{proposition}
\begin{proof}
(a) We have 
$$(S^{[\beta]} A S^{[\alpha]})_{i,j} = \sum_{1\leq k\leq n} (-1)^{k+1}S^{[\beta]}_{i,k} S^{[\alpha]}_{\bar{k}, j}
= S^{[\beta*\alpha]}_{i,j} = (S^{[\beta*\alpha]})_{i,j}$$
where the second equality is because of relation (\ref{wzh.seven}).
Thus $A S^{[\beta*\alpha]} = A S^{[\beta]} A S^{[\alpha]}$.

%

(b)We have 
$$\text{det}(S^{[\eta]}) = \sum_{\sigma\in S_n} (-1)^{l(\sigma)} S^{[\eta]}_{1,\sigma(1)}
 S^{[\eta]}_{2,\sigma(2)} \dots  S^{[\eta]}_{n,\sigma(n)}
= \raisebox{-.30in}{
\begin{tikzpicture}
\tikzset{->-/.style=
{decoration={markings,mark=at position #1 with
{\arrow{latex}}},postaction={decorate}}}
\filldraw[draw=white,fill=gray!20] (0,-0.7) rectangle (1.2,1.3);
\draw [line width =1.5pt,decoration={markings, mark=at position 1 with {\arrow{>}}},postaction={decorate}](1.2,-0.7)--(1.2,1.3);
\draw [line width =1pt,decoration={markings, mark=at position 0.5 with {\arrow{<}}},postaction={decorate}](1.2,1)  --(0.2,0);
\draw [line width =1pt,decoration={markings, mark=at position 0.5 with {\arrow{<}}},postaction={decorate}](1.2,0)  --(0.2,0);
\draw [line width =1pt,decoration={markings, mark=at position 0.5 with {\arrow{<}}},postaction={decorate}](1.2,-0.4)--(0.2,0);
\node  at(1,0.5) {$\vdots$};
\node [right] at(1.2,1) {$1$};
\node [right] at(1.2,0) {$n-1$};
\node [right] at(1.2,-0.4) {$n$};
\end{tikzpicture}}
 = 1$$
where the second equality is from relation (\ref{wzh.five}) and 
the last equality is because of equation (54) in  \cite{le2021stated}.

(c) For $1\leq i,j\leq n$, we have 
$$S^{[o]}_{i,j} = \raisebox{-.20in}{
\begin{tikzpicture}
\tikzset{->-/.style=
{decoration={markings,mark=at position #1 with
{\arrow{latex}}},postaction={decorate}}}
\filldraw[draw=white,fill=gray!20] (-0.7,-0.7) rectangle (0,0.7);
\draw [line width =1.5pt,decoration={markings, mark=at position 1 with {\arrow{>}}},postaction={decorate}](0,-0.7)--(0,0.7);
\draw [color = black, line width =1pt] (0 ,0.3) arc (90:270:0.5 and 0.3);
\node [right]  at(0,0.3) {$i$};
\node [right] at(0,-0.3){$j$};
\draw [line width =1pt,decoration={markings, mark=at position 0.5 with {\arrow{<}}},postaction={decorate}](-0.5,0.02)--(-0.5,-0.02);
\end{tikzpicture}} = d_n A_{i,j}.$$
Thus $ S^{[o]} = d_n A$.

\end{proof}

\subsection{Isomomorphism between $S_n(M,\mathcal{N},1)$ and $\Gamma_n(M)$ when $\N$ has one component}

In this subsection, $\N$ always containes one component unless specified.
If $\N$ has only one component, then $\pi_1^{Mor}(M,\mathcal{N}) = \pi_1(M)$ (we choose the base point for $\pi_1(M)$ to be a point in $\N$).

\begin{lemma}\label{lmm5.3}

Let $(M,\mathcal{N})$ be a marked 3-manifold with $\N$ consisting of one component. There exists an algebra homomorphism
$F:\Gamma_n(M)\rightarrow S_n(M,\mathcal{N},1)$ defined by 
$$F([\alpha]_{i,j}) = F((Q_{[\alpha]})_{i,j}) = (A S^{[\alpha]})_{i,j}$$
where $[\alpha]\in \pi_1^{Mor}(M,\mathcal{N}),1\leq i,j\leq n.$

\end{lemma}

\begin{proof}

Lemma \ref{prop5.2} shows $F$ respects all the relations defined for $\Gamma_n(M)$. Thus
$F$ is a well-defined algebra homomorphism.
\end{proof}

Let $\alpha$ be a framed oriented arc in $S_n(M,\mathcal{N},1)$. Then $[\alpha]$ is an element in $\pi_1^{Mor}(M,\mathcal{N})$ by forgetting  the framing of $\alpha$. We define $G(\alpha_{i,j}) = d_n^{h(\tilde{\alpha})+1}(-1)^{i+1}[\alpha]_{\bar{i},j}\in \Gamma_n(M)$.
For a  framed oriented  knot $\alpha$, first we forget the framing of $\alpha$ and then we use a path $\beta$ to connect $\alpha$ and $\N$. Then we obtain an elemnt in $\pi_1^{Mor}(M,\mathcal{N})$, which is denoted as $[\alpha_{\beta}]$. 
We define $G(\alpha) = d_n^{h(\tilde{\alpha})}\text{Trace}(Q_{[\alpha_{\beta}]})\in \Gamma_n(M)$. It  is easy to show $G(\alpha)$ is independent of the choice of $\beta$.

For any stated $n$-web $l$, we use relation (\ref{wzh.five}) to kill all the sinks and sources  to obtain a new stated $n$-web $l'$. Suppose $l' = \cup_{\alpha}\alpha$ where each $\alpha$ is a stated framed oriented boundary arc or a framed oriented knot, define $G(l) = \Pi_{\alpha} G(\alpha)$.

\begin{lemma}\label{lmm5.4}
The above map $G:S_n(M,\mathcal{N},1)\rightarrow \Gamma_n(M)$ is a well-defined algebra homomorphism.
\end{lemma}
\begin{proof}
We prove $G$ is well-defined in two steps. First we prove the definition of $G$ is independent of the choice of how we kill sinks and sources, then we prove $G$ respects all the relations defined for $S_n(M,\mathcal{N},1)$. Note that these two steps appeared when we tried to prove Theorem \ref{thm3.11}. Actually the proving techniques here are the same with the techniques used in Subsections \ref{sub4.2} and \ref{sub4.3}. So here we  omit all the details. 
\end{proof}

\begin{theorem}\label{thm5.5}
Let $(M,\mathcal{N})$ be a marked 3-manifold with $\N$ consisting of one component. There exist  algebra homomorphisms
$F:\Gamma_n(M)\rightarrow S_n(M,\mathcal{N},1), G:S_n(M,\mathcal{N},1)\rightarrow \Gamma_n(M)$ such that 
$F\circ G = Id_{S_n(M,\mathcal{N},1)}$ and $G\circ F = Id_{\Gamma_n(M)}$. Especially $\Gamma_n(M)\simeq S_n(M,\mathcal{N},1)$.
\end{theorem}
\begin{proof}
Lemmas \ref{lmm5.3} and \ref{lmm5.4} show the existence of $F$  and $G$. It remains to show they are inverse to each other. 

For any $[\alpha]\in \pi_1^{Mor}(M,\mathcal{N}), 1\leq i,j\leq n$, we have 
$$G(F([\alpha]_{i,j})) = G ((-1)^{i+1} S^{[\alpha]}_{\bar{i}, j})
= (-1)^{i+1}d_n (-1)^{\bar{i}+1}[\alpha]_{i,j} = [\alpha]_{i,j}.$$
Thus $G\circ F = Id_{\Gamma_n(M)}$ since $[\alpha]_{i,j}$,  $[\alpha]\in \pi_1^{Mor}(M,\mathcal{N}), 1\leq i,j\leq n$,  generate $\Gamma_n(M)$ as an algebra.

Obviously all the stated framed oriented  boundary arcs generate $S_n(M,\mathcal{N},1)$ as an algebra. For any stated oriented framed boundary arc $\alpha_{i,j}\in S_n(M,\mathcal{N},1)$, we have
$$F(G(\alpha_{i,j})) = d_n^{h(\tilde{\alpha})+1}(-1)^{i+1}F([\alpha]_{\bar{i},j})
= d_n^{h(\tilde{\alpha})+1}(-1)^{i+1} (-1)^{\bar{i}+1}S^{[\alpha]}_{i,j} = d_n^{h(\tilde{\alpha})}S^{[\alpha]}_{i,j} = \alpha_{i,j}. $$
Thus $F\circ G = Id_{S_n(M,\mathcal{N},1)}.$

\end{proof}

\begin{rem}
Korinman and Murakami proved the isomorphism between $S_2(M,\mathcal{N},1)$ and $\Gamma_2(M)$  using a  different technique
\cite{korinman2022relating}.
\end{rem}

\subsection{Adding one extra marking to marked 3-manifold}

In this subsection, we will investigate the effects on $S_n(M,\mathcal{N},1)$ when we put one extra marking on $\partial M$. 

\begin{definition}\label{dddddddd}
Let $(M,\mathcal{N})$ be a marked 3-manifold. We say that $\N'$  is obtained from $\N$ by adding one extra marking if
$\N'= \N\cup e$ where $e$ is an embedded oriented open interval in $\partial M$ such that $cl(e)\cap cl(\N) =\emptyset$. We call the linear map 
 $S_n(M,\mathcal{N},v)\rightarrow S_n(M,\mathcal{N}',v)$, induced by embedding $(M,\mathcal{N})\rightarrow (M,\mathcal{N}')$, adding marking map.
Obviously this map is an algebra homomorphism when $v=1$. We will use $\lambda_{ad}^{e}$ to denote the adding marking map. We can omit the superscript when there is no confusion with marking $e$.
\end{definition}

\begin{rem}\label{rem5.8}
Suppose $(M,\mathcal{N})$ is a marked 3-manifold with $\N\neq\emptyset$, and $\N'$ is obtained from $\N$ by adding one extra marking $e$.
As in Remark \ref{rrrmmm}, we can extand the relative spin structure $h$ for $\MN$ to a relative spin structure for $(M,\N')$, which is still denoted as $h$.

 Let $\alpha$ be an oriented path connecting $\N$ and $e$. We require $\alpha$ does not intersect itself, $\alpha$ only intersects $\partial M$ at its endpoints, and $\alpha(0)$ belongs to a component $e_1\subset \N$, and $\alpha(1)\in e$.
 Then we give a framing to $\alpha$ to obtain a  framed oriented boundary arc in $(M,\mathcal{N}')$, which is still denoted as $\alpha$, such that $h(\tilde{\alpha}) = 0$. Similarly we give a framing to $\alpha^{-1}$ such that
$h(\widetilde{\alpha^{-1}}) = 0$. Then we have $\alpha_{i,j} = S^{[\alpha]}_{i,j}, \alpha^{-1}_{i,j}
= S^{[\alpha^{-1}]}_{i,j},1\leq i,j\leq n.$ From Proposition \ref{prop5.2}, we know $AS^{[\alpha]}AS^{[\alpha^{-1}]} = AS^{[\alpha^{-1}]}AS^{[\alpha]} = AS^{[o]} = I.$
\end{rem}

We can regard $S_n(M,\mathcal{N}',1)$ as an $S_n(M,\mathcal{N},1)$-algebra because of the adding marking map. 

\begin{lemma}\label{lmm5.8}
Suppose $(M,\mathcal{N})$ is a marked 3-manifold with $\N\neq\emptyset$, and $\N'$ is obtained from $\N$ by adding one extra marking $e$. Then as an $S_n(M,\mathcal{N},1)$-algebra, $S_n(M,\mathcal{N}',1)$ is generated by $\alpha_{i,j},1\leq i,j\leq n$.
\end{lemma}
\begin{proof}
Let $T$ be the $S_n(M,\mathcal{N},1)$-subalgebra of $S_n(M,\mathcal{N}',1)$ generated by $\alpha_{i,j},1\leq i,j\leq n$. 
Since det$(S^{[\alpha]}) = 1\in S_n(M,\mathcal{N}',1)$, we have $(S^{[\alpha]})^{-1}$ is well-defined and 
$[(S^{[\alpha]})^{-1}]_{i,j}, 1\leq i,j\leq n,$  are polynomials in $\alpha_{i,j},1\leq i,j\leq n$. Especially  $[(S^{[\alpha]})^{-1}]_{i,j}\in T , 1\leq i,j\leq n$.
We know $S^{[\alpha^{-1}]}= A^{-1} (S^{[\alpha]})^{-1} A^{-1} $, which implies 
$\alpha^{-1}_{i,j}
= S^{[\alpha^{-1}]}_{i,j}\in T,1\leq i,j\leq n.$

From relation (\ref{wzh.five}), we know, as an $S_n(M,\mathcal{N},1)$-algebra, $S_n(M,\mathcal{N}',1)$ is generated by stated framed oriented boundary arcs with at least one end point in $e$. Suppose $\beta_{i,j}$ is  such  an arc in $S_n(M,\mathcal{N}',1)$. Recall that $\beta_{i,j} = d_n^{h(\tilde{\beta})}S^{[\beta]}_{i,j}$.

For the case when  $\beta(0),\beta(1)\in e$, we have 
$$A S^{[\beta]} = AS^{[\alpha]} AS^{[\alpha^{-1}*\beta*\alpha]} AS^{[\alpha^{-1}]}$$
where $[\alpha^{-1}*\beta*\alpha]$ is a path with two end points in $\N$. Especially we get
$$S^{[\beta]} = S^{[\alpha]} AS^{[\alpha^{-1}*\beta*\alpha]} AS^{[\alpha^{-1}]}.$$ Then 
$S^{[\beta]}_{i,j}\in T,1\leq i,j\leq n$, because $S^{[\alpha]}_{i,j}, S^{[\alpha^{-1}*\beta*\alpha]}_{i,j},
S^{[\alpha^{-1}]}_{i,j}\in T ,1\leq i,j\leq n.$ Thus $\beta_{i,j} = d_n^{h(\tilde{\beta})}S^{[\beta]}_{i,j}\in T$.

For the other cases, we can use the same way to show $\beta_{i,j} = d_n^{h(\tilde{\beta})}S^{[\beta]}_{i,j}\in T$.
Thus $T = S_n(M,\mathcal{N}',1).$

\end{proof}

Recall that
$$O(SL_n)=\mathbb{C}[x_{i,j}\mid 1\leq i,j\leq n]/(\text{det}(X) = 1)$$
where $X$ is an $n$ by $n$ matrix such that $X_{i,j} = x_{i,j},1\leq i,j\leq n.$ 
We have $X^{-1}$ makes sense and is an $n$ be $n$ matrix in $O(SL_n)$ because det$(X) = 1$. For $1\leq i,j\leq n$, We use $x^{-1}_{i,j}$ to denote $(X^{-1})_{i,j}$.
Obviously $S_n(M,\mathcal{N},1)\otimes O(SL_n)$ is an $S_n(M,\mathcal{N},1)$-algebra, and as an $S_n(M,\mathcal{N},1)$-algebra,
$$S_n(M,\mathcal{N},1)\otimes O(SL_n) = S_n(M,\mathcal{N},1)[x_{i,j}\mid 1\leq i,j\leq n]/(\text{det}(X) = 1)$$ 
by regarding $1\otimes x_{i,j}$ as $x_{i,j}$.

\begin{lemma}\label{lmm5.9}
Suppose $(M,\mathcal{N})$ is a marked 3-manifold with $\N\neq\emptyset$, and $\N'$ is obtained from $\N$ by adding one extra marking $e$. Then there exists an $S_n(M,\mathcal{N},1)$-algebra homomorphism
\begin{align*}
 \imath:S_n(M,\mathcal{N},1)\otimes O(SL_n)&\rightarrow S_n(M,\mathcal{N}',1)\\
 1\otimes x_{i,j}&\mapsto (A S^{[\alpha]})_{i,j}.
\end{align*}
\end{lemma}
\begin{proof}
Since  $S_n(M,\mathcal{N}',1)$ is a commutative $S_n(M,\mathcal{N},1)$-algebra and det$(AS^{[\alpha]}) = 1\in S_n(M,\mathcal{N}',1)$, then
$\imath$ is a well-defined $S_n(M,\mathcal{N},1)$-algebra homomorphism. 
\end{proof}

Next we try to define an $S_n(M,\mathcal{N},1)$-algebra homormorphism $$\jmath:S_n(M,\mathcal{N}',1)\rightarrow 
S_n(M,\mathcal{N},1)\otimes O(SL_n).$$

Let $l$ be a stated $n$-web in $(M,\mathcal{N}')$, and $s_{l}$ be the state of $l$. If $l\cap e =\emptyset$, define 
$\jmath(l) = l\otimes 1\in S_n(M,\mathcal{N},1)\otimes O(SL_n)$.

 If $l\cap e\neq \emptyset$, we 
suppose $|l\cap e| = m$,  and label the ends of $l$ on $e$ from $1$ to $m$. We use $E_k$ to denote the end of $l$ at $e$ labeled by the number $k, 1\leq k\leq m.$
Define
\begin{align*}
\begin{split}
 f_k&= \left \{
 \begin{array}{ll}
     -1,                    & \text{if }E_k\text{ points towards }e,\\
     1,     & \text{if }E_k\text{ points out of }e,\\
 \end{array}
 \right.\\
 g_k&= \left \{
 \begin{array}{ll}
     Id\in S_n,                    & \text{if }E_k\text{ points towards }e,\\
     \delta\in S_n,     & \text{if }E_k\text{ points out of }e,\\
 \end{array}
 \right.\\
 h_k (i,j)&= \left \{
 \begin{array}{ll}
     i,j,                    & \text{if }E_k\text{ points towards }e,\\
     j,i,     & \text{if }E_k\text{ points out of }e,\\
 \end{array}
 \right.
 \end{split}
 \end{align*}
where $\delta(\lambda) = \bar{\lambda},1\leq \lambda\leq n, 1\leq i,j\leq n, 1\leq k\leq m.$

We can connect $E_k$ with $\alpha^{-1}$ or $\alpha$ by the following way:
Suppose $E_k$ points towards $e$. First we isotope $\alpha^{-1}$ by moving $\alpha^{-1}(0)$ along $e$ to meet the end $E_k$. Then we isotope $l,  \alpha^{-1}$ nearby their endpoints at $e$ such that they are both in good position with respect to $e$.  Then we  connect $E_k$ with $\alpha^{-1}$. When $E_k$ points out of $e$, we can use the  same way to connect $E_k$ with $\alpha$.


Then we try to define an element $l(\alpha_{j_1}^{f_1},\alpha_{j_2}^{f_2},\dots, \alpha_{j_m}^{f_m})\in S_n(M,\mathcal{N},1)$ by the following way: For each $1\leq k\leq m$ we connect $E_k$ with $\alpha^{f_k}$, and assign the state
$j_k$ to the other end of $\alpha^{f_k}$ that is not used to connect $E_k$. During the process of connecting each $E_k$ and $\alpha^{f_k}$, we can isotope $\alpha^{f_k}$ such that 
$l(\alpha_{j_1}^{f_1},\alpha_{j_2}^{f_2},\dots, \alpha_{j_m}^{f_m})$ does not intersect itself. After connecting each   $E_k$ and $\alpha^{f_k}$, we can isotope the parts nearby the connecting points  such that $l(\alpha_{j_1}^{f_1},\alpha_{j_2}^{f_2},\dots, \alpha_{j_m}^{f_m})$ only intersects $\partial M$ at its endpoints. Then
  $l(\alpha_{j_1}^{f_1},\alpha_{j_2}^{f_2},\dots, \alpha_{j_m}^{f_m})\in S_n(M,\mathcal{N},1)$. Obviously 
$l(\alpha_{j_1}^{f_1},\alpha_{j_2}^{f_2},\dots, \alpha_{j_m}^{f_m})$ is a well-defined element in $S_n(M,\mathcal{N},1)$. 
We define 
$$\jmath(l) = \sum_{1\leq j_1,\dots,j_m\leq n}c_{g_1(j_1)}\dots
c_{g_m(j_m)}\; l(\alpha_{j_1}^{f_1},\dots, \alpha_{j_m}^{f_m})\otimes
\mu(\alpha^{-f_1}_{h_1(i_1,\overline{j_1})})\dots
\mu(\alpha^{-f_m}_{h_m(i_m,\overline{j_m})})$$
where $\mu(\alpha_{i,j}) = d_n (-1)^{i+1} x_{\bar{i},j}, \mu(\alpha^{-1}_{i,j}) = d_n (-1)^{i+1} x^{-1}_{\bar{i},j}$, $i_k = s_l(E_k)$, $1\leq k\leq m,$ $c_{t} = (-1)^{n-t}, 1\leq t\leq n$.

Note that if  $l_1$ and $l_2$ are isotopic to each other, we have $$l_1(\alpha_{j_1}^{f_1},\alpha_{j_2}^{f_2},\dots, \alpha_{j_m}^{f_m})= l_2(\alpha_{j_1}^{f_1},\alpha_{j_2}^{f_2},\dots, \alpha_{j_m}^{f_m})\in S_n(M,\mathcal{N},1)$$ where the labelings of endpoints of 
$l_i$, $i=1,2$, on $e$ are preserved by the isotopy. Then $\jmath$ respects isotopy classes, that is, $\jmath$
is defined on the set of  isotopy classes of  stated $n$-webs. For any two stated $n$-webs $l_1,l_2$, we isotope $l_t,t=1,2,$ such that $l_1\cap l_2 =\emptyset$, then we have $\jmath(l_1\cup l_2) = \jmath(l_1)\jmath(l_2)$.

\begin{lemma}\label{lmm5.10}
Suppose $(M,\mathcal{N})$ is a marked 3-manifold with $\N\neq\emptyset$, and $\N'$ is obtained from $\N$ by adding one extra marking $e$. Then $\jmath:S_n(M,\mathcal{N}',1)\rightarrow 
S_n(M,\mathcal{N},1)\otimes O(SL_n)$ is a well-defined 
$S_n(M,\mathcal{N},1)$-algebra homormorphism.
\end{lemma}
\begin{proof}
From the above discussion, it suffices to show $\jmath$ preserves  relations 
(\ref{w.cross})-(\ref{wzh.eight}) for well-definedness. 

It is obvious that $\jmath$ preserves relations (\ref{w.cross})-(\ref{wzh.four}), and  (\ref{wzh.eight}).

It is obvious that $\jmath$ preserves relations (\ref{wzh.five})-(\ref{wzh.seven}) if the boundary component in the picture is not $e$. Then we suppose the boundary component in  these pictures is $e$.
We use $l$ (respectively $l'$) to denote the left handside (respectively right handside) of "=" in these relations.

Relation (\ref{wzh.five}):  We only prove the case where the white dot represents an arrow going from left to right, that is, all the arrows point towards $e$.
We choose a labeling for endpoints of $l$ on $e$.
From bottom to top, we label the endpoints in  the right picture from $m+1$ to $m+n$. The other endpoints of $l'$ not in the picture are labeled in the same way as $l$.

Then we have
\begin{eqnarray*}
&&\jmath(l') = \sum_{\sigma\in S_n}(-1)^{\ell(\sigma)}
\sum_{\substack{1\leq j_1,\dots,j_m\leq n\\1\leq k_1,\dots,k_n\leq n}}c_{g_1(j_1)}\dots
c_{g_m(j_m)} c_{k_1}\dots c_{k_n}\\ 
&&l'(\alpha_{j_1}^{f_1},\dots, \alpha_{j_m}^{f_m},\alpha^{-1}_{k_1},\dots,\alpha^{-1}_{k_n}))\otimes
\mu(\alpha^{-f_1}_{h_1(i_1,\overline{j_1})}, \dots
\mu(\alpha^{-f_m}_{h_m(i_m,\overline{j_m})}) x_{\sigma(1),\overline{k_1}}\dots x_{\sigma(n),\overline{k_n}}\\
&&= 
\sum_{\substack{1\leq j_1,\dots,j_m\leq n\\1\leq k_1,\dots,k_n\leq n}}c_{g_1(j_1)}\dots
c_{g_m(j_m)} c_{k_1}\dots c_{k_n}\\ 
&&l'(\alpha_{j_1}^{f_1},\dots, \alpha_{j_m}^{f_m},\alpha^{-1}_{k_1},\dots,\alpha^{-1}_{k_n}))\otimes
\mu(\alpha^{-f_1}_{h_1(i_1,\overline{j_1})}, \dots
\mu(\alpha^{-f_m}_{h_m(i_m,\overline{j_m})})
\sum_{\sigma\in S_n}(-1)^{\ell(\sigma)} x_{\sigma(1),\overline{k_1}}\dots x_{\sigma(n),\overline{k_n}}\\
&&= 
\sum_{\substack{1\leq j_1,\dots,j_m\leq n\\1\leq k_1,\dots,k_n\leq n}}c_{g_1(j_1)}\dots
c_{g_m(j_m)} c_{k_1}\dots c_{k_n}\\ 
&&l'(\alpha_{j_1}^{f_1},\dots, \alpha_{j_m}^{f_m},\alpha^{-1}_{k_1},\dots,\alpha^{-1}_{k_n}))\otimes
\mu(\alpha^{-f_1}_{h_1(i_1,\overline{j_1})}, \dots
\mu(\alpha^{-f_m}_{h_m(i_m,\overline{j_m})})\text{det}
\begin{pmatrix}
x_{1,\overline{k_1}}&\dots&x_{1,\overline{k_n}}\\
\vdots& &\vdots\\
x_{n,\overline{k_1}}&\dots&x_{n,\overline{k_n}}\\
\end{pmatrix}\\
&&= (-1)^{\frac{n(n-1)}{2}}
\sum_{\substack{1\leq j_1,\dots,j_m\leq n\\\tau\in S_n}}c_{g_1(j_1)}\dots
c_{g_m(j_m)} \\ 
&&l'(\alpha_{j_1}^{f_1},\dots, \alpha_{j_m}^{f_m},\alpha^{-1}_{\tau(1)},\dots,\alpha^{-1}_{\tau(n)}))\otimes
\mu(\alpha^{-f_1}_{h_1(i_1,\overline{j_1})}, \dots
\mu(\alpha^{-f_m}_{h_m(i_m,\overline{j_m})})\text{det}
\begin{pmatrix}
x_{1,\overline{\tau(1)}}&\dots&x_{1,\overline{\tau(n)}}\\
\vdots& &\vdots\\
x_{n,\overline{\tau(1)}}&\dots&x_{n,\overline{\tau(n)}}\\
\end{pmatrix}\\
&&= 
\sum_{\substack{1\leq j_1,\dots,j_m\leq n\\\tau\in S_n}} (-1)^{\ell(\tau)} c_{g_1(j_1)}\dots
c_{g_m(j_m)} \\
&&l'(\alpha_{j_1}^{f_1},\dots, \alpha_{j_m}^{f_m},\alpha^{-1}_{\tau(1)},\dots,\alpha^{-1}_{\tau(n)}))\otimes
\mu(\alpha^{-f_1}_{h_1(i_1,\overline{j_1})}, \dots
\mu(\alpha^{-f_m}_{h_m(i_m,\overline{j_m})})\\
&&= 
\sum_{\substack{1\leq j_1,\dots,j_m\leq n}} c_{g_1(j_1)}\dots
c_{g_m(j_m)} 
l(\alpha_{j_1}^{f_1},\dots, \alpha_{j_m}^{f_m})\otimes
\mu(\alpha^{-f_1}_{h_1(i_1,\overline{j_1})}, \dots
\mu(\alpha^{-f_m}_{h_m(i_m,\overline{j_m})}) = \jmath(l).\\
\end{eqnarray*}

Relation (\ref{wzh.six}): Here we only prove $\jmath$
preserves 
$\raisebox{-.20in}{
\begin{tikzpicture}
\tikzset{->-/.style=
{decoration={markings,mark=at position #1 with
{\arrow{latex}}},postaction={decorate}}}
\filldraw[draw=white,fill=gray!20] (-0.7,-0.7) rectangle (0,0.7);
\draw [line width =1.5pt,decoration={markings, mark=at position 1 with {\arrow{>}}},postaction={decorate}](0,0.7)--(0,-0.7);
\draw [color = black, line width =1pt] (0 ,0.3) arc (90:270:0.5 and 0.3);
\node [right]  at(0,0.3) {$i$};
\node [right] at(0,-0.3){$j$};
\draw [line width =1pt,decoration={markings, mark=at position 0.5 with {\arrow{>}}},postaction={decorate}](-0.5,0.02)--(-0.5,-0.02);
\end{tikzpicture}}   = \delta_{\bar j,i }\,  (-1)^{n-i}$. We label the top endpoint by $1$ and the other one by 2.  
Then we  have
\begin{align*}
\jmath( \raisebox{-.20in}{
\begin{tikzpicture}
\tikzset{->-/.style=
{decoration={markings,mark=at position #1 with
{\arrow{latex}}},postaction={decorate}}}
\filldraw[draw=white,fill=gray!20] (-0.7,-0.7) rectangle (0,0.7);
\draw [line width =1.5pt,decoration={markings, mark=at position 1 with {\arrow{>}}},postaction={decorate}](0,0.7)--(0,-0.7);
\draw [color = black, line width =1pt] (0 ,0.3) arc (90:270:0.5 and 0.3);
\node [right]  at(0,0.3) {$i$};
\node [right] at(0,-0.3){$j$};
\draw [line width =1pt,decoration={markings, mark=at position 0.5 with {\arrow{>}}},postaction={decorate}](-0.5,0.02)--(-0.5,-0.02);
\end{tikzpicture}})&=\sum_{1\leq j_1,j_2\leq n}(-1)^{j+\overline{j_1}}c_{\overline{j_1}}c_{j_2} l(\alpha_{j_1},\alpha_{j_2}^{-1})\otimes x^{-1}_{j_1,i} x_{\overline{j},\overline{j_2}}\\
&=\sum_{1\leq j_1,j_2\leq n}(-1)^{j+\overline{j_1}}c_{\overline{j_1}}c_{j_2}
 \raisebox{-.20in}{
\begin{tikzpicture}
\tikzset{->-/.style=
{decoration={markings,mark=at position #1 with
{\arrow{latex}}},postaction={decorate}}}
\filldraw[draw=white,fill=gray!20] (-0.7,-0.7) rectangle (0,0.7);
\draw [line width =1.5pt,decoration={markings, mark=at position 1 with {\arrow{>}}},postaction={decorate}](0,0.7)--(0,-0.7);
\draw [color = black, line width =1pt] (0 ,0.3) arc (90:270:0.5 and 0.3);
\node [right]  at(0,0.3) {$j_1$};
\node [right] at(0,-0.3){$j_2$};
\draw [line width =1pt,decoration={markings, mark=at position 0.5 with {\arrow{>}}},postaction={decorate}](-0.5,0.02)--(-0.5,-0.02);
\end{tikzpicture}}
\otimes x^{-1}_{j_1,i} x_{\overline{j},\overline{j_2}}\\
&=(-1)^{j+1}\sum_{1\leq j_1\leq n}1\otimes x^{-1}_{j_1,i} x_{\overline{j},j_1}
=(-1)^{j+1}\delta_{\bar{j},i}1\otimes 1\\&= \delta_{\bar j,i }\,  (-1)^{n-i} 1\otimes 1.\\
\end{align*}

Relation (\ref{wzh.seven}): Here we only prove the case where the white dot represents an arrow going from right to left. We choose a labeling for endpoints of $l$ on $e$.
We label the top (respectively bottom) endpoint in the right picture by $m+1$ (respectively $m+2$). The other endpoints of $l'$ not in the picture are labeled in the same way as $l$. Then we have 
\begin{align*}
\jmath(l') &= \sum_{1\leq i\leq n}(-1)^{i+1}
\sum_{\substack{1\leq j_1,\dots,j_m\leq n\\1\leq k_1,k_2\leq n}}c_{g_1(j_1)}\dots
c_{g_m(j_m)} c_{\overline{k_1}} c_{k_2} (-1)^{\bar{i}+\overline{k_1}}\\ 
&l'(\alpha_{j_1}^{f_1},\dots, \alpha_{j_m}^{f_m},\alpha_{k_1},\alpha^{-1}_{k_2}))\otimes
\mu(\alpha^{-f_1}_{h_1(i_1,\overline{j_1})}) \dots
\mu(\alpha^{-f_m}_{h_m(i_m,\overline{j_m})}) x_{i,\overline{k_2}} x^{-1}_{k_1,i}\\
& = 
\sum_{\substack{1\leq j_1,\dots,j_m\leq n\\1\leq k_1,k_2\leq n}}c_{g_1(j_1)}\dots
c_{g_m(j_m)}  c_{k_2}\\ 
&l'(\alpha_{j_1}^{f_1},\dots, \alpha_{j_m}^{f_m},\alpha_{k_1},\alpha^{-1}_{k_2}))\otimes
\mu(\alpha^{-f_1}_{h_1(i_1,\overline{j_1})}) \dots
\mu(\alpha^{-f_m}_{h_m(i_m,\overline{j_m})}) \sum_{1\leq i\leq n} x_{i,\overline{k_2}} x^{-1}_{k_1,i}\\
& = 
\sum_{\substack{1\leq j_1,\dots,j_m\leq n\\1\leq k_1\leq n}}c_{g_1(j_1)}\dots
c_{g_m(j_m)}  (-1)^{k_1+1}\\
& l'(\alpha_{j_1}^{f_1},\dots, \alpha_{j_m}^{f_m},\alpha_{k_1},\alpha^{-1}_{\overline{k_1}}))\otimes
\mu(\alpha^{-f_1}_{h_1(i_1,\overline{j_1})}) \dots
\mu(\alpha^{-f_m}_{h_m(i_m,\overline{j_m})}) \\
& = 
\sum_{\substack{1\leq j_1,\dots,j_m\leq n}}c_{g_1(j_1)}\dots
c_{g_m(j_m)} 
 l'(\alpha_{j_1}^{f_1},\dots, \alpha_{j_m}^{f_m})\otimes
\mu(\alpha^{-f_1}_{h_1(i_1,\overline{j_1})}) \dots
\mu(\alpha^{-f_m}_{h_m(i_m,\overline{j_m})}) 
\\
& =\jmath(l).
\end{align*}

Then $\jmath$ is well-defined. Trivially it is an algebra homomorphism. For any $\alpha\in S_n(M,\mathcal{N},1)$, we have $\jmath(\lambda_{ad}(\alpha)) = \alpha\otimes 1$. Thus $\jmath$ is an $S_n(M,\mathcal{N},1)$-algebra homomorphism.


\end{proof}

\begin{theorem}\label{thm5.11}
Suppose $(M,\mathcal{N})$ is a marked 3-manifold with $\N\neq\emptyset$, and $\N'$ is obtained from $\N$ by adding one extra marking. Then there exist $S_n(M,\mathcal{N},1)$-algebra homomorphisms $\imath: S_n(M,\mathcal{N},1)\otimes
O(SL_n)\rightarrow S_n(M,\mathcal{N}',1)$ and $\jmath:S_n(M,\mathcal{N}',1)\rightarrow 
S_n(M,\mathcal{N},1)\otimes O(SL_n)$ such that $$\jmath\circ\imath =Id_{S_n(M,\mathcal{N},1)\otimes
O(SL_n)} ,\; \imath\circ\jmath = Id_{S_n(M,\mathcal{N}',1)}.$$ Especially $S_n(M,\mathcal{N}',1)\simeq
S_n(M,\mathcal{N},1)\otimes O(SL_n)$.
\end{theorem}
\begin{proof}
The existence of $\imath$ and $\jmath$ are given by Lemmas \ref{lmm5.9} and \ref{lmm5.10}.

Let $i,j\in\{1,2,\cdots,n\}$. Then we have 
\begin{align*}
\jmath(\imath(1\otimes x_{i,j})) &= (-1)^{i+1}\jmath(\alpha_{\bar{i},j})=(-1)^{i+1}\sum_{1\leq k\leq n}C_kd_n (-1)^{\bar{i}+1} \alpha(\alpha^{-1}_{k})\otimes x_{i,\bar{k}}\\
&= \sum_{1\leq k\leq n}(-1)^{n-k}
\raisebox{-.20in}{
\begin{tikzpicture}
\tikzset{->-/.style=
{decoration={markings,mark=at position #1 with
{\arrow{latex}}},postaction={decorate}}}
\filldraw[draw=white,fill=gray!20] (-0.7,-0.7) rectangle (0,0.7);
\draw [line width =1.5pt,decoration={markings, mark=at position 1 with {\arrow{>}}},postaction={decorate}](0,-0.7)--(0,0.7);
\draw [color = black, line width =1pt] (0 ,0.3) arc (90:270:0.5 and 0.3);
\node [right]  at(0,0.3) {$k$};
\node [right] at(0,-0.3){$j$};
\draw [line width =1pt,decoration={markings, mark=at position 0.5 with {\arrow{<}}},postaction={decorate}](-0.5,0.02)--(-0.5,-0.02);
\end{tikzpicture}}
\otimes x_{i,\bar{k}} = 1\otimes x_{i,j}.
\end{align*}
Since $1\otimes x_{i,j}$ are $S_n(M,\mathcal{N},1)$-algebra generators, we have $\jmath\circ\imath =Id_{S_n(M,\mathcal{N},1)\otimes
O(SL_n)}$.

We also have 
$$\imath(\jmath(\alpha_{i,j})) = (-1)^{\bar{i}+1}\imath(1\otimes x_{\bar{i},j}) = \alpha_{i,j}.$$
From Lemma \ref{lmm5.8}, we get $\imath\circ\jmath = Id_{S_n(M,\mathcal{N}',1)}$.
\end{proof}

\begin{theorem}\label{thm5.13}
Let $(M,\mathcal{N})$ be a marked 3-manifold. If $\N=\emptyset$, we have $S_n(M,\mathcal{N},1)\simeq G_n(M)$.
If $\N\neq \emptyset$, we have $S_n(M,\mathcal{N},1)\simeq \Gamma_n(M)\otimes O(SL_n)^{\otimes(\sharp \N-1)}$.
\end{theorem}
\begin{proof}
Subsection \ref{sub4.1}, Theorems \ref{thm5.5} and \ref{thm5.11}.
\end{proof}

\begin{corollary}\label{cli}
Let $(M,\mathcal{N})$ be a marked 3-manifold with $\N\neq \emptyset$. Suppose $\pi_1(M)$ is a free group generated by $m$ elements. Then we have 
 $S_n(M,\mathcal{N},1)\simeq  O(SL_n)^{\otimes(m + \sharp \N-1)}$.
\end{corollary}

The second conclusion in Theorem 7.13 in \cite{le2021stated} indicates the classical limit for essentially bordered pb surfaces, which coincides with Corollary \ref{cli}.


\begin{corollary}\label{Cor5.13}
Suppose $(M,\mathcal{N})$ is a marked 3-manifold, and $\N'$ is obtained from $\N$ by adding one extra marking.
Then the adding mark map $\lambda_{ad}:S_n(M,\mathcal{N},1)\rightarrow S_n(M,\mathcal{N}',1)$ is injective.
\end{corollary}
\begin{proof}
If $\N$ is empty. We look at the following diagram:
$$\begin{tikzcd}
S_n(M,\emptyset,1)  \arrow[r, "\lambda_{ad}"]
\arrow[d, "\simeq"]  
&  S_n(M,\mathcal{N}',1) \arrow[d, "G"] \\
G_n(M) \arrow[r, "\lambda "] 
&  \Gamma_n(M)\\
\end{tikzcd}$$
where the isomorphism from $S_n(M,\emptyset,1)$ to $G_n(M)$ is the one introduced in Subsection \ref{sub4.1} (the spin structure used for this isomorphism is the restriction of the relative spin structure for $(M,\mathcal{N}')$), and $\lambda$ is the embedding. It is easy to check the above diagram is commutative. Then $\lambda_{ad}$ is injective because $G$ is an isomorphism.

If $\N$ is not empty. For any $\alpha\in S_n(M,\mathcal{N},1)$, we have $\jmath(\lambda_{ad}(\alpha)) = \alpha\otimes 1$. Then $\lambda_{ad}$ is injective because $\jmath$ is an isomorphism and the map from 
$S_n(M,\mathcal{N},1)$ to $S_n(M,\mathcal{N},1)\otimes O(SL_n)$ given by $\alpha\mapsto \alpha\otimes 1$ is injective.
\end{proof}

We can regard $S_n(M,\mathcal{N},1)$
as a subalgebra of $S_n(M,\mathcal{N}',1)$ because of
 Corollary \ref{Cor5.13} and the adding marking map.

Suppose  the components of $\N$ consist of $e_0,e_1,\dots,e_{k-1}$ where $k$ is a positive integer.
If $k\geq 2$,
for each $1\leq t\leq k-1$, let $\alpha_t$ be an oriented path connecting $e_0$ and $e_t$ with $\alpha_t(0)\in e_0$ and $\alpha_t(1)\in e_t$. We use $[o]$ to denote the identity element in $\pi_1(M,e_0)$. 
The following Theorem offers algebraic generators and relations among these generators for the commutative algebra $S_n(M,\mathcal{N},1)$.
\begin{theorem}
The commutative algebra $S_n(M,\mathcal{N},1)$ is generated by $$S^{[\alpha]}_{i,j},\;[\alpha]\in \pi_1(M,e_0)
\cup \{[\alpha_1],\dots,[\alpha_{k-1}]\}, 1\leq i,j\leq n,$$ subject to the following relations.
\begin{equation}\label{eqqq}
	\begin{split}
		det(S^{[\alpha]}) = 1&\text{ for all } [\alpha] \in \pi_1(M,e_0)
		\cup \{[\alpha_1],\dots,[\alpha_{k-1}]\}, \;A S^{[o]} = I, \\
		&AS^{[\beta]}A S^{[\eta]} =A S^{[\beta*\eta]}\text{ for all } [\beta],[\eta]\in\pi_1(M,e_0).
	\end{split}
\end{equation}
Note that if $k=1$, the set $\{[\alpha_1],\dots,[\alpha_{k-1}]\}$ is empty.
\end{theorem}
\begin{proof}
	Theorems \ref{thm5.5} and \ref{thm5.11}.
\end{proof}

\subsection{Ker$\Phi$ = $\sqrt{0}$}

Suppose $\N$ has only one component.
Then we define an alegbra isomorphism $H:\Gamma_n(M)\rightarrow \Gamma_n(M)$ and 
a surjective algebra homormorphism $\tau:\Gamma_n(M)\rightarrow R_n(M,\mathcal{N})$.
Let $[\alpha]$ be an element in  $\pi_1^{Mor}(M,\mathcal{N})$ and 
$i,j$ be two integers between $1$ and $n$. Define $H([\alpha]_{i,j})
= [\alpha]_{\bar{i},\bar{j}}$. It is easy to show $H$ is a well-defined algebra isomorphism. For any $\rho\in\tilde{\chi}_n(M,\mathcal{N})$, define  $\tau([\alpha]_{i,j})(\rho) = [\rho(\tilde{\alpha})]_{i,j}$
where $\tilde{\alpha}\in \pi_1(UM, \tilde{\N})$ is a lift for $\alpha$ such 
that $h(\tilde{\alpha}) = 0$. From the proof of Proposition \ref{prop3.6}, we know the definition of $\tau$ is independent of the choice of the lift for $\alpha$. It is also obvious to show $\tau$ is a well-defined surjective algebra homomorphism. Especially from Proposition \ref{prop3.6} and definitions for $\Gamma_n(M)$ and $ R_n(M,\mathcal{N})$, we have Ker$\tau = \sqrt{0}_{\Gamma_n(M)}.$

\begin{lemma}\label{lmm5.15}
Let $(M,\mathcal{N})$ be marked 3-manifold with $\N$ consisting of one open oriented interval. Then we have the following commutative diagram:
$$\begin{tikzcd}
\Gamma_n(M)  \arrow[r, "F"]
\arrow[d, "H"]  
&  S_n(M,\mathcal{N},1) \arrow[d, "\Phi"] \\
\Gamma_n(M) \arrow[r, "\tau "] 
&  R_n(M,\mathcal{N})\\
\end{tikzcd}.$$ Especially Ker\,$\Phi=\sqrt{0}_{S_n(M,\mathcal{N},1)}.$
\end{lemma}
\begin{proof}
For any $[\alpha]\in\pi_1^{Mor}(M,\mathcal{N}),1\leq i,j\leq n$, we know $F([\alpha]_{i,j}) = (-1)^{i+1} \hat{\alpha}_{\bar{i},j}$
where $\hat{\alpha}$ is a framed oriented boundary arc such that $h(\widetilde{\hat{\alpha}})
= 0$ and $[\hat{\alpha}] = [\alpha]\in\pi_1^{Mor}(M,\mathcal{N})$. Then for any $\rho\in\tilde{\chi}_n(M,\mathcal{N})$, we have
$$\Phi(F([\alpha]_{i,j}))(\rho) =(-1)^{i+1} \Phi(\hat{\alpha}_{\bar{i},j})(\rho) =
 (-1)^{i+1} [A\rho(\widetilde{\hat{\alpha}})]_{i,\bar{j}} = [\rho(\widetilde{\hat{\alpha}})]_{\bar{i},\bar{j}}.$$
Since $\widetilde{\hat{\alpha}}\in \pi_1(UM,\tilde{\N})$ is a lift for $\alpha$ and $h(\widetilde{\hat{\alpha}})=0$,
we have 
$$\tau(H([\alpha]_{i,j}))(\rho) = \tau([\alpha]_{\bar{i},\bar{j}})(\rho)
= [\rho(\widetilde{\hat{\alpha}})]_{\bar{i},\bar{j}}.$$
Thus the diagram commutes.

Since both $F$ and $H$ are isomorphisms and Ker$\tau=\sqrt{0}_{\Gamma_n(M)}$, we get Ker$\Phi=\sqrt{0}_{S_n(M,\mathcal{N},1)}$.

\end{proof}

\begin{lemma}\label{lmm5.16}
Suppose $(M,\mathcal{N})$ is a marked 3-manifold with $\N\neq\emptyset$, and $\N'$ is obtained from $\N$ by adding one extra marking. 
Then Ker\,$\Phi^{(M,\mathcal{N}')}$ is the ideal generated by Ker\,$\Phi^{(M,\mathcal{N})}$ (here we regard $S_n(M,\mathcal{N},1)$ as a subalgebra of $S_n(M,\mathcal{N}',1)$). 
\end{lemma}
\begin{proof}
Here we  use the notations in Remark \ref{rem5.8}.

From Proposition \ref{prop3.6} and Lemma 8.1 in \cite{CL2022stated}, we know there is an algebra isomorphism 
$h: R_n(M,\mathcal{N})\otimes O(SL_n)\rightarrow R_n(M,\mathcal{N}')$ defined by
$$h(r\otimes x_{i,j})(\rho) = r(\rho|_{\pi_1(UM,\tilde{\N})})[\rho(\tilde{\alpha})]_{i,j}$$
where $r\in R_n(M,\mathcal{N}), 1\leq i,j\leq n,$ $\rho\in \tilde{\chi}_n(M,\mathcal{N}')$ and 
$\rho|_{\pi_1(UM,\tilde{\N})} $ is the restriction of $\rho$ on $\pi_1(UM,\tilde{\N})$.
We  have another algebra isomorphism $f:O(SL_n)\rightarrow O(SL_n)$ given by $x_{i,j}\rightarrow x_{\bar{i},\bar{j}}$.

Then we want to show the following diagram is commutative:
$$\begin{tikzcd}
S_n(M,\mathcal{N},1)\otimes O(SL_n)  \arrow[r, "\imath"]
\arrow[d, "\Phi^{(M,\mathcal{N})}\otimes f"]  
&  S_n(M,\mathcal{N}',1) \arrow[d, "\Phi^{(M,\mathcal{N}')}"] \\
R_n(M,\mathcal{N})\otimes O(SL_n) \arrow[r, "h"] 
&  R_n(M,\mathcal{N}')\\
\end{tikzcd}.$$
Let $\rho$ be element in $\tilde{\chi}_n(M,\mathcal{N}')$, $i,j$ be two integers between $1$ and $n$,
$\beta_{k,t}$ be a stated framed oriented boundary arc in $(M,\mathcal{N})$. We have 
$$
\Phi^{(M,\mathcal{N}')}(\imath(\beta_{k,t}\otimes 1))(\rho) = \Phi^{(M,\mathcal{N}')}(\beta_{k,t})(\rho)
= [A\rho(\tilde{\beta})]_{\bar{k},\bar{t}}\;,$$
and
\begin{align*}
(h\circ (\Phi^{(M,\mathcal{N})}\otimes f)) (\beta_{k,t}\otimes 1)(\rho)
&= h(\Phi^{(M,\mathcal{N})}(\beta_{k,t})\otimes 1)(\rho)  = \Phi^{(M,\mathcal{N})}(\beta_{k,t}) (\rho|_{\pi_1(UM,\tilde{\N})})\\
&= [A \rho|_{\pi_1(UM,\tilde{\N})} (\tilde{\beta})]_{\bar{k},\bar{t}} = [A\rho(\tilde{\beta})]_{\bar{k},\bar{t}}\, .
\end{align*}
We also have 
$$
\Phi^{(M,\mathcal{N}')}(\imath(1\otimes x_{i,j}))(\rho) = (-1)^{i+1} \Phi^{(M,\mathcal{N}')}(\alpha_{\bar{i},j})(\rho)
= (-1)^{i+1} [A \rho(\tilde{\alpha})]_{i,\bar{j}} =  [ \rho(\tilde{\alpha})]_{\bar{i},\bar{j}}\;,
$$
and 
$$(h\circ (\Phi^{(M,\mathcal{N})}\otimes f)) (1\otimes x_{i,j})(\rho) =
h(1\otimes x_{\bar{i},\bar{j}})(\rho) = [\rho(\tilde{\alpha})]_{\bar{i},\bar{j}}\, .$$
Thus the above diagram is commutative because $\beta_{k,t}\otimes 1, 1\otimes x_{i,j}$ generate
$S_n(M,\mathcal{N},1)\otimes O(SL_n)$ as an algebra and all the maps in the diagram are algebra homomorphisms.

We have Ker$(\Phi^{(M,\mathcal{N})}\otimes f) =$Ker($\Phi^{(M,\mathcal{N})}\otimes Id_{O(SL_n)}$)
$=(\text{Ker}\,\Phi^{(M,\mathcal{N})})\otimes O(SL_n)$, where $(\text{Ker}\,\Phi^{(M,\mathcal{N})})\otimes O(SL_n)$ is an ideal generated by $(\text{Ker}\,\Phi^{(M,\mathcal{N})})\otimes1$. Then Ker\,$\Phi^{(M,\mathcal{N}')}$is the ideal generated by
Ker\,$\Phi^{(M,\mathcal{N})}$ since $\imath((\text{Ker}\,\Phi^{(M,\mathcal{N})})\otimes1) =$Ker\,$\Phi^{(M,\mathcal{N})}$.

\end{proof}

\begin{lemma}\label{lmm5.17}
Let $(M,\mathcal{N})$ be a marked 3-manifold with $\N\neq \emptyset$. We have (a) Ker\,$\Phi^{(M,\mathcal{N})}
=\sqrt{0}_{S_n(M,\mathcal{N},1)}$, and (b) $\sqrt{0}_{S_n(M,\mathcal{N},1)}$ is the ideal generated by
$\sqrt{0}_{S_n(M,e,1)}$ where $e$ is a component of $\N$ (here we regard $S_n(M,e,1)$ as a subalgebra of $S_n(M,\mathcal{N},1)$).
\end{lemma}
\begin{proof}
 
For any marked 3-manifold $(M_1,\mathcal{N}_1)$ with $\N_1\neq \emptyset$, suppose $\N_1'$ is obtained from $\N_1$ by adding one extra marking. Then Ker\,$\Phi^{(M_1,\mathcal{N}_1')} = (\text{Ker\,}\Phi^{(M_1,\mathcal{N}_1)})$ from Lemma \ref{lmm5.16}. Thus if Ker\,$\Phi^{(M_1,\mathcal{N}_1)} =\sqrt{0}_{S_n(M_1,\mathcal{N}_1,1)}$, then 
$\text{Ker\,}\Phi^{(M_1,\mathcal{N}_1')} = (\text{Ker\,}\Phi^{(M_1,\mathcal{N}_1)})\subset \sqrt{0}_{S_n(M_1,\mathcal{N}_1',1)}$. We also have $\sqrt{0}_{S_n(M_1,\mathcal{N}_1',1)}\subset \text{Ker\,}\Phi^{(M_1,\mathcal{N}_1')}$ since the coordinate ring has no nonzero nilponents. Then $\text{Ker\,}\Phi^{(M_1,\mathcal{N}_1')}=\sqrt{0}_{S_n(M_1,\mathcal{N}_1',1)}$. Thus 
Ker\,$\Phi^{(M_1,\mathcal{N}_1)} =\sqrt{0}_{S_n(M_1,\mathcal{N}_1,1)}$ implies $\text{Ker\,}\Phi^{(M_1,\mathcal{N}_1')}=\sqrt{0}_{S_n(M_1,\mathcal{N}_1',1)}$. Combine with the fact that (a) is true if $\N$ consists of only one oriented open interval (Lemma \ref{lmm5.15}),  we get (a) is true for general marked 3-manifold $(M,\mathcal{N})$ with $\N\neq \emptyset$.

If $\N$ consists of one component, clearly (b) holds. If $\sharp \N>1$, suppose $Com(N)=\{e_1,e_2,\dots,e_m\}$.
For any $1\leq i\leq m$, define $\N_{(i)}=e_1\cup\dots\cup e_i$. Then we have 
$$S_n(M,\mathcal{N}_{(1)},1)\subset S_n(M,\mathcal{N}_{(2)},1)\subset,\dots,\subset S_n(M,\mathcal{N}_{(m)},1).$$
Since Ker\,$\Phi^{(M,\mathcal{N}_{(i+1)})}$ is an ideal of $S_n(M,\mathcal{N}_{(i+1)},1)$ generated by
Ker\,$\Phi^{(M,\mathcal{N}_{(i)})}$, then we have 
$$\text{Ker}\,\Phi^{(M,\mathcal{N})} = \text{Ker\,}\Phi^{(M,\mathcal{N}_{(m)})}$$
is an ideal of $S_n(M,\mathcal{N},1)$
generated by Ker\,$\Phi_{(M,\mathcal{N}_{(1)})}$, which actually is $\text{Ker\,}\Phi_{(M,e_1)}$. From (a), we know 
Ker\,$\Phi^{(M,\mathcal{N})}= \sqrt{0}_{S_n(M,\mathcal{N},1)}$ and Ker\,$\Phi^{(M,\{e_1\})}= \sqrt{0}_{S_n(M,e_1,1)}$. Since we can label any component of $\N$ as $e_1$, then (b) is true.
\end{proof}

\begin{theorem}
For any marked 3-manifold $(M,\mathcal{N})$,
 we have Ker\,$\Phi^{(M,\mathcal{N})}
=\sqrt{0}_{S_n(M,\mathcal{N},1)}$. 
\end{theorem}
\begin{proof}
Subsection \ref{sub4.1} and (a) in Lemma \ref{lmm5.17}.
\end{proof}

\section{Generalized marked 3-manifold}

Costantino and L{\^e} defined the generalized marked 3-manifold in \cite{CL2022TQFT}, in which they allow $\N$ contains oriented closed circles. For a generalized marked 3-manifold $\MN$, obviously $S_n(M,\N,1)$ has a commutative algebra structure, given by taking the disjoint union of the stated $n$-webs in $\MN$.
In this section, we will focus on the classical limit of the stated $SL_n$-skein module of the generalized marked 3-manifold.

\subsection{Cutting out the closure of a small open interval from $\N$}
Let $(M,\mathcal{N})$ be a generalized marked 3-manifold with $\N\neq \emptyset$. Suppose $U$ is a small open interval contained in $e$ such that $cl(U)\subset e$, where $e$ is a component of $\N$. Let $\N' = (\N\setminus e)\cup e'$ where $e' = e\setminus cl(U)$.
Let $l_{U}: S_n(M,\mathcal{N}',v)\rightarrow S_n(M,\mathcal{N},v)$ be the linear map induced by the embedding $(M,\mathcal{N}')\rightarrow (M,\mathcal{N})$. Clearly $l_U$ is surjective, and is an algebra homomorphism when $v=1$.

\begin{proposition}\label{pro7.1}
The  above map $l_U$ induces an isomorphism 
$$\bar{l}_U:S_n(M,\mathcal{N}',v)/\simeq\; \rightarrow S_n(M,\mathcal{N},v),$$
 where $\simeq$ is the equivalence relation  given by  the following picture:

\begin{align}\label{eq10}
\raisebox{-.35in}{
\begin{tikzpicture}
\tikzset{->-/.style=
{decoration={markings,mark=at position #1 with
{\arrow{latex}}},postaction={decorate}}}
\filldraw[draw=white,fill=gray!20] (-0.3,0) rectangle (2, 2);
\draw[line width =1.5pt,decoration={markings, mark=at position 1.0 with {\arrow{>}}},postaction={decorate}](2,0) --(2,0.5);
\draw [color = black, line width =1.5pt](2,0.5) --(2,0.9);
\draw[line width =1.5pt,decoration={markings, mark=at position 1.0 with {\arrow{>}}},postaction={decorate}](2,1.1) --(2,1.75);
\draw [color = black, line width =1.5pt](2,1.75) --(2,2);
\node[right] at(2,0.6) {\small $i$};
\filldraw[draw=black!80,fill=white!20] (1,0.8) circle(0.1);
\draw[color = black, line width =1pt] (0,1)--(0.9,0.82);
\draw[color = black, line width =1pt] (1.1,0.78)--(2,0.6);
\end{tikzpicture}}
\simeq
\raisebox{-.35in}{
\begin{tikzpicture}
\tikzset{->-/.style=
{decoration={markings,mark=at position #1 with
{\arrow{latex}}},postaction={decorate}}}
\filldraw[draw=white,fill=gray!20] (-0.3,0) rectangle (2, 2);
\draw[line width =1.5pt,decoration={markings, mark=at position 1.0 with {\arrow{>}}},postaction={decorate}](2,0) --(2,0.5);
\draw [color = black, line width =1.5pt](2,0.5) --(2,0.9);
\draw[line width =1.5pt,decoration={markings, mark=at position 1.0 with {\arrow{>}}},postaction={decorate}](2,1.1) --(2,1.75);
\draw [color = black, line width =1.5pt](2,1.75) --(2,2);
\node[right] at(2,1.4) {\small $i$};
\filldraw[draw=black!80,fill=white!20] (1,1.2) circle(0.1);
\draw[color = black, line width =1pt] (0,1)--(0.9,1.18);
\draw[color = black, line width =1pt] (1.1,1.22)--(2,1.4);
\end{tikzpicture}}\;.
\end{align}
The missing part between two arrows is $cl(U)$.

\end{proposition}\label{prop7.1}
\begin{proof}
Clearly $l_U$ induces a linear map $\bar{l}_U: S_n(M,\mathcal{N}',v)/\simeq\; \rightarrow S_n(M,\mathcal{N},v)$.
For a stated $n$-web $\alpha$ in $(M,\mathcal{N}')$, we use cls$(\alpha)$ to denote the element in $S_n(M,\mathcal{N}',v)/\simeq$ determined by $\alpha$. Let $\beta$ be any stated $n$-web in $(M,\mathcal{N})$. We can isotope $\beta$ such that $cl(U)\cap  \beta = \emptyset$, and define $j_U(\beta) =\text{cls}(\beta)\in S_n(M,\mathcal{N}',v)/\simeq$.
We have $j_U(\beta)$ is independent of  how we isotope $\beta$ because of relation (\ref{eq10}).
 If $\beta$ and $\beta'$ are isotopic stated $n$-webs in $(M,\mathcal{N})$, clearly we have $j_U(\beta)
=j_U(\beta')$ because of relation (\ref{eq10}). Trivially $j_U$ preserves the defining skein relations for $S_n(M,\mathcal{N},v)$. Thus $j_U$ is a well-defined linear map from $S_n(M,\mathcal{N},v)$ to $S_n(M,\mathcal{N}',v)/\simeq$. It is easy to check $l_U$ and $j_U$ are inverse to each other.
\end{proof}

\subsection{Classical limit for stated $SL_n$-skein modules for
 generalized marked 3-manifolds} In this subsection, we will try to find out the classical limit for generalized marked 3-manifolds  using results in Section \ref{subb5} and Proposition \ref{pro7.1}. In this subsection, we will  assume the 3-manifold 
is connected (the corresponding results can be easily  generalized to general  3-manifolds).
\begin{rem}\label{rem7.2}
Let $(M,\mathcal{N})$ be a generalized marked 3-manifold. Suppose $\N$ contains $k$ ($k\geq 1$) closed oriented circles, which are denoted as $e_0,e_1,\dots,e_{k-1}$. We denote other oriented open intervals in $\N$, if any, as $e_k,\dots,e_{m-1}$. For each $0\leq i\leq k-1$, we pick a small open interval $U_i$ contained in $e_i$ such that $cl(U_i)\subset e_i$, and set $e_i' = e_i\setminus cl(U_i)$.  
Set $e_i' = e_i$ for $k\leq i\leq m-1$.
Let
$\N' =\{e_0',e_1',\dots, e_{m-1}'\}$, then $(M,\mathcal{N}')$ is a circle free marked 3-manifold. We choose a relative spin structure $h$ for $(M,\mathcal{N}')$. For each $1\leq i\leq m-1$,
 let $\alpha_i$ be an oriented path connecting $e_0'$ and $e_i'$ such that $\alpha_i(0)\in e_0'$ and $\alpha_i(1)\in e_i'$.  We still use $l_U$ to denote the algebra homomorphism from $S_n(M,\mathcal{N}',1)$ to $S_n(M,\mathcal{N},1)$ induced by the embedding from $(M,\mathcal{N}')$ to $(M,\mathcal{N}).$
\end{rem}

 We know there is an isomorphism $L$ from $\Gamma_n(M)\otimes
O(SL_n)^{\otimes(m-1)}$ to $S_n(M,\mathcal{N}',1)$. For any element $y\in \Gamma_n(M)$, we will use
$y_{\otimes}$ to denote  $y\otimes 1\otimes\dots\otimes 1\in  \Gamma_n(M)\otimes
O(SL_n)^{\otimes(m-1)}$. For any $1\leq i,j\leq n, 1\leq t\leq m-1,$ we use $x_{i,j}^{t}$ to denote 
$1\otimes1\otimes\dots \otimes x_{i,j}\otimes \dots\otimes 1\in \Gamma_n(M)\otimes O(SL_n)^{\otimes(m-1)}$ where $x_{i,j}$ is in the $t$-th tensor factor for $O(SL_n)^{\otimes(m-1)}$. Then the isomorphism $L$ is given by:
$$([\alpha]_{i,j})_{\otimes}\rightarrow (A S^{[\alpha]})_{i,j}\text{\;and\;} x^{t}_{i,j}\rightarrow (A S^{[\alpha_t]})_{i,j}$$
where $[\alpha]\in \pi_1(M,e'_0),1\leq i,j\leq n, 1\leq t\leq m-1.$ For each $0\leq t\leq m-1$, 
set $X_t = (x^{t}_{i,j})_{n\times n}$. For each element $[\alpha]\in\pi_1(M,e_0')$, set 
$ Q_{[\alpha],\otimes} = (([\alpha]_{i,j})_{\otimes})_{n\times n}$. Then $L(X_t) = A S^{[\alpha_t]}, 
L(Q_{[\alpha],\otimes}) = A S^{[\alpha]}$.

As in Subsection \ref{sss3.3}, any component $e\in N$ can be lifted to $\tilde{e}\subset UM$. Note that when $e$ is an oriented closed circle, the lifting $\tilde{e}$ is also an oriented closed circle in $UM$, which means $\tilde{e}$ is an element in $H_1(UM)$.

\begin{definition}\label{df7.3}
Let $(M,\mathcal{N})$ be a generalized marked 3-manifold, and $h_{s}$ be a spin structure for $M$. Suppose $\N$ contains $k$ oriented closed circles. 

If $k=0$, we define $\Gamma_n(M,\mathcal{N}) = \Gamma_n(M)$.

 If $k\geq1$,   we denote all the oriented closed circles in $\N$  as $e_0,\dots, e_{k-1}$. For each $e_t$, we  use a path $\beta_t$ ($\beta_t(0)$ is the base point for $\pi_1(M)$ and $\beta_t(1)\in e_t$) to connect the base point for $\pi_1(M)$ and $e_t$ to obtain an element in $\pi_1(M)$,  denoted as $[e_t^{\beta_t}]$.
Define $\Gamma_n(M,\mathcal{N}) = \Gamma_n(M)/(D)$ where
$D = \{(Q_{[e_t^{\beta_t}]} - d_n^{h_s(\widetilde{e_t})}I)_{i,j}\mid 0\leq t\leq k-1, 1\leq i,j\leq n\}$ and $(D)$ is the ideal of $\Gamma_n(M)$ generated by $D$. For any element $x\in \Gamma_n(M)$, we use $\bar{x}$ to denote
$x+(D)\in \Gamma_n(M,\mathcal{N})$.
\end{definition}

Note that the definition of $\Gamma_n(M,\mathcal{N})$ is independent of the choice of $\beta_t,0\leq t\leq k-1$.
Suppose for each 
$0\leq t\leq k-1$, we make another choice $\gamma_t$. Then the relation $Q_{[e_t^{\beta_t}]} = d^{h_s(\widetilde{e_t})} I$ becomes $Q_{[e_t^{\gamma_t}]} = d^{h_s(\widetilde{e_t})} I$. Since $Q_{[e_t^{\beta_t}]}$  and $Q_{[e_t^{\gamma_t}]}$ are conjugate to each other, then the relation $Q_{[e_t^{\beta_t}]} = d^{h_s(\widetilde{e_t})} I$ is the same with the relation $Q_{[e_t^{\gamma_t}]} = d^{h_s(\widetilde{e_t})} I$. 

Note that we do not distinguish  $\pi_1(M)$ and $\pi_1^{Mor}(M,e_0)$ where $e_0$ is an embedded open interval in $\partial M$.
The definition for $\Gamma_n(M,\mathcal{N})$ is related to the spin structure $h_s$ for $M$.
Here we make a convention that the spin structure used for the definition of $\Gamma_n(M,\mathcal{N})$ is obtained by restricting the relative spin structure when the relative spin structure is given.

\begin{lemma}\label{lmm7.4}
With  the conventions and notations in Remark \ref{rem7.2}, we have 
$$\begin{tikzcd}
\Gamma_n(M)\otimes
O(SL_n)^{\otimes(m-1)} \arrow[r, "L"] & S_n(M,\mathcal{N}',1)\arrow[r, "l_U"] & S_n(M,\mathcal{N},1)
\end{tikzcd}$$
induces a surjective algebra homomorphism $\overline{L}:\Gamma_n(M,\mathcal{N})\otimes
O(SL_n)^{\otimes(m-1)} \rightarrow S_n(M,\mathcal{N},1)$. Here we regard $\pi_1(M)$ as $\pi_1^{Mor}(M,e_0')$.
\end{lemma}
\begin{proof}
We have the exact sequence: 
$$\begin{tikzcd}
(D) \arrow[r, tail] & \Gamma_n(M)\arrow[r,two heads] & \Gamma_n(M,\mathcal{N}) 
\end{tikzcd}$$
where $D$ and $(D)$ are defined in Definition \ref{df7.3} (the arrow with two heads means the corresponding map is surjective). After using functor $-\otimes O(SL_n)^{\otimes(m-1)}$ acting on the above exact sequence, we get the following new exact sequence:
$$\begin{tikzcd}
(D)\otimes O(SL_n)^{\otimes(m-1)} \arrow[r, tail] & \Gamma_n(M)\otimes O(SL_n)^{\otimes(m-1)}\arrow[r,two heads] & \Gamma_n(M,\mathcal{N})\otimes O(SL_n)^{\otimes(m-1)}
\end{tikzcd}.$$
Note that $(D)\otimes O(SL_n)^{\otimes(m-1)}$ is the ideal of 
$\Gamma_n(M)\otimes O(SL_n)^{\otimes(m-1)}$ generated by $d_{\otimes},d\in D$.
Thus to show $l_U\circ L$ induces $\overline{L}$, it suffices to show $l_U(L(d_{\otimes})) = 0$ for all $d\in D$.

Let $i,j$ be any two integers between $1$ and $n$, and $t$ be an integer between $0$ and $k-1$. Then we have 
$$l_U(L(([e_t^{\beta_t}]_{i,j})_{\otimes}))= (-1)^{i+1}l_U(S^{[e_t^{\beta_t}]}_{\bar{i},j}).$$
Thus we need to show $(-1)^{i+1}l_U(S^{[e_t^{\beta_t}]}_{\bar{i},j}) = d_n^{h(\widetilde{e_t})} \delta_{i,j}$, that is, to show $A\,l_U(S^{[e_t^{\beta_t}]}) =  d_n^{h(\widetilde{e_t})} I$. From the definition of $[e_t^{\beta_t}]$, we know $[e_t^{\beta_t}] = [\beta_t^{-1}*e_t*\beta_t]$. Then we have 
$$A\,l_U(S^{[\beta_t]} )A\,l_U(S^{[e_t^{\beta_t}]}) = l_U(A S^{[\beta_t]} A 
S^{[\beta_t^{-1}*e_t*\beta_t]}) = l_U(A S^{[e_t*\beta_t]}) =A\, l_U(S^{[e_t*\beta_t]}).$$
Note that in $S_n(M,\mathcal{N},1)$, we have $l_U(S^{[e_t*\beta_t]}) = d_n^{h(\widetilde{e_t})} l_U(S^{[\beta_t]})$. Then we get $$l_U(S^{[\beta_t]} )A\,l_U(S^{[e_t^{\beta_t}]}) = l_U(S^{[e_t*\beta_t]}) = d_n^{h(\widetilde{e_t})} l_U(S^{[\beta_t]}).$$
Then we have $A\,l_U(S^{[e_t^{\beta_t}]})= d_n^{h(\widetilde{e_t})} I$ because 
$l_U(S^{[\beta_t]} )$ is invertible.

The above discussion shows $l_U\circ L$ induces $\overline{L}$. The algebra homomorphism $\overline{L}$ is  surjective  since $l_U\circ L$ is  surjective.
\end{proof}

Note that for any $x\in \Gamma_n(M), y\in 
O(SL_n)^{\otimes(m-1)}$, we have $\overline{L}(\bar{x}\otimes y) =l_U(L(x\otimes y))$. 
We use $\pi$ to denote the projection from $\Gamma_n(M)\otimes
O(SL_n)^{\otimes(m-1)}$ to $\Gamma_n(M,\mathcal{N})\otimes
O(SL_n)^{\otimes(m-1)}$. Then $\overline{L}\circ \pi = l_U\circ L$.

\begin{lemma}\label{lmm7.5}
With  the conventions and notations in Remark \ref{rem7.2}, we have 
$$\begin{tikzcd}
S_n(M,\mathcal{N}',1)
 \arrow[r, "L^{-1}"] & \Gamma_n(M)\otimes
O(SL_n)^{\otimes(m-1)} \arrow[r, "\pi"] & \Gamma_n(M,\mathcal{N})\otimes
O(SL_n)^{\otimes(m-1)}
\end{tikzcd}$$
induces a surjective algebra homomorphism $\overline{L^{-1}}:S_n(M,\mathcal{N},1) \rightarrow \Gamma_n(M,\mathcal{N})\otimes
O(SL_n)^{\otimes(m-1)}$. Here we regard $\pi_1(M)$ as $\pi_1^{Mor}(M,e_0')$.
\end{lemma}
\begin{proof}
From Proposition \ref{pro7.1}, it suffices to show $\pi\circ L^{-1}$ preserves the equivalence relation (\ref{eq10}) for every $U_t$, $0\leq t\leq k-1$. Let $\alpha$ be any stated $n$-web for $(M,\mathcal{N}')$. Suppose there exists $0\leq t\leq k-1$ such that, nearby $U_t$, the stated $n$-web $\alpha$ looks like the left picture in the equivalence relation (\ref{eq10}). Let $\alpha'$ be the same stated $n$-web as $\alpha$ except, nearby  $U_t$, $\alpha'$ looks like the right picture in equivalence relation (\ref{eq10}). Then we want to show $\pi(L^{-1}(\alpha)) = \pi(L^{-1}(\alpha'))$.

We can use the same way  
 to kill all the sinks and sources in $\alpha$ and $\alpha'$. Then the resulting two stated $n$-webs only differ on a single stated arc. Since $\pi\circ L^{-1}$ is an algebra homomorphism, we can just assume $\alpha$ is a stated framed oriented boundary arc. Without loss of generality, we assume the white dot in equivalence relation (\ref{eq10}) represents an arrow pointing from left to right, that is, pointing towards the boundary. It is easy to show $h(\widetilde{\alpha'}) = h(\tilde{\alpha})+h(\widetilde{e_t})$. Suppose $s(\alpha(0)) = s(\alpha'(0)) = j$. We have $\alpha = d_n^{ h(\tilde{\alpha})} S^{[\alpha]}_{i,j}$ and 
$\alpha' = d_n^{ h(\widetilde{\alpha'}) }S^{[\alpha']}_{i,j}$.

Suppose $\alpha(0)\in e'_{t_1}$ where $0\leq t_1\leq m-1$.
We have four cases to consider: (1) $t=t_1=0,$ (2) $t=0$ and $t_1\neq 0$, (3) $t\neq 0$ and $t_1=0$, (4) $t\neq 0$ and $t_1\neq 0$.

 Here we only prove the case when $t\neq 0$ and $t_1\neq 0$. We have
\begin{align*} 
\pi(L^{-1}(S^{[\alpha]})) &=  A^{-1}\pi(L^{-1}(AS^{[\alpha_t][\alpha_t^{-1}*\alpha*\alpha_{t_1}][\alpha_{t_1}^{-1}]} )) = A^{-1}\pi(L^{-1}(AS^{[\alpha_t]} AS^{[\alpha_t^{-1}*\alpha*\alpha_{t_1}]} AS^{[\alpha_{t_1}^{-1}]} )) \\
&= A^{-1} \pi( X_t\, Q_{[\alpha_t^{-1}*\alpha *\alpha_{t_1}],\otimes} \,X_{t_1}^{-1})
=  A^{-1} \pi( X_t)\pi( Q_{[\alpha_t^{-1}*\alpha *\alpha_{t_1}],\otimes}) \pi( X_{t_1}^{-1}).
\end{align*}
Similarly we have 
$$\pi(L^{-1}(S^{[\alpha']}))
=  A^{-1} \pi( X_t)\pi( Q_{[\alpha_t^{-1}*\alpha' *\alpha_{t_1}],\otimes})  \pi(X_{t_1}^{-1}).$$
We also have 
$$Q_{[\alpha_t^{-1}*\alpha' *\alpha_{t_1}],\otimes}
= Q_{[\alpha_t^{-1}*\alpha *\alpha_{t_1}],\otimes}
Q_{[\alpha_{t_1}^{-1}*\alpha^{-1}*\alpha' *\alpha_{t_1}],\otimes}$$
where $[\alpha_{t_1}^{-1}*\alpha^{-1}*\alpha' *\alpha_{t_1}] = 
[(\alpha' *\alpha_{t_1})^{-1}*e_t*\alpha' *\alpha_{t_1}]$. Thus 
$$\pi(Q_{[\alpha_t^{-1}*\alpha' *\alpha_{t_1}],\otimes})=
\pi( Q_{[\alpha_t^{-1}*\alpha *\alpha_{t_1}],\otimes}) \pi(Q_{[\alpha_{t_1}^{-1}*\alpha^{-1}*\alpha' *\alpha_{t_1}],\otimes}) = d_n^{h(\widetilde{e_t})}\pi( Q_{[\alpha_t^{-1}*\alpha *\alpha_{t_1}],\otimes}).$$
Then 
\begin{align*}
\pi(L^{-1}(\alpha'))& = d_n^{ h(\widetilde{\alpha'}) } \pi L^{-1}(S^{[\alpha']}_{i,j})
= d_n^{  h(\tilde{\alpha})+h(\widetilde{e_t})} [A^{-1} \pi (X_t)\pi( Q_{[\alpha_t^{-1}*\alpha' *\alpha_{t_1}],\otimes})\pi( X_{t_1}^{-1})]_{i,j}\\
& = d_n^{  h(\tilde{\alpha})+h(\widetilde{e_t})} d_n^{h(\widetilde{e_t})} [A^{-1} \pi( X_t)\pi( Q_{[\alpha_t^{-1}*\alpha *\alpha_{t_1}],\otimes})\pi( X_{t_1}^{-1})]_{i,j}\\&=
d_n^{  h(\tilde{\alpha})} [A^{-1} \pi( X_t)\pi( Q_{[\alpha_t^{-1}*\alpha *\alpha_{t_1}],\otimes}) \pi(X_{t_1}^{-1})]_{i,j}= \pi(L^{-1}(\alpha)).
\end{align*}

\end{proof}

For any stated $n$-web $\alpha$ in $S_n(M,\mathcal{N},1)$, we can isotope $\alpha$ such that $cl(U_t)\cap \alpha = \emptyset$ for all
$0\leq t\leq k-1$. Then $\alpha$ is also a stated $n$-web $\alpha$ in $S_n(M,\mathcal{N}',1)$, we still use $\alpha$ to denote this element in  $S_n(M,\mathcal{N}',1)$. Then $\overline{L^{-1}}(\alpha) = \pi (L^{-1}(\alpha))$, that is, we have $\overline{L^{-1}}\circ l_U
= \pi\circ L^{-1}$.

\begin{lemma}\label{lmm7.6}
The algebra homomorphism $\overline{L}$ obtained in Lemma \ref{lmm7.4} and 
the algebra homomorphism $\overline{L^{-1}}$ obtained in Lemma \ref{lmm7.5} are inverse to each other. Especially for any generalized marked 3-manifold $(M,\mathcal{N})$ with $\N$ containing at least one closed oriented circle, we have $\Gamma_n(M,\mathcal{N})\otimes
O(SL_n)^{\otimes(\sharp \N-1)} \simeq S_n(M,\mathcal{N},1)$.
\end{lemma}
\begin{proof}
For any stated $n$-web $\alpha$ in $S_n(M,\mathcal{N},1)$, we isotope $\alpha$ such that $cl(U_t)\cap \alpha = \emptyset$ for all
$0\leq t\leq k-1$. Then 
$$\overline{L}(\overline{L^{-1}}(\alpha)) = \overline{L} (\pi (L^{-1}(\alpha)))=
(l_U\circ L)(L^{-1}(\alpha)) = l_U(\alpha) = \alpha .$$

For any $x\in \Gamma_n(M), y\in 
O(SL_n)^{\otimes(m-1)}$, we have 
$$\overline{L^{-1}}(\overline{L}(\bar{x}\otimes y)) =
\overline{L^{-1}}(l_U(L(x\otimes y))) = (\pi\circ L^{-1}) (L(x\otimes y))
= \pi(x\otimes y) =\bar{x}\otimes y.$$ 
\end{proof}

\begin{theorem}
Let $(M,\mathcal{N})$ be a generalized marked 3-manifold with $\N\neq \emptyset$.
Then $S_n(M,\mathcal{N},1)\simeq \Gamma_n(M,\mathcal{N})\otimes O(SL_n)^{\otimes(\sharp \N - 1)}$.
\end{theorem}
\begin{proof}
If $\N$ is circle free, then $\Gamma_n(M,\mathcal{N}) = \Gamma_n(M)$. From Theorem \ref{thm5.13}, we have 
$S_n(M,\mathcal{N},1)\simeq \Gamma_n(M,\mathcal{N})\otimes O(SL_n)^{\otimes(\sharp \N - 1)}$.

If $\N$ containes at least one oriented closed circle, then Lemma \ref{lmm7.6} shows 
$S_n(M,\mathcal{N},1)\simeq \Gamma_n(M,\mathcal{N})\otimes O(SL_n)^{\otimes(\sharp \N - 1)}$.
\end{proof}

For generalized marked 3-manifold $(M,\mathcal{N})$, we can also define the corresponding adding marking map. Suppose $\N_{ad}=\N\cup e$ where $e$ is an oriented open interval or an oriented closed circle such that there is no intersection between the closure of $\N$ and the closure of $e$. We also say $\N_{ad}$ is obtained from $\N$ by adding one extra marking. The linear map from $S_n(M,\mathcal{N},v)$ to $S_n(M,\mathcal{N}_{ad},v)$ induced by the embedding $(M,\mathcal{N})\rightarrow (M,\mathcal{N}_{ad})$ is also denoted as $\lambda_{ad}$. Clearly when $v=1$, we have $\lambda_{ad}$ is an algebra homomorphism.

\begin{corollary}
Let $(M,\mathcal{N})$ be a generalized marked 3-manifold. Suppose $\N_{ad}$ is obtained from $\N$ by adding one extra oriented open interval. Then $\lambda_{ad} :S_n(M,\mathcal{N},1)\rightarrow S_n(M,\mathcal{N}_{ad},1) $ is injective.

\end{corollary}
\begin{proof}
We already proved the injectivity for $\lambda_{ad}$ when $\N$ is circle free in Corollary \ref{Cor5.13}.

Then we suppose $\N$ contains at least one oriented circle.
When we cut the closure of small open intervals as in Remark \ref{rem7.2}, we choose the same way to cut them for $(M,\mathcal{N})$ and $(M,\mathcal{N}_{ad})$,
and the choices for $\alpha_i$ as in Remark \ref{rem7.2} are compatible between $(M,\mathcal{N}')$ and $(M,(\N_{ad})')$.  The relative spin structure used for $(M,\mathcal{N}')$ is the restriction of the relative spin structure used for $(M,(\N_{ad})')$.
Since $\N_{ad}$ is obtained from $\N$ by adding one extra oriented open interval, we have $\Gamma_n(M,\mathcal{N}_{ad})
= \Gamma_n(M,\mathcal{N})$. Then it is easy to check we have the following commutative diagram:
$$\begin{tikzcd}
\Gamma_n(M,\mathcal{N})\otimes O(SL_n)^{\otimes(\sharp \N -1)}  \arrow[r, "\overline{L}_{(M,\mathcal{N})}"]
\arrow[d, "J"]  
&  S_n(M,\mathcal{N},1)  \arrow[d, "\lambda_{ad}"] \\
\Gamma_n(M,\mathcal{N})\otimes O(SL_n)^{\otimes(\sharp \N )}  \arrow[r, "\overline{L}_{(M,\mathcal{N}_{ad})}"] 
&  S_n(M,\mathcal{N}_{ad},1)\\
\end{tikzcd}$$
where $J$ is the obvious embedding. Since both $\overline{L}_{(M,\mathcal{N}_{ad})}$ and 
$\overline{L}_{(M,\mathcal{N})}$ are isomorphisms, we have $\lambda_{ad}$ is injective.
\end{proof}


%
%
%
%
%
%
%
%

\bibliographystyle{plain}

\bibliography{ref.bib}

\hspace*{\fill} \\

School of Physical and Mathematical Sciences, Nanyang Technological University, 21 Nanyang Link Singapore 637371

$\emph{Email address}$: zhihao003@e.ntu.edu.sg

\end{document}